\documentclass{amsart}
\allowdisplaybreaks[4]
\usepackage{graphicx}
\usepackage{color}

\newcommand{\R}{\mathbf{R}}

\newcommand{\sig}{\sigma}

\newcommand{{\ba}}{\bf a}
\newcommand{\ve}{\varepsilon}
\newcommand{\la}{\lambda}

\newcommand{\ga}{\gamma}
\newcommand{\Ga}{\Gamma}
\newcommand{\pa}{\partial}
\newcommand{\ra}{\rightarrow}

\newcommand{\del}{\delta}

\newcommand{\al}{\alpha}

\newcommand{\be}{\begin{equation}}
\newcommand{\ee}{\end{equation}}

\newtheorem{lem}{Lemma}{\bf}{\it}
{\it}{\rm}
\newtheorem{rem}{Remark}{\it}{\rm}
{\it}{\rm}

\newtheorem{theorem}{Theorem}
\newtheorem{proposition}{Proposition}

\numberwithin{theorem}{section}
\numberwithin{lem}{section}
\numberwithin{equation}{section}
\numberwithin{proposition}{section}
\numberwithin{corollary}{section}
\title[Global Stability]{Global Stability for a Class of Nonlinear PDE with non-local term}

\author{Joseph G. Conlon and Michael Dabkowski }

\address{(Joseph G. Conlon): University of Michigan\\ Department of Mathematics\\ Ann Arbor,
  MI 48109-1109}
\email{conlon@umich.edu}

\address{(Michael Dabkowski): Lawrence Technical University \\ Department of Mathematics\\ Southfield,
  MI 48075-1058}
\email{mdabkowsk@ltu.edu}


\begin{document}

\maketitle

\begin{abstract}
This paper is concerned with establishing global asymptotic stability results for a class of non-linear PDE which have some similarity to the PDE of the 
Lifschitz-Slyozov-Wagner model.  The method of proof does not involve a Lyapounov function. It is shown that stability for the PDE is equivalent to stability for a differential delay equation. Stability for the delay equation is proven by exploiting certain maximal properties. These are established by using the methods of optimal control theory. 
\end{abstract}

    \section{Introduction}
    In this paper we shall be concerned with proving global asymptotic stability for a class of first order non-linear PDE, in which the non-linearity is non-local.  The PDE we study has similarities to the Lifschitz-Slyozov-Wagner (LSW) model \cite{ls,w}, a well-known model of material science, in the sense that it is {\it linear} in its derivatives, with the nonlinearity occurring as a scalar coefficient. In $\S3$ we prove {\it local} asymptotic stability of the the critical point for the PDE by using the theory of Volterra integral equations \cite{grip}. Results from the theory of Volterra integral equations were also used in Niethammer-Vel\'{a}zquez \cite{nv}  to prove local asymptotic stability for the LSW model.
    
Our proof of global asymptotic stability is based on the fact that the stability problem for the PDE is equivalent to proving global stability for a  scalar differential delay equation (DDE)  \cite{hale}. We derive the DDE in $\S4$ and use the theory of Volterra integral equations to prove local asymptotic stability. In $\S8$ we prove global asymptotic stability of the DDE. Being unable to find a suitable Lyapounov function, we resort to a different approach based on exploiting certain maximal properties of the nonlinearity.  This approach to proving stability properties of scalar DDEs seems to have been pioneered by Yorke \cite{yorke}.  The proofs of the maximal properties are contained in $\S7$. We carry this out by using the methods of optimal control theory \cite{fr}. 
    
    We consider the evolution PDE
 \be \label{A1}
 \frac{\pa \xi(y,t)}{\pa t} -h(y) -\frac{\pa\xi(y,t)}{\pa y} +\rho(\xi(\cdot,t))\left[\xi(y,t)-y\frac{\pa \xi(y,t)}{\pa y}\right] \ = \ 0, \quad y>0, \ t>0,
  \ee
  with given non-negative  initial data $\xi(y,0), \ y> 0$.  We assume that the function $h:(0,\infty)\ra\R$ has the properties:
  \be \label{B1}
  h(\cdot) \ {\rm is \ continuous  \ positive,  \ } \lim_{y\ra\infty}h(y)=h_\infty>0, \ {\rm and \ } \int_0^1h(y) \ dy<\infty \ .
  \ee
  Note that the condition (\ref{B1}) allows $h(y)$ to become unbounded as $y\ra 0$. 
  The functional $\rho(\cdot)$ maps continuous nonnegative functions $\zeta:(0,\infty)\ra\R$ to $\R$. 
   For an equilibrium to exist corresponding to $\rho=1/p>0$ we need the equilibrium function $\xi_p(y), \ y> 0,$ to satisfy
  \be \label{C1}
  -h(y) -\frac{d\xi_p(y)}{d y} +\frac{1}{p}\left[\xi_p( y)-y\frac{d \xi_p(y)}{d y}\right] \ = \ 0, \quad  \lim_{y\ra\infty}\xi_p(y)=ph_{\infty} \ .
  \ee 
The solution to (\ref{C1}) is given by 
\be \label{D1}
\xi_p(y) \ = \ (p+y)\int_y^\infty \frac{ph(y')}{(p+y')^2} \ dy' \  .
\ee

We wish to impose conditions on  the functional $\rho(\cdot)$  so that  the equilibrium $\xi_p$ is a global attractor for (\ref{A1}).  To do this we assume there is a positive functional $I(\cdot)$ on continuous nonnegative functions $\zeta:(0,\infty)\ra\R$ with the property that
\be \label{E1}
\frac{1}{p}\frac{d}{dt} I(\xi(\cdot,t)) \ = \ \left[\rho(\xi(\cdot,t))-\frac{1}{p}\right]I(\xi(\cdot,t)) \  \quad {\rm for \ solutions \ } \xi(\cdot,t) \ {\rm of \ (\ref{A1})}  \ . 
\ee
From (\ref{E1}) it follows that if we can show that $\log I(\xi(\cdot,t))$ remains bounded as $t\ra\infty$ then $\rho(\xi(\cdot,t))$ converges as $t\ra\infty$ to $1/p$ in the averaged sense
\be \label{F1}
\lim_{T\ra\infty}\frac{1}{T}\int_0^T \rho(\xi(\cdot,t)) \ dt \ = \ \frac{1}{p} \ .
\ee
 
Let $[\cdot,\cdot]$ denote the Euclidean inner product on $L^2[(0,\infty)]$ and $dI(\zeta(\cdot)):(0,\infty)\ra\R$ the gradient of the functional $I(\cdot)$ at $\zeta(\cdot)$. Then from(\ref{A1}), (\ref{E1}) we have that
\be \label{G1}
p\left[\rho(\zeta(\cdot))-\frac{1}{p}\right]I(\zeta(\cdot)) \ = \ [dI(\zeta(\cdot)), h+D\zeta]+\rho(\zeta(\cdot))[dI(\zeta(\cdot)), yD\zeta-\zeta] \ .
\ee
We conclude from (\ref{G1}) that
\be \label{H1}
\rho(\zeta(\cdot)) \ = \ \frac{I(\zeta(\cdot))+[dI(\zeta(\cdot)), h+D\zeta]}{pI(\zeta(\cdot))+[dI(\zeta(\cdot)), \zeta-yD\zeta]}  \ .
\ee

Our goal in this paper is to prove global existence and asymptotic stability theorems for solutions to (\ref{A1}) in the case when the functional $\rho(\cdot)$ is given by (\ref{H1}), and $I(\cdot)$ belongs to a fairly large class of functionals.  In order to do this we define for $m=1,2,..,$  norms on $C^m$ functions $\zeta:(0,\infty)\ra\R$  by
\be \label{I1}
\|\zeta(\cdot)\|_{m,\infty} \ = \ \sup_{0<y<\infty}\sum_{k=0}^m y^k\left|\frac{d^k\zeta(y)}{dy^k}\right|  \ .
\ee
We assume that  $I(\cdot)$ has the following properties: \\
{\it (a) There exists $\ve_0>0$ such that the functional $\zeta\ra I(\zeta(\cdot))$ from nonnegative continuous functions $\zeta:(0,\infty)\ra\R^+$ to $\R^+$ is independent of $\zeta(y)$ when $0<y<\ve_0$.  \\
(b) For any continuous nonnegative function $\zeta:(0,\infty)\ra\R^+$ the gradient of $I(\cdot)$ at $\zeta(\cdot)$ is an integrable function. Furthermore, the mapping $dI$ from functions $\zeta:(0,\infty)\ra\R^+$ to $L^1(\R^+)$ is uniformly Lipschitz continuous in the $L^\infty(\R^+)$ norm. } That is
\be \label{J1}
\|dI(\zeta_1(\cdot))-dI(\zeta_2(\cdot))\|_{L^1(\R^+)} \ \le \   C\|\zeta_1(\cdot)-\zeta_2(\cdot)\|_{L^\infty(\R^+)}
\ee
for some constant $C$.  \\
{\it (c) For any $M\ge 0$ there exists $c_M>0$ such that $I(\zeta(\cdot))\ge c_M$ if $\zeta(\cdot)$ is non-negative and $\|\zeta(\cdot)\|_\infty\le M$.  \\
(d) The gradient $dI(\zeta(\cdot)):\R^+\ra\R$ is a non-positive function for all continuous $\zeta:(0,\infty)\ra\R^+$ and $\inf_{\zeta(\cdot)}\{I(\zeta(\cdot))+[dI(\zeta(\cdot)),h]\}> 0$. }

An important example of a functional $I(\cdot)$ which satisfies (a), (b), (c), (d) above is given by
\be \label{K1}
I(\zeta(\cdot)) \ = \ \int_{\ve_0}^\infty \frac{a(y)}{[b(y)+\zeta(y)]^q} \ dy \ ,
\ee
where $q>0$ and  $a,b:[\ve_0,\infty)\ra\R^+$ are non-negative functions with the property that $y\ra a(y)/b(y)^r, \ r=q,q+1,q+2$ are integrable on 
$(\ve_0,\infty)$ and $b(y)\ge qh(y), \ y\ge \ve_0$.  

It is evident that if $I(\cdot)$ satisfies (a), (b)  and $\zeta:(0,\infty)\ra\R^+$ satisfies $\|\zeta(\cdot)\|_{1,\infty}<\infty$ then the numerator and denominator of the RHS of (\ref{H1}) are finite, whence $\rho(\zeta(\cdot))$ is finite provided
\be \label{L1}
pI(\zeta(\cdot))+[dI(\zeta(\cdot)), \zeta-yD\zeta] \ > \ 0 \ .
\ee
 In $\S2$ we prove the following existence and uniqueness theorem:
\begin{theorem}
Assume the function $h(\cdot)$ is $C^2$ decreasing, satisfies (\ref{B1}) and $\sup_{y>0}[y|h'(y)|+y^2|h''(y)|]<\infty$. Assume also that the functional $I(\cdot)$ satisfies $(a),(b),(c),(d)$  and in addition that the initial data $\xi(\cdot,0)$ for (\ref{A1}) is $C^2$ non-negative decreasing, $\|\xi(\cdot,0)\|_{2,\infty}<\infty$ and (\ref{L1}) holds for $\zeta(\cdot)=\xi(\cdot,0)$. Then there exists  a unique solution $\xi(\cdot,t), t\ge 0$, globally in time to the initial value problem for (\ref{A1}). Furthermore, $\sup_{t\ge 0}\|\xi(\cdot,t)\|_{2,\infty}<\infty$ and the infimum of the LHS of (\ref{L1}) over all $\zeta(\cdot)=\xi(\cdot,t), \ t\ge 0$, is strictly positive.   \end{theorem}
We require some further properties of the function $h(\cdot)$ in order to prove asymptotic stability. These are given by
\be \label{M1}
yh''(y)+h'(y) \ \ge  \ 0 \ {\rm and} \quad y\ra y^2h''(y)  \ {\rm decreasing \ for \ } y>0 \ . 
\ee
\begin{theorem}
Assume that $h(\cdot),I(\cdot)$ and $\xi(\cdot,0)$ satisfy the conditions of Theorem 1.1, and in addition that (\ref{M1}) holds. Then for any $q>p$ there exists a constant $C_q$ such that
\be \label{N1}
\|\xi(\cdot,t)-\xi_p(\cdot)\|_{2,\infty} \ \le \ C_qe^{-t/q} \quad  {\rm for \ } t\ge 0 \ .
\ee
\end{theorem}
In the proof of Theorem 1.2  we first show that for any $\ve>0$ there is a time $T_\ve>0$ such that $\|\xi(\cdot,T_\ve)-\xi_p(\cdot)\|_{1,\infty} \ \le \ve$. This result is a consequence of our stability theorem for the corresponding DDE. The exponential decay in (\ref{N1}) then follows from the local asymptotic stability theorem proved in $\S3$.  

\vspace{.1in}

\section{Existence and Uniqueness Theorems}
In this section we prove Theorem 1.1. We begin by  first studying the linear PDE
 \be \label{A2}
 \frac{\pa \xi(y,t)}{\pa t} -h(y) -\frac{\pa\xi(y,t)}{\pa y} +\rho(t)\left[\xi(y,t)-y\frac{\pa \xi(y,t)}{\pa y}\right] \ = \ 0, \quad y> 0, \ t>0,
  \ee
where $\rho:[0,\infty)\ra\R$ is assumed to be a known continuous function.  The evolution PDE (\ref{A2}) is uniquely solvable by the method of characteristics for given  initial data $\xi(y,0), \ y> 0$. The characteristic $y(\cdot)$ defined as the solution to the terminal value problem,
\be \label{B2}
\frac{dy(s)}{ds} \ = \ -1-\rho(s)y(s) \ , \quad 0\le s<t, \ y(t)=y \ ,
\ee
has the property that for $y> 0$ then $y(s)> 0, \ 0\le s< t$. The solution to (\ref{B2})  is evidently given by the formula
\be \label{C2}
y(s) \ = \ \exp\left[\int_s^t\rho(s') \ ds'\right]y+\int_s^t ds' \  \exp\left[\int^{s'}_s\rho(s'') \ ds''\right] \ .
\ee
The solution to (\ref{A2}) with the given initial data is then expressed in terms of the characteristic (\ref{C2})  by
\be \label{D2}
\xi(y,t) \ = \ \exp\left[-\int_0^t \rho(s) \ ds\right]\xi(y(0),0)+\int_0^t ds \  h(y(s))\exp\left[-\int_s^t \rho(s') \ ds'\right]  \ . 
\ee
It follows from (\ref{B1}), (\ref{D2}) that if initial data $\xi(\cdot,0)$ is non-negative then the function $\xi(\cdot,t)$ is non-negative for all $t>0$.

We prove a local existence and uniqueness theorem for the initial value problem for (\ref{A1}). 
\begin{lem}
Assume $I(\cdot)$ satisfies the conditions $(a),(b)$ of $\S1$, and the function $h(\cdot)$ is $C^2$  satisfying $\sup_{y>0}[y|h'(y)|+y^2|h''(y)|]<\infty$ and (\ref{B1}).  Let $\xi(\cdot,0)$ be the initial data for (\ref{A1}) and assume it is $C^2$ non-negative with $\|\xi(\cdot,0)\|_{2,\infty}<\infty$ and such that (\ref{L1}) holds with $\zeta(\cdot)=\xi(\cdot,0)$.   Then there exists $T>0$ and a unique solution $\xi(\cdot,t), \ 0\le t\le T$, to the initial value problem for (\ref{A1}) such that $\|\xi(\cdot,t)\|_{2,\infty}<\infty, \ 0\le t\le T$, and (\ref{L1}) holds for $\zeta(\cdot)=\xi(\cdot,t), \ 0\le t\le T$. 
\end{lem}
\begin{proof}
Let $\rho_0=\rho(\xi(\cdot,0))$ and for $\ve,T>0$ let $\mathcal{E}_{\ve,T}$ be the metric space of continuous functions $\rho:[0,T]\ra\R$ such that  $\rho(0)=\rho_0$ and $\|\rho(\cdot)-\rho_0\|_\infty<\ve$. If $\rho(\cdot)\in\mathcal{E}_{\ve,T}$ we define the function $K\rho:[0,T]\ra\R$ by
\be \label{E2}
K\rho(t) \ = \ \rho(\xi(\cdot,t)) \ , \ \ 0\le t\le T, \quad  {\rm where \ } \xi(\cdot,t) \ {\rm is \  given \  by \  (\ref{D2}).}
\ee
Evidently fixed points of $K$ correspond to solutions $\xi(\cdot,\cdot)$ of (\ref{A1}).   

We first show that $K$ maps $\mathcal{E}_{\ve,T}$ to itself provided $\ve,T>0$ are sufficiently small. To do this we use the representation
\be \label{F2}
I(\xi(t))-I(\xi(0)) \ = \ \int_0^1 [dI(\la\xi(t)+(1-\la)\xi(0)),\xi(t)-\xi(0)] \ d\la \ .
\ee
It follows from property (a) of $I(\cdot)$ and (\ref{J1}) that for $0\le \la\le 1$, 
\be \label{G2}
\|dI(\la\xi(t)+(1-\la)\xi(0))-dI(\xi(0))\|_{L^1(\R^+)} \ \le \ C\sup_{y\ge \ve_0}|\xi(y,t)-\xi(y,0)| \  .
\ee
From (\ref{D2}) and using the inequality $\sup_{y>0}y|h'(y)|<\infty$,  we see that
\be \label{H2}
\sup_{y\ge \ve_0}|\xi(y,t)-\xi(y,0)| \ \le \  C_1t\left[\ \|\xi(\cdot,0)\|_{1,\infty}+1\right] \ , \quad 0\le t\le T,
\ee
for a constant $C_1$ depending only on $\rho_0,\ve,$ provided $T\le 1$. It follows now from (\ref{F2})-(\ref{H2}) that for any $\del>0$ we may choose $T>0$ sufficiently small so that  $|I(\xi(t))-I(\xi(0))|<\del$ if $\rho(\cdot)\in\mathcal{E}_{\ve,T}$.  Next we write
\begin{multline} \label{I2}
[dI(\xi(t)),h+D\xi(t)]-[dI(\xi(0)),h+D\xi(0)] \\
= \   [dI(\xi(t)),D\xi(t)-D\xi(0)]+[dI(\xi(t))-dI(\xi(0)),h+D\xi(0)] \  \  .
\end{multline}
Using the fact that $\sup_{y>0}|y^2h''(y)|<\infty$, we find similarly to (\ref{H2}) that 
\be \label{J2}
\sup_{y\ge \ve_0}|yD\xi(y,t)-yD\xi(y,0)| \ \le \  C_2t\left[ \  \|\xi(\cdot,0)\|_{2,\infty}+1\right] \ , \quad 0\le t\le T,
\ee
for a constant $C_2$ depending only on $\rho_0,\ve,$ provided $T\le 1$.  We conclude that for any $\del>0$ we may choose $T>0$ sufficiently small so that for $0<t\le T$ the numerator of (\ref{H1}) evaluated at $\zeta=\xi(t)$ differs from the numerator evaluated at $\zeta=\xi(0)$ by at most $\del$. Since the same holds for the denominator, we conclude that for $\rho(\cdot)\in \mathcal{E}_{\ve,T}$ the function $K\rho(\cdot)$ has $\|K\rho(\cdot)-\rho_0\|_\infty<\ve$ if $T>0$ is sufficiently small. We have therefore shown that $K$ maps $\mathcal{E}_{\ve,T}$  to itself if $T>0$ is sufficiently small. 

We can similarly show that for $T>0$ small the mapping $K$ is a contraction on $\mathcal{E}_{\ve,T}$. Let $\xi_j(\cdot,t), \ j=1,2, \ 0\le t\le T,$ denote the functions (\ref{D2}) with $\rho(\cdot)=\rho_j(\cdot) \ j=1,2$. We have that
\be \label{K2}
\sup_{y\ge \ve_0}|\xi_1(y,t)-\xi_2(y,t)| \ \le \  C_3t\|\rho_1(\cdot)-\rho_2(\cdot)\|_\infty\left[ \  \|\xi(\cdot,0)\|_{1,\infty}+1\right] \ , \quad 0\le t\le T,
\ee
for a constant $C_3$ depending only on $\rho_0,\ve,$ provided $T\le 1$. Corresponding to (\ref{J2}) we also have that
 \be \label{L2}
\sup_{y\ge \ve_0}|yD\xi_1(y,t)-yD\xi_2(y,t)| \ \le \  C_4t\|\rho_1(\cdot)-\rho_2(\cdot)\|_\infty\left[ \  \|\xi(\cdot,0)\|_{2,\infty}+1\right] \ , \quad 0\le t\le T,
\ee
for a constant $C_4$ depending only on $\rho_0,\ve,$ provided $T\le 1$. The inequalities (\ref{K2}), (\ref{L2}) are sufficient to show that $K$ is a contraction provided $T>0$ is sufficiently small. Now substituting the fixed point $\rho(\cdot)$ for $K$ into (\ref{D2}) it is easy to see that $\|\xi(\cdot,t)\|_{2,\infty}<\infty$ for $0\le t\le T$. 
\end{proof}

\begin{rem}
Observe that in (\ref{D2}) the function $h(\cdot)$ appears in an integral, whence the second derivative of $\xi(\cdot,t)$ can be bounded in terms of the first derivative of $h(\cdot)$.  Therefore a condition weaker than $\sup_{y>0}y^2|h''(y)|<\infty$ is sufficient to establish a local existence and uniqueness theorem. It does not seem possible to make a contraction mapping argument using the function $\xi(\cdot,t)$ and the norm $\|\cdot\|_{1,\infty}$. The reason for this is that $\xi(\cdot,0)$ can have oscillations of order $1$ close to the origin since  $D\xi(y,0)$ can be unbounded for $y$ close to $0$. 
\end{rem}

The main observation in the proof of global existence of solutions to the IVP for (\ref{A1}) is the following:
\begin{lem}
Assume $I(\cdot)$ satisfies the conditions $(a),(b),(c),(d)$ of $\S1$ and that $h(\cdot)$, in addition to satisfying the conditions of Lemma 2.1,  is also a decreasing function. Similarly assume $\xi(\cdot,0)$ satisfies the conditions of Lemma 2.1 and is  decreasing. Suppose a solution $\xi(\cdot,t), \ 0\le t<T$, to the IVP for (\ref{A1}) exists in the interval $[0,T)$, has $\|\xi(\cdot,t)\|_{1,\infty}<\infty$ and (\ref{L1}) holds with $\zeta(\cdot)=\xi(\cdot,t)$ for all $t\in[0,T)$.  Then  $I(\xi(\cdot,t))\ge c_0, \ 0\le t<T$, for some positive constant $c_0$ depending only on $\|\xi(\cdot,0)\|_{2,\infty}$ and the value of the LHS of (\ref{L1}) when $\zeta(\cdot)=\xi(\cdot,0)$.
 \end{lem}
\begin{proof}
From Lemma 2.1 we may assume $T>0$, and by differentiating (\ref{D2}) we see that $\xi(\cdot,t)$ is a positive decreasing function for all $0<t<T$.   Hence property (d) of $I(\cdot)$ implies that $\rho(\xi(\cdot,t))\ge 0, \ 0\le t<T$. It then follows from (\ref{C2}), (\ref{D2}) that
\be \label{M2}
\xi(y,t) \ \le \ \exp\left[-\int_0^t \rho(s') \ ds'\right]\xi(\ve_0,0)+h(\ve_0)\int_0^t ds \ \exp\left[-\int_s^t \rho(s') \ ds'\right] \ ,
\ee
for $y\ge \ve_0$ and also that
\be \label{N2}
\xi(y,t) \ \ge \ h_\infty \int_0^t ds \ \exp\left[-\int_s^t \rho(s') \ ds'\right] \ , \quad y\ge \ve_0 \ .
\ee
From  Lemma 2.1 we have that for any $\ve>0$ there exists $\del>0$, depending only on $\|\xi(\cdot,0)\|_{2,\infty}$ and the value of the LHS of (\ref{L1}) when $\zeta(\cdot)=\xi(\cdot,0)$,  such that $|\rho(\xi(\cdot,t))-\rho(\xi(\cdot,0))|<\ve$ if $0\le t<\del$.  We conclude from (\ref{M2}), (\ref{N2}) that there exists a constant $C>0$, depending only on $\|\xi(\cdot,0)\|_{2,\infty}$ and the value of the LHS of (\ref{L1}) when $\zeta(\cdot)=\xi(\cdot,0)$, such that
\be \label{O2}
\sup_{y\ge \ve_0}\xi(y,t) \ = \ \xi(\ve_0,t) \ \le  \ C\xi(\infty,t) \ = \ C\inf_{y\ge\ve_0}\xi(y,t) \ , \quad 0\le t<T \ .
\ee
We choose any $M>Cph(\ve_0)$ and observe from (\ref{O2}) that if $\|\xi(\cdot,t)\|_\infty\ge M$ then $\inf_{y\ge\ve_0}\xi(y,t)> ph(\ve_0)$. Using the fact that both $\xi(\cdot,t)$ and $h(\cdot)$ are decreasing functions we see from property $(d)$ of $I(\cdot)$ that $\rho(\xi(\cdot,t))>1/p$ if $\inf_{y\ge\ve_0}\xi(y,t)> ph(\ve_0)$. Hence (\ref{E1}) implies that if $\|\xi(\cdot,\tau)\|_\infty\ge M$  then the function $t\ra I(\xi(\cdot,t))$ is increasing at $t=\tau$.  We conclude that $I(\xi(\cdot,t))\ge \min\{c_M, \ I(\xi(\cdot,0))\}$ for all $0<t<T$, where $c_M$ is the constant in property $(c)$ of $I(\cdot)$. 
\end{proof}
\begin{proof}[Proof of Theorem 1.1]
From property $(d)$ of $I(\cdot)$ it follows that $I(\zeta(\cdot))\le I(0(\cdot))$ for all non-negative continuous functions $\zeta(\cdot)$.  Suppose now a solution $\xi(\cdot,t)$ exists in the interval $0<t<T$ satisfying the conditions of Lemma 2.2. Then we have from (\ref{E1}) that
\be \label{P2}
\left|\int_s^t\left\{\rho(\xi(\cdot,s'))-\frac{1}{p}\right\} \ ds'\right| \ \le \  \frac{1}{p}\log\left(\frac{I(0(\cdot))}{c_0}\right) \ , \quad 0<s<t<T \ ,
\ee
where $c_0$ is the lower bound for $I(\cdot)$ of Lemma 2.2. It follows from (\ref{D2}), (\ref{P2}) that there is a constant $C$, independent of $T$, such that $\|\xi(\cdot,t)\|_{2,\infty}\le C$ for $0\le t<T$.  Suppose now that $T<\infty$ and there is an increasing sequence of times $T_n, \ n=1,2,..,$ such that $\lim_{n\ra\infty}T_n=T$ and $\liminf_{n\ra\infty} F(T_n)>0$, where  $F(t)=pI(\xi(\cdot,t))+[dI(\xi(\cdot, t)), \xi(\cdot,t)-yD\xi(\cdot,t)]$.  Since $\|\xi(\cdot,T_n)\|_{2,\infty}\le C, \ n=1,2,..,$  Lemma 2.1 implies that we may extend the solution to (\ref{A1}) beyond time $T$. Alternatively we have $\lim_{t\ra T}F(t)=0$,  and since $\|\xi(\cdot,t)\|_{2,\infty}\le C, \ 0\le t<T,$ there is a constant $C_1$ independent of $T$ such that $|F(t_2)-F(t_1)|\le C_1(t_2-t_1), \ $ for $0<t_1<t_2<T, \ t_2-t_1\le 1$.  It follows then from property $(d)$ of $I(\cdot)$ that $\liminf_{t\ra T}(T-t)\rho(\xi(\cdot,t))>0$, but this contradicts (\ref{P2}). We conclude that  a global solution of (\ref{A1}) with the property $\sup_{t>0}\|\xi(\cdot,t)\|_{2,\infty}<\infty$ exists. The strictly positive lower bound on the infimum of the LHS of (\ref{L1}) over all $\xi(\cdot,t), \ t\ge 0,$ follows by a similar argument.  
\end{proof}

\vspace{.2in}

\section{Local Asymptotic Stability}
We first linearize (\ref{A1}) with $\rho(\cdot)$ given by (\ref{H1}) about the equilibrium $\xi_p(\cdot)$ and study its stability.   To do this we denote by $A,B$ the operators
\be \label{A3}
A\zeta(y) \ = \ \zeta(y)-y\frac{d\zeta(y)}{dy} \ , \quad B\zeta(y) \ = \ \frac{1}{p}\zeta(y)-\left[1+\frac{y}{p}\right]\frac{d\zeta(y)}{dy} \ . 
\ee
Observe now that the functional $\rho(\cdot)$ of (\ref{H1}) satisfies the identity
\be \label{B3}
\rho(\zeta(\cdot))-\frac{1}{p} \ = \ -\frac{\left[dI(\zeta(\cdot)), \ B\{\zeta(\cdot)-\xi_p(\cdot)\}\right]}{pI(\zeta(\cdot))+\left[dI(\zeta(\cdot)),  \ A\zeta(\cdot)\right]} \ .
\ee
Hence we may rewrite (\ref{A1}) as
\be \label{C3}
\frac{\pa \xi(\cdot,t)}{\pa t}+B\{\xi(\cdot,t)-\xi_p(\cdot)\}-\frac{\left[dI(\xi(\cdot,t)), \ B\{\xi(\cdot,t)-\xi_p(\cdot)\}\right]}{pI(\xi(\cdot,t))+\left[dI(\xi(\cdot,t),  \ A\xi(\cdot,t)\right]}A\xi(\cdot,t) \ = \ 0 \ .
\ee
Setting $\tilde{\xi}(\cdot,t)=\xi(\cdot,t)-\xi_p(\cdot)$, it follows from (\ref{C3}) that the linearization of (\ref{A1}) about $\xi_p(\cdot)$ is given by 
\be \label{D3}
\frac{d\tilde{\xi}(t)}{dt} +B\tilde{\xi}(t) -[dI(\xi_p),B\tilde{\xi}(t)]\frac{A\xi_p}{pI(\xi_p)+[dI(\xi_p), \ A\xi_p]} \ = \ 0 \ .
\ee
The solution to (\ref{D3}) satisfies
\be \label{E3}
\tilde{\xi}(t) \ = \ e^{-Bt}\tilde{\xi}(0)+\int_0^t ds \ [dI(\xi_p),B\tilde{\xi}(s)]\frac{e^{-B(t-s)}A\xi_p}{pI(\xi_p)+[dI(\xi_p), \ A\xi_p]}  \ .
\ee
Hence if we set $u(t)=[dI(\xi_p),B\tilde{\xi}(t)]$ then (\ref{E3}) yields an integral equation for $u$,
\be \label{F3}
u(t)+\int_0^t  K(t-s)u(s) \ ds \ = \ g(t)   \ , \quad t>0,
\ee
where the functions $K,g$ are given by
\be \label{G3}
K(t)  = -\frac{[dI(\xi_p),e^{-Bt}BA\xi_p]}{pI(\xi_p)+[dI(\xi_p), \ A\xi_p]}  \ , \quad g(t)=[dI(\xi_p),e^{-Bt}B\tilde{\xi}(0)] \quad t\ge 0.
\ee

Equation (\ref{F3}) is a Volterra integral equation and it may be studied using Laplace transform methods.  
Extending the functions $u,K,g$ on $\R^+$ to $\R$ by setting them to be zero on $\R^-$, then (\ref{F3}) is simply the convolution equation $u+K*u=g$.  If $K,g\in L^1_{\rm loc}(\R^+),$ there is by Theorem 3.5 of Chapter II of \cite{grip} a unique solution $u$ in $L^1_{\rm loc}(\R^+)$ to (\ref{F3}).  It is given by the formula $u=g-r*g$, where the resolvent $r$ is also in $L^1_{\rm loc}(\R^+)$. In order to prove asymptotic stability for the linearized equation (\ref{D3}) we shall need to show that the solution $u(\cdot)$ of (\ref{F3}) satisfies $\lim_{t\ra\infty}u(t)=0$. Suppose now that the function $t\ra e^{-\sig t}K(t), \ t>0,$ is in $L^1(\R^+)$ for $\sig>-1/p$. It then follows from Corollary 4.2 of Chapter II of \cite{grip} that the function  $t\ra e^{-c t}r(t), \ t>0,$ is also in $L^1(\R^+)$ for $c>-1/p$ provided
\be \label{H3}
1+\hat{K}(z)\ne 0 \quad {\rm for \ } \Re z> -1/p \ ,
\ee
where $\hat{K}$ is the Laplace transform of $K$,
\be \label{I3}
\hat{K}(z) \ = \ \int_0^\infty K(t)e^{-z t} \ dt \ .
\ee
If the function $t\ra e^{-c t}r(t)$ is  in $L^1(\R^+)$ and the function $t\ra e^{-c t}g(t)$ is in $L^\infty(\R^+)$ for $c>-1/p$,  then it is easy to conclude that $\lim_{t\ra\infty} e^{-ct}u(t)=0$ for $c>-1/p$. 

Suppose now that the function $K(\cdot)$ has the property that $t\ra e^{ t/p}K(t)$ is positive and decreasing. It is easy to see then that $\Im \hat{K}(z)\ne 0$ if $\Im z\ne 0$ and $\Re z>-1/p$.  Since $\hat{K}(z)>0$ for real $z$ we conclude that (\ref{H3}) holds in this case. In the following we obtain conditions  so that the function $t\ra e^{t/p}K(t)$  with $K(\cdot)$ defined by (\ref{G3}) is positive decreasing. 
\begin{lem}
 Assume that  the function $h:(0,\infty)\ra\R$ of (\ref{B1})  is $C^2$, positive decreasing  convex, and in addition satisfies the inequalities  
\be \label{J3}
\sup_{y>0} \{y|h'(y)|\}<\infty, \quad h'(y)+\left[1+\frac{y}{p}\right] yh''(y) \ \ge  \ 0, \quad y> 0.
\ee
Assume  further that  $I(\cdot)$ satisfies property (a) of the introduction, and in addition that the gradient of $I(\cdot)$ at $\xi_p$ is a negative integrable function on $(0,\infty)$ with $pI(\xi_p)+[dI(\xi_p), \ A\xi_p]>0$. \\
Then $K(\cdot)$ defined by (\ref{G3}) has the property that the function $t\ra e^{t/p}K(t), \ t>0,$ is positive and decreasing.  
\end{lem}
\begin{proof}
We first observe  that $A\xi_p(\cdot)$ is a positive bounded function in the interval $(0,\infty)$.  In fact from (\ref{D1}) we have that
\be \label{K3}
\frac{1}{p}A\xi_p(y) \ = \ \int_y^\infty \frac{ph(y')}{(p+y')^2} \ dy'+\frac{yh(y)}{p+y} \ , \quad y> 0 \ ,
\ee
whence the positivity of  $A\xi_p(\cdot)$ follows.  To obtain the boundedness we use the identity
\be \label{P3}
\int_y^1zh'(z) \ dz \ = \ h(1)-yh(y)-\int_y^1 h(z) \ dz \  .
\ee
It follows from (\ref{B1}), (\ref{P3}) that the function $y\ra yh(y)$ converges as $y\ra 0$. Since (\ref{B1}) implies the integrability of $h(\cdot)$ on the interval $[0,1]$, we conclude that  $\lim_{y\ra 0} yh(y)=0$.  Furthermore,  $\lim_{y\ra 0} A\xi_p(y)$ exists and is finite. 
 Since $dI(\xi_p)$ is integrable on $(0,\infty)$ we conclude that the inner product $[dI(\xi_p), \ A\xi_p]$ is well defined.
Using also the fact that $h(\cdot)$ is decreasing, we further see  that $BA\xi_p(\cdot)$ is a non-negative function.  In fact we have that
\begin{multline} \label{L3}
BA\xi_p(y) \ = \ AB\xi_p(y) +\frac{d\xi_p(y)}{dy} \\
 = \  Ah(y)+\frac{d\xi_p(y)}{dy} \ = \ \frac{1}{p}A\xi_p(y)-yh'(y) \ \ge \ 0 \ .
\end{multline}
Since $h(\cdot)$ is decreasing, $dI(\xi_p)$ is a negative function, and  $pI(\xi_p)+[dI(\xi_p), \ A\xi_p]>0$, we conclude that $K(\cdot)$ is a positive function and $K(0)<\infty$.

To show that the function $t\ra e^{t/p}K(t)$  is decreasing we observe that
\begin{multline} \label{M3}
(B-1/p)BA\xi_p(y) \ = \  B\left[\frac{1}{p}A\xi_p(y)-yh'(y)\right]-\frac{1}{p}BA\xi_p(y) \\
= \ -B[yh'(y)] \ = \ h'(y)+\left[1+\frac{y}{p}\right] yh''(y)  \ .
\end{multline}
Hence the result follows from (\ref{G3}), (\ref{J3}) upon using the fact that $dI(\xi_p)(\cdot)$ is a negative function.
\end{proof}
\begin{rem}
Note that $h(\cdot)$ is decreasing and satisfies the second inequality in (\ref{J3}) if and only if there is a $C^1$ non-negative  decreasing function $k:(0,\infty)\ra\R$ such that $h'(y)=-(1+p/y)k(y), \ y>0$. Evidently $ h(\cdot)$ satisfies the first inequality of (\ref{J3}) if and only if $\lim_{y\ra 0}k(y)<\infty$.  Assuming $k(\cdot)$ is integrable on all intervals $(y,\infty), \ y>0$, we then have from (\ref{B1}) that
\be \label{N3}
h(y) \ = \ h_\infty+\int_y^\infty \left(1+\frac{p}{y'}\right)k(y') \ dy' \ , \quad y>0 \ .
\ee
Note from (\ref{N3}) that $\lim_{y\ra 0} h(y)=\infty$ unless $h(\cdot)$ is a constant function. 
\end{rem}
Next we prove an asymptotic stability result for the linearized equation (\ref{D3}). Note from (\ref{D1}) that if $h(\cdot)$ is positive decreasing and satisfies the first inequality of (\ref{J3}), then using (\ref{P3}) we see that $\|\xi_p(\cdot)\|_{2,\infty}<\infty$. 
\begin{proposition}
Assume that $h(\cdot)$ and $I(\cdot)$ satisfy the conditions of Lemma 3.1. Then the linear evolution equation (\ref{D3}) is asymptotically stable in the following sense:  Let the initial data $\tilde{\xi}_0:(0,\infty)\ra\R$  satisfy $\|\tilde{\xi}_0(\cdot)\|_{m,\infty}<\infty$ for either $m=1$ or $m=2$.  If $m=1$ there is for any $q>p$  a constant $C_q$ depending  on $q$ such that 
\be \label{Q*3}
\|\tilde{\xi}(\cdot,t)\|_{m,\infty}\le C_q e^{-t/q}\|\tilde{\xi}_0(\cdot)\|_{m,\infty} \quad  {\rm when \ }  t\ge 0 \ . 
\ee
If in addition to the assumptions of Lemma 3.1 the function $h(\cdot)$ also satisfies $\sup_{0<y<\infty} y^2h''(y)<\infty$, then (\ref{Q*3}) holds for $m=2$. 
\end{proposition}
\begin{proof}
Observe from (\ref{C2}), (\ref{D2})  that  the action of $e^{-Bt}$ on a function $\zeta:(0,\infty)\ra\R$ is given by
\be \label{Q3}
e^{-Bt}\zeta(y) \ = \  e^{-t/p}\zeta\left( e^{t/p}y+p\left[e^{t/p}-1\right] \         \right) \ .
\ee
We have already shown in Lemma 3.1 that the function $t\ra K(t)$ of (\ref{G3}) is positive and $t\ra e^{t/p}K(t)$ is decreasing. To see that  the function $t\ra e^{t/p}g(t)$ is bounded we first note from (\ref{A3}) that $\sup_{y>\ve_0} |B\zeta(y)|\le (1/p+1/\ve_0)\|\zeta(\cdot)\|_{1,\infty}$. Since $dI(\xi_p)$ is integrable and supported in $[\ve_0,\infty)$ we see from (\ref{Q3})  that $t\ra e^{t/p}g(t)$ is bounded by a constant times $\|\tilde{\xi}_0(\cdot)\|_{1,\infty}$. We conclude that if $u(\cdot)$ is the solution to the Volterra equation (\ref{F3}) then  $t\ra e^{t/q}u(t)$ is bounded by a constant times $\|\tilde{\xi}_0(\cdot)\|_{1,\infty}$ for any $q>p$. It follows from (\ref{E3}), (\ref{Q3}) and the boundedness of the function $A\xi_p$ on the interval $(0,\infty)$  that for any $q>p$ there is a constant $C_q$ such that 
$\sup_{0<y<\infty} |\tilde{\xi}(y,t)| \ \le C_q e^{-t/q}\|\tilde{\xi}_0(\cdot)\|_{1,\infty}$ when $t\ge 0$.  To bound the derivative we apply $D=\pa/\pa y$ to (\ref{E3}) and use (\ref{Q3}) to obtain the equation
\be \label{R3}
D\tilde{\xi}(t) \ = \ e^{-(B-1/p)t}D\tilde{\xi}(0)+\int_0^t ds \ [dI(\xi_p),B\tilde{\xi}(s)]\frac{e^{-(B-1/p)(t-s)}DA\xi_p}{pI(\xi_p)+[dI(\xi_p), \ A\xi_p]}  \ .
\ee
On differentiating (\ref{K3}) we see that $\sup_{0<y<\infty} |yDA\xi_p(y)|<\infty$.  We conclude from (\ref{Q3}), (\ref{R3}) that for any $q>p$ there is a constant $C_q$ such that  $\sup_{0<y<\infty} |y\pa\tilde{\xi}(y,t)/\pa y| \ \le C_q e^{-t/q}\|\tilde{\xi}_0(\cdot)\|_{1,\infty}$ when $t\ge 0$.    On differentiating (\ref{R3}) we have
\be \label{S3}
D^2\tilde{\xi}(t) \ = \ e^{-(B-2/p)t}D^2\tilde{\xi}(0)+\int_0^t ds \ [dI(\xi_p),B\tilde{\xi}(s)]\frac{e^{-(B-2/p)(t-s)}D^2A\xi_p}{pI(\xi_p)+[dI(\xi_p), \ A\xi_p]}  \ .
\ee
Now $\sup_{0<y<\infty} |y^2D^2A\xi_p(y)|<\infty$ provided $\sup_{0<y<\infty} y^2h''(y)<\infty$, in addition to the assumptions of Lemma 3.1. It follows  then from (\ref{Q3}), (\ref{S3}) that for any $q>p$ there is a constant $C_q$ such that  $\sup_{0<y<\infty} |y^2\pa^2\tilde{\xi}(y,t)/\pa y^2| \ \le C_q e^{-t/q}\|\tilde{\xi}_0(\cdot)\|_{2,\infty}$ when $t\ge 0$. The result follows.
\end{proof}
We generalize the result of Proposition 3.1 to apply to the non-linear PDE (\ref{C3}) by considering (\ref{C3}) as a perturbation of (\ref{D3}) of the form
\be \label{T3}
\frac{d\tilde{\xi}(t)}{dt} +[B+\del_1(\tilde{\xi}(t))A]\tilde{\xi}(t) -\left\{[dI(\xi_p),B\tilde{\xi}(t)]+\del_2(\tilde{\xi}(t))\right\}\frac{A\xi_p}{pI(\xi_p)+[dI(\xi_p), \ A\xi_p]} \ = \ 0 \ ,
\ee
where $\del_1(\cdot), \ \del_2(\cdot)$ are real valued functionals of $C^1$ functions $\tilde{\zeta}:(0,\infty)\ra\R$. If we take
\be \label{U3}
\del_1(\tilde{\zeta}(\cdot)) \ = \ -\frac{\left[dI(\xi_p+\tilde{\zeta}), \ B\tilde{\zeta}\right]}{pI(\xi_p+\tilde{\zeta})+\left[dI(\xi_p
+\tilde{\zeta}),  \ A\{\xi_p+\tilde{\zeta}\}\right]} \  ,
\ee 
and
 \begin{multline} \label{V3}
\del_2(\tilde{\zeta}(\cdot)) \ = \\
 \left[dI(\xi_p+\tilde{\zeta}), \ B\tilde{\zeta}\right]
\frac{pI(\xi_p)+\left[dI(\xi_p
),  \ A\xi_p\right]}
{pI(\xi_p+\tilde{\zeta})+\left[dI(\xi_p
+\tilde{\zeta}),  \ A\{\xi_p+\tilde{\zeta}\}\right]}- \left[dI(\xi_p), \ B\tilde{\zeta}\right] \  ,
\end{multline} 
then (\ref{C3}), (\ref{T3}) are equivalent. Next we obtain conditions on the functional $I(\cdot)$ which imply that $\del_1(\cdot), \ \del_2(\cdot)$ given by (\ref{U3}), (\ref{V3})  are Lipschitz continuous in the $m=1$ norm (\ref{I1}). 
\begin{lem}
Assume $I(\cdot), \ h(\cdot)$ satisfy the conditions of Lemma 3.1, and in addition $I(\cdot)$ is differentiable and has the property that there exist constants $C_1,\ve_1>0$  such that
\be \label{W3}
\|dI(\xi_p+\tilde{\zeta}_1)-dI(\xi_p+\tilde{\zeta}_2) \|_{L^1(\R^+)} \ \le \  C_1\|\tilde{\zeta}_1-\tilde{\zeta}_2\|_{1,\infty} \quad {\rm for \ } \|\tilde{\zeta}_j\|_{1,\infty}<\ve_1, \ j=1,2.
\ee
Then if $\del_1(\cdot), \ \del_2(\cdot)$ are given by (\ref{U3}), (\ref{V3}), there exist constants $C_2,\ve_2>0$ such that
\begin{eqnarray}  \label{X3}
|\del_1(\tilde{\zeta}_1)-\del_1(\tilde{\zeta}_2)| \ &\le& \  C_2\|\tilde{\zeta}_1-\tilde{\zeta}_2\|_{1,\infty} \ , \\
|\del_2(\tilde{\zeta}_1)-\del_2(\tilde{\zeta}_2)| \ &\le& \  
C_2\{\|\tilde{\zeta}_1\|_{1,\infty}+\|\tilde{\zeta}_2\|_{1,\infty}\}\|\tilde{\zeta}_1-\tilde{\zeta}_2\|_{1,\infty} \ , \nonumber
\end{eqnarray}
provided  $\|\tilde{\zeta}_j\|_{1,\infty}<\ve_2, \ j=1,2$.
\end{lem}
\begin{proof}
We have that
\be \label{Y3}
I(\xi_p+\tilde{\zeta}_1)-I(\xi_p+\tilde{\zeta}_2) \ = \ \int_0^1 d\la \left[dI(\xi_p+\la\tilde{\zeta}_1+(1-\la)\tilde{\zeta}_2), \tilde{\zeta}_1-\tilde{\zeta}_2\right] \ .
\ee
It follows from (\ref{W3}), (\ref{Y3}) that
\be \label{Z3}
\left|I(\xi_p+\tilde{\zeta}_1)-I(\xi_p+\tilde{\zeta}_2) \right| \ \le \ \left\{ \|dI(\xi_p) \|_{L^1(\R^+)}+C_1\ve_1\right\}\|\tilde{\zeta}_1-\tilde{\zeta}_2\|_\infty 
\quad {\rm for \ } \|\tilde{\zeta}_j\|_{1,\infty}<\ve_1, \ j=1,2.
\ee
The result follows from (\ref{W3}), (\ref{Z3}) and the inequality $\sup_{y>\ve_0} |B\zeta(y)|\le (1/p+1/\ve_0)\|\zeta(\cdot)\|_{1,\infty}$.
\end{proof}
Let $\del:[0,\infty)\ra\R$ be a continuous function  and consider the linear PDE
\be \label{AA3}
\frac{d\tilde{\xi}(t)}{dt} +[B+\del(t)A]\tilde{\xi}(t)-[dI(\xi_p),B\tilde{\xi}(t)]\frac{A\xi_p}{pI(\xi_p)+[dI(\xi_p), \ A\xi_p]} \ = \ 0 \ .  
\ee
We show that the results of Proposition 3.1 extend to solutions of (\ref{AA3}) provided $\|\del(\cdot)\|_\infty$ is sufficiently small. 
\begin{lem}
Assume that $h(\cdot)$ and $I(\cdot)$ satisfy the conditions of Lemma 3.1 and also that $\sup_{y>0}y^2h''(y)<\infty$.   Assume further that $\del:[0,\infty)\ra\R$ is a continuous function and $\|\del(\cdot)\|_\infty\le 1/p$. Then  the linear evolution equation  (\ref{AA3}) with initial data $\tilde{\xi}_0:(0,\infty)\ra\R$  satisfying $\|\tilde{\xi}_0(\cdot)\|_{1,\infty}<\infty$ has a unique solution globally in time,
$\tilde{\xi}(y,t;\del(\cdot)), \ y,t\ge 0,$ which has $\|\tilde{\xi}(\cdot,t;\del(\cdot))\|_{1,\infty}<\infty$ for all $t\ge 0$.  For any $q>p$ there exists $C_q,\ve_q>0$ such that
if $\|\del(\cdot)\|_\infty<\ve_q$ then for $m=1,2,$
\be \label{AB3}
\|\tilde{\xi}(\cdot,t;\del(\cdot))\|_{m,\infty}\le C_q e^{-t/q}\|\tilde{\xi}_0(\cdot)\|_{m,\infty} \quad  {\rm when \ }  t\ge 0 \ . 
\ee
We may further choose $C_q,\ve_q>0$  such that if $\|\del_j(\cdot)\|_\infty<\ve_q, \ j=1,2,$ then
\begin{multline} \label{AB*3}
\|\tilde{\xi}(\cdot,t;\del_1(\cdot))-\tilde{\xi}(\cdot,t;\del_2(\cdot))\|_{1,\infty}\le \\
C_q te^{-t/q}\|\tilde{\xi}_0(\cdot)\|_{2,\infty}\|\del_1(\cdot)-\del_2(\cdot)\|_\infty \quad  {\rm when \ }  t\ge 0 \ . 
\end{multline}
\end{lem}
\begin{proof}
We observe analogously to (\ref{E3})  that the solution to (\ref{AA3}) satisfies
\be \label{AC3}
\tilde{\xi}(t;\del(\cdot)) \ = \ G(t,0;\del(\cdot))\tilde{\xi}(0)+\int_0^t ds \ [dI(\xi_p),B\tilde{\xi}(s;\del(\cdot))]\frac{G(t,s;\del(\cdot))A\xi_p}{pI(\xi_p)+[dI(\xi_p), \ A\xi_p]}  \ ,
\ee
where   $G(t,s;\del(\cdot)),  \ 0<s<t,$ acts on functions $\zeta:(0,\infty)\ra\R$ as
\be \label{AD3}
G(t,s;\del(\cdot))\zeta(y) \ = \ \exp\left[-\int_s^t\rho(s') \ ds'\right]\zeta(y_{\rho(\cdot)}(s)) \ , \quad \rho(s')=\frac{1}{p}+\del(s') \ ,
\ee
and $y(\cdot)=y_{\rho(\cdot)}(s)$ is given by (\ref{C2}). We set $u(t)=[dI(\xi_p),B\tilde{\xi}(t;\del(\cdot))]$, and  then (\ref{AC3}) yields an integral equation for $u$,
\be \label{AE3}
u(t;\del(\cdot))+\int_0^t  K(t,s;\del(\cdot))u(s;\del(\cdot)) \ ds \ = \ g(t;\del(\cdot))   \ , \quad t>0,
\ee
where the functions $K,g$ are given by
\be \label{AF3}
K(t,s;\del(\cdot))  = -\frac{[dI(\xi_p),BG(t,s;\del(\cdot))A\xi_p]}{pI(\xi_p)+[dI(\xi_p), \ A\xi_p]}  \ , \quad g(t;\del(\cdot))=[dI(\xi_p),BG(t,0;\del(\cdot))\tilde{\xi}(0)] \quad t\ge 0.
\ee
Using the fact that $\sup_{y>\ve_0} |B\zeta(y)|\le (1/p+1/\ve_0)\|\zeta(\cdot)\|_{1,\infty}$, it follows from (\ref{AD3}), (\ref{AF3}) that $g:[0,\infty)\ra\R$ and $K:\{(t,s): \ 0\le s\le t\}\ra\R$ are continuous functions. In addition because the function $\rho(\cdot)$ in (\ref{AD3}) is non-negative,  there is a constant $C$ such that
\begin{eqnarray} \label{AG3}
 |K(t,s;\del(\cdot))| \ &\le& \  C\exp\left[-\int_s^t\rho(s') \ ds'\right]  \quad t\ge s\ge 0 \ , \\
 |g(t;\del(\cdot))| \ &\le& \ C\exp\left[-\int_0^t\rho(s') \ ds'\right]\|\tilde{\xi}(0)\|_{1,\infty} \quad t\ge 0 \ . \nonumber
\end{eqnarray}
It follows from (\ref{AG3}) and the theory of Volterra integral equations (see Chapter 9 of \cite{grip}) that  there is a unique continuous solution $u:[0,\infty)\ra\R$  to (\ref{AE3}). Global existence and the inequality  $\|\tilde{\xi}(\cdot,t;\del(\cdot))\|_{1,\infty}<\infty$ now follows as in Proposition 3.1 from the representation (\ref{AC3}). 

To obtain the inequality (\ref{AB3}) it is sufficient to show that for any $q>p$ there exists $C_q,\ve_q>0$ such that if $\|\del(\cdot)\|_\infty<\ve_q$ then the solution $u(t;\del(\cdot))$ of (\ref{AE3}) satisfies $|u(t;\del(\cdot))|\le C_qe^{-t/q}\|\tilde{\xi}_0(\cdot)\|_{1,\infty}, \ t\ge 0$. To do this we write (\ref{AE3}) in operator notation as
\be \label{AH3}
u+K_{\del(\cdot)}u \ = \ g \ , \quad K_{\del(\cdot)}u(t) \ = \ \int_0^t K(t,s;\del(\cdot))u(s) \ ds \ ,
\ee
with solution
\be \label{AI3}
u \ = \ g-R_{\del(\cdot)}g \ , \quad R_{\del(\cdot)}g(t,\del(\cdot)) \ = \ \int_0^t r(t,s;\del(\cdot)) g(s;\del(\cdot)) \ ds \ .
\ee
For $q>p$ we set $g_q(t,\del(\cdot))=e^{t/q}g(t,\del(\cdot)), \ t\ge 0, $ and $K_q(t,s,\del(\cdot))=e^{(t-s)/q}K(t,s,\del(\cdot)), \ t\ge s\ge 0$. Then (\ref{AH3}) is equivalent to the integral equation
\be \label{AJ3}
u_q+K_{\del(\cdot),q}u_q \ = \ g_q \ , \quad K_{\del(\cdot),q}u(t) \ = \ \int_0^t K_q(t,s,\del(\cdot))u_q(s) \ ds \ .
\ee
For $\|\del(\cdot)\|_\infty<1/p-1/q$ we have from (\ref{AG3}) that $\sup_{t>0} \int_0^t |K_q(t,s,\del(\cdot))| \ ds<\infty$, whence the operator $K_{\del(\cdot),q}$ is bounded on $L^\infty(\R^+)$. 
The solution to (\ref{AJ3}) is given by $u_q(t;\del(\cdot))=e^{t/q}u(t;\del(\cdot)), \ t\ge0$, where $u(t;\del(\cdot))$ is the solution to (\ref{AH3}), and it can be represented as
\be \label{AK3}
u_q \ = \ g_q-R_{\del(\cdot),q}g_q \ , \quad R_{\del(\cdot),q}g_q(t,\del(\cdot)) \ = \ \int_0^t r_q(t,s,\del(\cdot)) g_q(s,\del(\cdot)) \ ds \ ,
\ee
where $r_q(t,s,\del(\cdot))=e^{(t-s)/q}r(t,s,\del(\cdot)), \ t\ge s\ge 0$.

We show that for $\|\del(\cdot)\|_\infty$ sufficiently small one has  $\sup_{t>0} \int_0^t |r_q(t,s,\del(\cdot))| \ ds<\infty$, whence $\|u_q\|_\infty\le C_q\|\tilde{\xi}(0)\|_{1,\infty}$ for some constant $C_q$ depending on $q$. To do this we regard (\ref{AJ3}) as a perturbation about the $\del(\cdot)\equiv 0$ integral equation when $K_{\del(\cdot),q}$ corresponds to the kernel $K_q(t,s,\del(\cdot))=e^{(t-s)/q}K(t-s)$ with $K(\cdot)$ given by (\ref{G3}).  From (\ref{K3}), (\ref{AF3})  we see that for any $\tau,\ve>0$ there exists $\eta>0$ such that $|K(t,s,\del(\cdot))-K(t-s)|\le \ve$ when $s+\tau\ge t\ge s\ge 0$ provided $\|\del(\cdot)\|_\infty<\eta$. It follows from this and (\ref{AG3}) that for any $q>p$ we can choose $\ve_q>0$ sufficiently small so that  if $\|\del(\cdot)\|_\infty<\ve_q$ then 
\be \label{AL3}
\sup_{t>0}\int_0^t\left|K_q(t,s,\del(\cdot))-e^{(t-s)/q}K(t-s)\right| \ ds \ < \ \left\{1+\int_0^\infty e^{s/q}|r(s)| \ ds\right\}^{-1} \ ,
\ee
where $r(\cdot)$ is the resolvent for $K(\cdot)$ of (\ref{G3}). We conclude from (\ref{AL3}) that the integral equation (\ref{AJ3}) is invertible in the space $L^\infty(\R^+)$. Now we can argue as in Proposition 3.1 to show using the representation  (\ref{AC3}) that the results of Proposition 3.1 continue to hold for any $q>p$, provided we choose $\ve_q>0$ sufficiently small. In particular, the inequality (\ref{AB3}) holds. 

Next  we examine the dependence  on the function $\del(\cdot)$ of the solution to  (\ref{AA3}).  Let $\del_1,\del_2:[0,\infty)\ra\R$ be two continuous functions and for $0<\la<1$ denote by $\rho_\la(\cdot)$ the function $\rho_\la(\cdot)=1/p+\la\del_1(\cdot)+(1-\la)\del_2(\cdot)$. Then we have from (\ref{C2}), (\ref{AD3}) and the fundamental theorem of calculus that
\begin{multline} \label{AM3}
G(t,s,\del_1(\cdot))\zeta(y)-G(t,s,\del_2(\cdot))\zeta(y) \ = \\
\int_0^1d\la \ \left[    \int_s^t \{\del_2(s')-\del_1(s')\} \ ds'              \right] \exp\left[-\int_s^t\rho_\la(s') \ ds'\right]\zeta(y_{\rho_\la(\cdot)}(s)) \\
+ \int_0^1d\la \ \left[    \int_s^t \{\del_1(s')-\del_2(s')\} \ ds'              \right] yD\zeta(y_{\rho_\la(\cdot)}(s)) \\
+ \int_0^1d\la \ \left[    \int_s^t ds' \ \left\{\int_s^{s'} \{\del_1(s'')-\del_2(s'')\} \ ds'' \right\}  \ \exp\left\{-\int_{s'}^t\rho_\la(s'') \ ds''\right\}  \right] D\zeta(y_{\rho_\la(\cdot)}(s)) \ .
\end{multline}
Observe now from (\ref{C2}) that
\begin{multline} \label{AN3}
\left| \  \left[    \int_s^t ds' \ (s'-s)  \ \exp\left\{-\int_{s'}^t\rho_\la(s') \ ds'\right\}  \right]D\zeta(y_{\rho_\la(\cdot)}(s)) \  \right| \ \le \\
\exp\left[-\int_s^t\rho_\la(s') \ ds'\right] \ \left[\sup_{y>0}y|D\zeta(y)|\right] \ \times \\
 \int_s^t ds' \ (s'-s)  \exp\left[\int_s^{s'}\rho_\la(s') \ ds'\right] \ \Bigg/   \int_s^t ds' \   \exp\left[\int_s^{s'}\rho_\la(s'') \ ds''\right]  \ .
\end{multline}
We conclude from (\ref{AM3}), (\ref{AN3}) that
\begin{multline} \label{AO3}
\|G(t,s,\del_1(\cdot))\zeta(\cdot)-G(t,s,\del_2(\cdot))\zeta(\cdot)\|_\infty \ \le \\
2(t-s)\sup_{0<\la<1}\exp\left[-\int_s^t\rho_\la(s') \ ds'\right]\|\zeta(\cdot)\|_{1,\infty}\|\del_1(\cdot)-\del_2(\cdot)\|_\infty  \ .
\end{multline} 
By differentiating (\ref{AM3}) we similarly see that
\begin{multline} \label{AP3}
\sup_{y>0}|yDG(t,s,\del_1(\cdot))\zeta(y)-yDG(t,s,\del_2(\cdot))\zeta(y)\ | \ \le \\
2(t-s)\sup_{0<\la<1}\exp\left[-\int_s^t\rho_\la(s') \ ds'\right]\|\zeta(\cdot)\|_{2,\infty}\|\del_1(\cdot)-\del_2(\cdot)\|_\infty  \ .
\end{multline} 
Since $\|A\xi_p(\cdot)\|_{2,\infty}<\infty$, it follows from (\ref{AF3}), (\ref{AO3}), (\ref{AP3}) that for some constant $C$, 
\begin{multline} \label{AQ3}
 |K(t,s;\del_1(\cdot))-K(t,s;\del_2(\cdot))| \ \le  \\
  C(t-s)\sup_{0<\la<1}\exp\left[-\int_s^t\rho_\la(s') \ ds'\right]\|\del_1(\cdot)-\del_2(\cdot)\|_\infty  \ .  \quad t\ge s\ge 0 \  .
\end{multline}
Similarly we have that
\begin{multline} \label{AR3}
 |g(t;\del_1(\cdot))-g(t;\del_2(\cdot))| \ \le  \\
  Ct\sup_{0<\la<1}\exp\left[-\int_0^t\rho_\la(s') \ ds'\right]\|\tilde{\xi}(0)\|_{2,\infty}\|\del_1(\cdot)-\del_2(\cdot)\|_\infty  \ .  \quad t\ge  0 \  ,
\end{multline}
for some constant $C$.

It follows from (\ref{AQ3}), (\ref{AR3}) that we may  choose $C_q,\ve_q>0$  such that if $\|\del_j(\cdot)\|_\infty<\ve_q, \ j=1,2,$ then
\begin{multline} \label{AS3}
|u(t;\del_1(\cdot))-u(t;\del_2(\cdot))| \ \le \\
C_q te^{-t/q}\|\tilde{\xi}_0(\cdot)\|_{2,\infty}\|\del_1(\cdot)-\del_2(\cdot)\|_\infty \quad  {\rm when \ }  t\ge 0 \ . 
\end{multline}
The inequality (\ref{AB*3}) then follows from the representation (\ref{AC3}) and the inequalities (\ref{AO3}), (\ref{AP3}), (\ref{AS3}). 
\end{proof}
\begin{theorem}
Assume that $h(\cdot)$ and $I(\cdot)$ satisfy the conditions of Lemma 3.1 and also that $\sup_{y>0}y^2h''(y)<\infty$.   Let $\del_1(\cdot), \ \del_2(\cdot)$ be real valued functionals of $C^1$ functions $\tilde{\zeta}:(0,\infty)\ra\R$, which satisfy $\del_1(0)=\del_2(0)=0$ and the local Lipschitz conditions (\ref{X3}).  Then there exists  $\ve>0$ such that  the nonlinear evolution equation  (\ref{T3}) with initial data $\tilde{\xi}_0:(0,\infty)\ra\R$  satisfying 
$\|\tilde{\xi}_0(\cdot)\|_{1,\infty}<\ve, \ \|\tilde{\xi}_0(\cdot)\|_{2,\infty}<\infty,$ has a unique solution globally in time. For any $q>p$ there exists $C_q,\ve_q>0$ such that for  $\ve\le \ve_q$ and $m=1,2$,
\be \label{AT3}
\|\tilde{\xi}(\cdot,t)\|_{m,\infty}\le C_q e^{-t/q}\|\tilde{\xi}_0(\cdot)\|_{m,\infty} \quad  {\rm when \ }  t\ge 0 \ . 
\ee
\end{theorem}
\begin{proof}
We first use a contraction mapping argument to prove local existence and uniqueness.  Let $\mathcal{E}$ be the Banach space of $C^1$ functions $\tilde{\zeta}:(0,\infty)\ra\R$ where the norm of $\tilde{\zeta}(\cdot)$ is given by (\ref{I1}) with $m=1$.  For $T,\ve>0$ we denote by $\mathcal{E}_{\ve,T}$ the space of continuous functions $\chi:[0,T]\ra\mathcal{E}$ satisfying $\sup_{0\le t\le T}\|\chi(t)\|_{1,\infty}<\ve$. We define the mapping on $\mathcal{E}_{\ve,T}$  by considering solutions $\tilde{\xi}(t), \ t>0,$ to the non-homogeneous linear evolution equation
\be \label{AU3}
\frac{d\tilde{\xi}(t)}{dt} +[B+\del_1(t)A]\tilde{\xi}(t) -\left\{[dI(\xi_p),B\tilde{\xi}(t)]+\del_2(t)\right\}\frac{A\xi_p}{pI(\xi_p)+[dI(\xi_p), \ A\xi_p]} \ = \ 0 \ ,
\ee
where $\del_1,\del_2:[0,\infty)\ra\R$ are continuous functions with $\|\del_1(\cdot)\|_\infty\le 1/p$. 

 The solution to the initial value problem for (\ref{AU3}) can be represented in terms of the Green's function for the homogeneous equation (\ref{AA3}). 
Let $\del:[0,\infty)\ra\R$ be a continuous function satisfying $\|\del(\cdot)\|_\infty\le 1/p$. We define the Green's function $G(t,s;\del(\cdot))$ for $t\ge s>0$ as the bounded linear operator on the Banach space $\mathcal{E}$ such that for $\tilde{\xi}_s\in\mathcal{E}$, the function $\tilde{\xi}(t)=G(t,s;\del(\cdot))\tilde{\xi}_s, \ t>s,$  is the solution to (\ref{AA3}) with initial condition $\tilde{\xi}(s)=\tilde{\xi}_s$.  Evidently the solution to (\ref{AU3}) with initial condition $\tilde{\xi}(0)=\tilde{\xi}_0\in\mathcal{E}$ has the representation
\be \label{AV3}
\tilde{\xi}(t) \ = \ G(t,0;\del_1(\cdot))\tilde{\xi}_0+\int_0^t \del_2(s)\frac{ G(t,s;\del_1(\cdot))A\xi_p}{pI(\xi_p)+[dI(\xi_p), \ A\xi_p]}  \ ds \ .
\ee

For $\chi\in\mathcal{E}_{\ve,T}$ we define $\del_1(t)=\del_1(\chi(t)), \ \del_2(t)=\del_2(\chi(t)), \ 0\le t\le T$. From (\ref{X3}) we see that if $\ve>0$ is sufficiently small then $\|\del_1(\cdot)\|_\infty\le 1/p$.  Hence we may use the representation (\ref{AV3}) to define the mapping $K\chi:[0,T]\ra\mathcal{E}$ by $K\chi(t)=\tilde{\xi}(t), \ 0\le t\le T$. It follows from Lemma 3.3 inequality (\ref{AB3}) with $m=1$ that there exists $\eta,T_0>0$ such that if $\|\tilde{\xi}_0(\cdot)\|_{1,\infty}<\eta$ and $T\le T_0$ then $K$ is a mapping on $\mathcal{E}_{\ve,T}$. Similarly we see from (\ref{AB*3}) that if $\|\tilde{\xi}_0(\cdot)\|_{2,\infty}<\infty$ then $K$ is a contraction mapping on $\mathcal{E}_{\ve,T}$ for $T_0$ sufficiently small,  if we define the distance function by the uniform norm, $d(\chi,\chi')=\sup_{0\le t\le T}\|\chi(t)-\chi'(t)\|_{1,\infty}$. The contraction mapping theorem then implies existence of a unique solution $\tilde{\xi}(t)\in\mathcal{E}$ to (\ref{T3}) in the interval $0<t\le T_0$. 

We extend the local solution of (\ref{T3}) to all time by obtaining a-priori bounds.  Assume that for some $\eta,T>0$ there is a solution $\tilde{\xi}(t)\in\mathcal{E}$ to (\ref{T3}) for $0<t\le T$ satisfying $\|\tilde{\xi}(t)\|_{1,\infty}<\eta$.  From (\ref{X3}) we have in (\ref{AV3}) that $|\del_1(t)|\le C_2\eta, \ |\del_2(t)|\le C_2\eta\|\tilde{\xi}(t)\|_{1,\infty}$ for $0<t\le T$.  It follows then from (\ref{AB3}), on choosing $\eta$ sufficiently small, that for  some constants $C_3,C_4$, 
\be \label{AW3}
\sup_{0<t\le T}\|\tilde{\xi}(t)\|_{1,\infty} \ \le \  C_3\|\tilde{\xi}_0\|_{1,\infty}+C_4\eta\sup_{0<t\le T} \|\tilde{\xi}(t)\|_{1,\infty} \ .
\ee
Choosing $C_4\eta\le 1/2$, we conclude that  
\be \label{AX3}
\sup_{0<t\le T}\|\tilde{\xi}(t)\|_{1,\infty} \ \le \  2C_3\|\tilde{\xi}_0\|_{1,\infty} \ .
\ee
 We can use (\ref{AX3}) to obtain from (\ref{AV3}) an a-priori bound on $\|\tilde{\xi}(t)\|_{2,\infty}$.  Thus using (\ref{AB3}) we have that
\be \label{AY3}
\sup_{0<t\le T}\|\tilde{\xi}(t)\|_{2,\infty} \ \le \  C_5\|\tilde{\xi}_0\|_{2,\infty}+C_6\|\tilde{\xi}_0\|_{1,\infty}^2 \ .
\ee
Global existence of a unique solution to (\ref{T3}) follows easily from (\ref{AX3}), (\ref{AY3}) by choosing $\|\tilde{\xi}_0\|_{1,\infty}$ sufficiently small so that  the RHS of (\ref{AX3}) is smaller than $\eta$.   To see this we assume that a solution $\tilde{\xi}(t)$ satisfying $\|\tilde{\xi}(t)\|_{1,\infty}<\eta$ exists for $0<t\le T$. From (\ref{AX3}), (\ref{AY3}) there exists $\del_0>0$ depending only on $\|\tilde{\xi}_0\|_{1,\infty}$ and $\|\tilde{\xi}_0\|_{2,\infty}$ such that a solution exists up to time $T+\del_0$ with $\|\tilde{\xi}(t)\|_{1,\infty}<\eta$ for $0<t<T+\del_0$. 

The exponential decay estimate (\ref{AT3})  follows in a similar way from (\ref{AV3}). Using (\ref{AB3}) we see that for $q>p$  there exists $\ve_q>0$ and for $\|\tilde{\xi}_0\|_{1,\infty}<\ve_q$, one has 
\be \label{AZ3}
\sup_{0<t\le T}e^{t/q}\|\tilde{\xi}(t)\|_{1,\infty} \ \le \  C_3\|\tilde{\xi}_0\|_{1,\infty}+\frac{1}{2}\sup_{0<t\le T} e^{t/q}\|\tilde{\xi}(t)\|_{1,\infty} \ .
\ee
Evidently (\ref{AT3}) with $m=1$ follows from (\ref{AZ3}).  The inequality for $m=2$ follows similarly by introducing an exponential factor into (\ref{AY3}). 
\end{proof}

\vspace{.3in}

\section {A Differential Delay Equation}
In this section we shall give an alternative proof of Theorem 3.1 by  obtaining results on the asymptotic behavior of solutions to the differential delay equation (DDE) satisfied by $I(t)=I(\xi(\cdot,t))$, where $\xi(\cdot,t), \ t\ge 0,$ is a solution to (\ref{A1}).  To derive the equation we first observe from (\ref{E1}) that
\be \label{A4}
\exp\left[\int_s^t \rho(s') \ ds'\right]  \ = \ e^{(t-s)/p}\left(\frac{I(t)}{I(s)}\right)^{1/p} \  .
\ee
It follows now from (\ref{C2}), (\ref{D2}) that $\xi(y,t)$ is a function of $I(s), \ 0\le s\le t$. The differential delay equation is therefore given from (\ref{E1}), (\ref{B3}) by 
\be \label{B4}
\frac{1}{p}\frac{d}{dt}\log I(t) \ = \ -\frac{\left[dI(\xi(\cdot,t)), \ B\{\xi(\cdot,t)-\xi_p(\cdot)\}\right]}{pI(\xi(\cdot,t))+\left[dI(\xi(\cdot,t)),  \ A\xi(\cdot,t)\right]} \ .
\ee 
Evidently $I(\cdot)\equiv I_p=I(\xi_p)$ is a solution to (\ref{B4}) in the case when $\xi(\cdot,0)=\xi_p(\cdot)$.  We obtain the linearization of (\ref{B4}) about the constant  $I_p$ when  $\xi(\cdot,0)=\xi_p(\cdot)+\tilde{\xi}(\cdot,0)$  by  writing $I(t)=[p\tilde{I}(t)+1]I_p$, whence
\be \label{C4}
\left(\frac{I(t)}{I(s)}\right)^{1/p} \ \simeq \ 1+[\tilde{I}(t)-\tilde{I}(s)] \ . 
\ee
Letting $y_p(s)=e^{(t-s)/p}y+p\left[e^{(t-s)/p}-1\right], \ s\le t,$ be the solution to (\ref{B2}) in the case $\rho(\cdot)\equiv 1/p$, we have from (\ref{A4}), (\ref{C4})  that the function $y(\cdot)$ of (\ref{C2}) is given to first order in $\tilde{I}(\cdot)$ by 
\be \label{D4}
y(s)-y_p(s)  \simeq \ e^{(t-s)/p}[\tilde{I}(t)-\tilde{I}(s)]y
+\int_s^te^{(s'-s)/p}[\tilde{I}(s')-\tilde{I}(s)]  \ ds' \ . 
\ee
It follows from (\ref{D2}), (\ref{D4}) that if $\xi(\cdot,0)=\xi_p(\cdot)+\tilde{\xi}(\cdot,0)$, then
\begin{multline} \label{E4}
\tilde{\xi}(y,t) \ = \ \xi(y,t)-\xi_p(y) \ \simeq \  e^{-t/p}\tilde{\xi}(y_p(0),0)+[\tilde{I}(0)-\tilde{I}(t)]e^{-t/p}\xi_p(y_p(0))+ \\
\int_0^tds \ h(y_p(s))e^{-(t-s)/p}[\tilde{I}(s)-\tilde{I}(t)] 
+\int_0^tds \ h'(y_p(s))e^{-(t-s)/p}[y(s)-y_p(s)] \ ,
\end{multline}
where $y(s)-y_p(s)$ is given by the RHS of (\ref{D4}). Hence the linearization of (\ref{B4}) about the constant $I_p$ is given by
\be \label{F4}
\frac{d\tilde{I}(t)}{dt} \ = \  -\frac{\left[dI(\xi_p), \ B\tilde{\xi}(t)\right]}{pI(\xi_p))+\left[dI(\xi_p),  \ A\xi_p\right]} \ ,
\ee
where $\tilde{\xi}(\cdot,t)$ is given by (\ref{E4}). 

A linear differential delay equation for $\tilde{I}(t)$ can be derived from the Volterra integral equation (\ref{F3}) by observing that the solution $u(t)=\left[dI(\xi_p), \ B\tilde{\xi}(t)\right]$ of (\ref{F3}) is a constant times  the derivative of $\tilde{I}(t)$.  From (\ref{F3}), (\ref{F4}) we have then that
\be \label{G4}
\frac{d\tilde{I}(t)}{dt} +\int_0^t K(t-s)\frac{d\tilde{I}(s)}{ds} \ ds \ = \  -\frac{g(t)}{pI(\xi_p))+\left[dI(\xi_p),  \ A\xi_p\right]} \  \ .
\ee
Integrating by parts in (\ref{G4}), we conclude that $\tilde{I}(t)$ satisfies the delay equation
\begin{multline} \label{H4}
\frac{d\tilde{I}(t)}{dt} +K(0)\tilde{I}(t)+\int_0^t K'(t-s)\tilde{I}(s) \ ds \\
 = \ K(t)\tilde{I}(0) -\frac{g(t)}{pI(\xi_p))+\left[dI(\xi_p),  \ A\xi_p\right]} \  .
\end{multline}

The differential delay equation for $\tilde{I}(\cdot)$ obtained from (\ref{D4}), (\ref{E4}), (\ref{F4}) is the same as (\ref{H4}) up to terms which decay exponentially at large time. To see this we observe from (\ref{E4}) that
\begin{multline} \label{I4}
B\tilde{\xi}(y,t)  \ \simeq \ Be^{-Bt}\tilde{\xi}(y,0)+ [\tilde{I}(0)-\tilde{I}(t)] \ e^{-Bt}h(y)\\
\frac{1}{p}\int_0^t ds \  h\left(   y_p(s)  \right)e^{-(t-s)/p}[\tilde{I}(s)-\tilde{I}(t)]  \\
+\frac{1}{p}\int_0^t ds \  h'\left(   y_p(s)  \right)e^{-(t-s)/p}[y(s)-y_p(s)]  \\
- \left(1+\frac{y}{p}\right) \int_0^t ds \  h''\left(   y_p(s)    \right) [y(s)-y_p(s)] \ ,
\end{multline}
where $y(s)-y_p(s)$ is the linear function of $\tilde{I}(\cdot)$ given on the RHS of (\ref{D4}).  The coefficient of $-\tilde{I}(t)$ on the RHS of 
(\ref{I4}) is
\begin{multline} \label{J4}
e^{-t/p}h(y_p(0))+\frac{1}{p}\int_0^t ds \  h\left(   y_p(s)  \right)e^{-(t-s)/p} \\
-\frac{y}{p}\int_0^t ds \  h'\left(   y_p(s)  \right) 
+ \left(1+\frac{y}{p}\right)y \int_0^t ds \  h''\left(   y_p(s)    \right) e^{(t-s)/p} \ .
\end{multline}
After doing some integration by parts we see that (\ref{J4}) is the same as
\be \label{K4}
BA\xi_p(y)+y h'(y_p(0))+\frac{p h(y_p(0))}{p+y_p(0)}- \int_{y_p(0)}^\infty dy' \ \frac{ph(y')}{(p+y')^2} \ .
\ee
We similarly see that the coefficient of $\tilde{I}(s), \ 0<s<t$ on the RHS of (\ref{I4}) in the integral over the interval $(0,t)$  is given by
\begin{multline} \label{L4}
e^{-B(t-s)}B^2A\xi_p(y) -\\
e^{-(t-s)/p}\left[ h'(y_p(0))-\frac{ h(y_p(0))}{p+y_p(0)}+ \int_{y_p(0)}^\infty dy' \ \frac{h(y')}{(p+y')^2} \right] \ .
\end{multline}
It follows now from (\ref{G3}), (\ref{K4}), (\ref{L4}) that the differential delay equation obtained from  (\ref{D4}), (\ref{E4}), (\ref{F4})  is the same as (\ref{H4}) modulo exponentially decaying  terms.
\begin{proposition}
Assume that $h(\cdot)$ and $I(\cdot)$ satisfy the conditions of Lemma 3.1. Then the linear DDE defined by  (\ref{D4}), (\ref{E4}), (\ref{F4}) is asymptotically stable in the following sense:  Let the initial data $\tilde{\xi}_0:(0,\infty)\ra\R$  satisfy $\|\tilde{\xi}_0(\cdot)\|_{1,\infty}<\infty$.  Then for any $q>p$  there is a constant $C_q$ depending  on $q$ such that 
\be \label{M4}
\left|\frac{d\tilde{I}(t)}{dt}\right|\le C_q e^{-t/q}\left[|\tilde{I}(0)|+\|\tilde{\xi}_0(\cdot)\|_{1,\infty}\right] \quad  {\rm when \ }  t\ge 0 \ . 
\ee
\end{proposition}
\begin{proof}
From (\ref{I4}) we may write (\ref{F4}) as 
\be \label{N4}
\frac{d\tilde{I}(t)}{dt} +M(t)\tilde{I}(t)-\int_0^t m(t,s)\tilde{I}(s) \ ds  \ = \ g(t) \ .
\ee
The function $g(\cdot)$ is given by the formula
\be \label{O4}
g(t) \ = \ -\frac{\left[dI(\xi_p), \ Be^{-Bt}\tilde{\xi}(0)+\tilde{I}(0)e^{-Bt}h\right]}{pI(\xi_p))+\left[dI(\xi_p),  \ A\xi_p\right]} \ .
\ee
It follows from the assumptions of Lemma 3.1 that there is a constant $C$ such that
\be \label{P4}
|g(t)| \ \le \ Ce^{-t/p}\left[|\tilde{I}(0)|+\|\tilde{\xi}_0(\cdot)\|_{1,\infty}\right]  \ , \quad t\ge 0.
\ee
Similarly we see from (\ref{K4}), (\ref{L4}) that if $K(\cdot)$ is the function (\ref{G3}) then
\be \label{Q4}
|M(t)-K(0)| \ \le \  Ce^{-t/p} \ , \quad |m(t,s)+K'(t-s)| \ \le \ Ce^{-(t-s)/p-t/p} \ ,\\ 0\le s\le t<\infty \ ,
\ee
for some constant $C$.

We show under  the assumptions (\ref{P4}), (\ref{Q4}) there is a constant $C$ such that the solution $\tilde{I}(t)$ to (\ref{N4}) satisfies
\be \label{R4}
\sup_{t>0}|\tilde{I}(t)| \ \le \ C\left[|\tilde{I}(0)|+\|\tilde{\xi}_0(\cdot)\|_{1,\infty}\right]  \ .
\ee
To see this we first observe that for any $0\le T_1<T_2$ one has on integrating (\ref{N4}) the representation
\begin{multline} \label{S4}
\tilde{I}(T_2) \ = \ \exp\left[-\int_{T_1}^{T_2}M(s) \ ds\right]\tilde{I}(T_1)+\int_{T_1}^{T_2}dt \ \exp\left[-\int_t^{T_2}M(s) \ ds\right] g(t)\\
+\int_{T_1}^{T_2}dt \ \exp\left[-\int_t^{T_2}M(s) \ ds\right]\int_0^t m(t,s)\tilde{I}(s) \ ds  \ .
\end{multline}
We assume that $T>1$ and that $\tilde{I}(T)=\sup_{0<t<T}|\tilde{I}(t)|>0$. Setting $T_1=T-1, \ T_2=T$ in (\ref{S4}) and using (\ref{Q4}), we see that for some constant $C_1$  the RHS of (\ref{S4})  is bounded above by
\be \label{T4}
\exp[-K(0)]\tilde{I}(T-1)+\left\{1-\exp[-K(0)]+C_1e^{-T/p}\right\}\tilde{I}(T)+C_1\int_{T-1}^T |g(t)| \ dt \ .
\ee
We conclude from (\ref{S4}), (\ref{T4}) that
\be \label{U4}
\sup_{0<t<T-1}|\tilde{I}(t)| \ \ge \ \left[1-C_2e^{-T/p}\right]\sup_{0<t<T}|\tilde{I}(t)|-C_2\int_{T-1}^T |g(t)| \ dt \ ,
\ee
for some constant $C_2$. Since we can make a similar argument in the case when $\tilde{I}(T)=-\sup_{0<t<T}|\tilde{I}(t)|<0$, we see that (\ref{U4}) holds provided
$|\tilde{I}(T)|=\sup_{0<t<T}|\tilde{I}(t)|$.  It follows upon iterating the inequality (\ref{U4}) that there exist constants $T_0,C_3>0$ and
\be \label{V4}
\sup_{t>0}|\tilde{I}(t)| \ \le \  C_3\left[\sup_{0<t<T_0}|\tilde{I}(t)|+\int_0^\infty|g(t)| \ dt\right] \ .
\ee
Since it is easy to show that $\sup_{0<t<T_0}|\tilde{I}(t)|$ is bounded by the RHS of (\ref{R4}), we conclude from (\ref{P4}), (\ref{V4})  that (\ref{R4}) holds.

The inequality (\ref{M4}) easily follows from (\ref{R4}) and Theorem 3.5 of Chapter II of \cite{grip} for the Volterra equation (\ref{G3}). Setting  $u(t)=d\tilde{I}(t)/dt, \ t>0,$ we integrate by parts as in (\ref{G4}), (\ref{H4}). We have then from (\ref{P4}), (\ref{Q4}), (\ref{R4}) that $u(\cdot)$ is the solution to (\ref{G3}), with $g(\cdot)$ on the RHS of (\ref{G3}) satisfying (\ref{P4}). The result follows.
\end{proof}
We consider next the nonlinear DDE (\ref{B4}). It follows  from (\ref{C2}), (\ref{D2})  that
\begin{multline} \label{W4}
B\xi(y,t) \ = \ \frac{1}{p}\exp\left[-\int_0^t \rho(s) \ ds\right]\xi(y(0),0)- \left(1+\frac{y}{p}\right)D\xi(y(0),0) + \\
\frac{1}{p}\int_0^t ds \  h(y(s))\exp\left[-\int_s^t \rho(s') \ ds'\right]  
- \left(1+\frac{y}{p}\right) \int_0^t ds \  h'(y(s)) \ , 
\end{multline}
where $\rho(\cdot)$ is determined from (\ref{A4}).  We  define the function $v_t(s)=\{I(s)/I(t)\}^{1/p}, 0<s\le t$. Then from (\ref{C2}), (\ref{A4})  we have that
\be \label{X4}
y(s) \ = \ \frac{z(s)}{v_t(s)} \ , \quad \ {\rm where \ } z(s) \ = \ e^{(t-s)/p}y+\int_s^tds' \ e^{(s'-s)/p} \  v_t(s') \ .
\ee
We define   a function $F(t,y,v_t(\cdot))$ by
\be \label{Y4}
F(t,y,v_t(\cdot)) \ = \ \frac{1}{p}\int_0^t ds \  h\left(   \frac{z(s)}{v_t(s)}  \right) e^{-(t-s)/p}v_t(s)
- \left(1+\frac{y}{p}\right) \int_0^t ds \  h'\left(    \frac{z(s)}{v_t(s)}  \right)  \ ,
\ee
so $B\xi(y,t)$ is the sum of terms depending on the initial data plus $F(t,y,v_t(\cdot))$. When $v_t(\cdot)\equiv 1$ then $z(\cdot)=y_p(\cdot)$ and so 
\be \label{Z4}
F(t,y,1(\cdot)) \ = \ h(y)-e^{-t/p}h(y_p(0)) \ = \ B\xi_p(y)-e^{-t/p}h(y_p(0)) \  .
\ee
We conclude from (\ref{W4})-(\ref{Z4}) that
\be \label{AA4}
B\xi(y,t)-B\xi_p(y) \ = \ F(t,y,v_t(\cdot))-F(t,y,1(\cdot)) + G(t,y,v_t(\cdot)) \ ,
\ee
where the function $G$ is given by
\be \label{AA*4}
G(t,y,v_t(\cdot)) \ = \  \frac{1}{p}e^{-t/p}v_t(0)\xi\left( \frac{z(0)}{v_t(0)},0\right)- \left(1+\frac{y}{p}\right)D\xi\left(\frac{z(0)}{v_t(0)},0\right)-e^{-t/p}h(y_p(0)) \ .
\ee
Observe that if $v_t(\cdot)$ is close to $1(\cdot)$ then $|G(t,y,v_t(\cdot))|$ is bounded by a constant times $e^{-t/p}$. 

The linearization (\ref{D4})-(\ref{F4}) of (\ref{B4}) can be obtained by computing the gradient $dF(t,y,v_t(\cdot);\cdot)$ of $F(t,y,v_t(\cdot))$ with respect to $v_t(\cdot)$  at $v_t(\cdot)\equiv1$.  To find the gradient $d\Ga(v_t(\cdot);\cdot)$ of a functional $\Ga(v_t(\cdot))$ we compute the directional derivative
\be \label{AB4}
\int_0^td\Ga(v_t(\cdot);\tau)\phi(\tau) \ d\tau \ = \ [d\Ga(v_t(\cdot)),\phi] \ = \ \lim_{\ve\ra 0}\frac{\Ga(v_t(\cdot)+\ve \phi(\cdot))-\Ga(v_t(\cdot))}{\ve} \ .
\ee
It follows from (\ref{X4}), (\ref{AB4}) that
\be \label{AC4}
dz(s)(v_t(\cdot);\tau) \ = \ e^{(\tau-s)/p}H(\tau-s)\  , \quad {\rm where \ } H(\cdot) \ {\rm is \ the \  Heaviside \ function.}
\ee
Similarly we have that 
\begin{multline} \label{AD4}
dF(t,y,v_t(\cdot);\tau) \ = \ \frac{1}{p}  h\left(   \frac{z(\tau)}{v_t(\tau)}  \right)e^{-(t-\tau)/p}- \frac{z(\tau)}{pv_t(\tau)}  h'\left(   \frac{z(\tau)}{v_t(\tau)}  \right)e^{-(t-\tau)/p}\\
+ \frac{1}{p}\int_0^t ds \  h'\left(   \frac{z(s)}{v_t(s)}  \right) e^{-(t-s)/p}dz(s)(v_t(\cdot);\tau) + \\
\left(1+\frac{y}{p}\right) \frac{z(\tau)}{v_t(\tau)^2} h''\left(   \frac{z(\tau)}{v_t(\tau)}  \right) 
-\left(1+\frac{y}{p}\right)\int_0^t ds \  \frac{1}{v_t(s)}h''\left(   \frac{z(s)}{v_t(s)}  \right) dz(s)(v_t(\cdot);\tau) \  .
\end{multline}
We conclude from (\ref{AC4}), (\ref{AD4}) that
\begin{multline} \label{AE4}
dF(t,y,v_t(\cdot);\tau) \ = \ \frac{1}{p}  h\left(   \frac{z(\tau)}{v_t(\tau)}  \right)e^{-(t-\tau)/p}- \frac{z(\tau)}{pv_t(\tau)}  h'\left(   \frac{z(\tau)}{v_t(\tau)}  \right)e^{-(t-\tau)/p}\\
+ \frac{e^{(\tau-t)/p}}{p}\int_0^\tau ds \  h'\left(   \frac{z(s)}{v_t(s)}  \right)  + \\
\left(1+\frac{y}{p}\right) \frac{z(\tau)}{v_t(\tau)^2} h''\left(   \frac{z(\tau)}{v_t(\tau)}  \right) 
-\left(1+\frac{y}{p}\right)\int_0^\tau ds \  \frac{1}{v_t(s)}h''\left(   \frac{z(s)}{v_t(s)}  \right)  e^{(\tau-s)/p} \  .
\end{multline}
Setting $v_t(\cdot)\equiv 1$ in (\ref{AE4}),  we have that
\begin{multline} \label{AF4}
dF(t,y,1(\cdot);\tau) \ = \ \frac{1}{p}  h\left(  y_p(\tau)  \right)e^{-(t-\tau)/p}- \frac{y_p(\tau)}{p}  h'\left(  y_p(\tau) \right)e^{-(t-\tau)/p}\\
+ \frac{e^{(\tau-t)/p}}{p}\int_0^\tau ds \  h'\left(   y_p(s)  \right)  + \\
\left(1+\frac{y}{p}\right)y_p(\tau) h''\left(   y_p(\tau)  \right)  \ 
-\left(1+\frac{y}{p}\right)\int_0^\tau ds \  h''\left(   y_p(s)  \right)  e^{(\tau-s)/p} \  .
\end{multline}
By doing some integration by parts in the RHS of (\ref{AF4}) we see that $dF(t,y,1(\cdot);s), \  0<s<t$, is the same as (\ref{L4}). 

From (\ref{C2}), (\ref{D2}), (\ref{AA4}) we may rewrite the DDE equation (\ref{B4}) as
\be \label{AH4}
\frac{1}{p}\frac{d}{dt}\log I(t) + f(t,v_t(\cdot)) \ = \ g(t,v_t(\cdot)) \ ,
\ee
where the functions $f,g$ are given by the formulae
\begin{eqnarray} \label{AI4}
f(t,v_t(\cdot)) \ &=& \ \frac{\left[dI(\xi(\cdot,t)), \ F(t,\cdot,v_t(\cdot))-F(t,\cdot,1(\cdot))\right]}{pI(\xi(\cdot,t))+\left[dI(\xi(\cdot,t)),  \ A\xi(\cdot,t)\right]} \ , \\
g(t,v_t(\cdot)) \ &=& \ -\frac{\left[dI(\xi(\cdot,t)), G(t,\cdot,v_t(\cdot)) \ \right]}{pI(\xi(\cdot,t))+\left[dI(\xi(\cdot,t)),  \ A\xi(\cdot,t)\right]} \ . \nonumber
\end{eqnarray}
In the following we give conditions on $I(\cdot), \ h(\cdot)$ and the initial data for (\ref{A1}) which imply a unique solution to the DDE (\ref{AH4}), (\ref{AI4}). 
\begin{lem} Assume that the function $h:(0,\infty)\ra\R$ of (\ref{B1}) is $C^2$ and $\sup_{y>0}\{y|h'(y)|+y^2|h''(y)|\}\}<\infty$. Assume  further that  $I(\cdot)$ satisfies properties (a) and (b) of the introduction. 

Let $G$ in (\ref{AI4}) be given by (\ref{AA*4}), where $\xi(\cdot,0)$ is a non-negative function satisfying $\|\xi(\cdot,0)\|_{2,\infty}<\infty$ and $pI(\xi(\cdot,0))+\left[dI(\xi(\cdot,0)),  \ A\xi(\cdot,0)\right]>0$. Then the initial value problem for (\ref{AH4}), (\ref{AI4}) with given $I(0)>0$ has a unique positive $C^1$ solution in some  interval $0<t<\del, \  \del>0$.

Suppose the positive $C^1$ solution to (\ref{AH4}), (\ref{AI4}) exists up to time $T$ and that there are constants $M,m>0$ with the property
\be \label{AJ4}
pI(\xi(\cdot,t))+\left[dI(\xi(\cdot,t)),  \ A\xi(\cdot,t)\right]\ge m, \quad  m\le v_T(t)\le M \ , \quad {\rm for \ } 0\le t\le T.
\ee
Then there exists $\del>0$, depending only on $M,m$ and $\|\xi(\cdot,0)\|_{2,\infty}$, such that the solution can be extended to a $C^1$ positive solution on the interval $[0,T+\del]$. 
\end{lem}
\begin{proof}
We assume that the positive $C^1$ solution to  (\ref{AH4}), (\ref{AI4}) exists up to time $T\ge 0$ and that (\ref{AJ4}) holds. We use the standard contraction mapping argument to extend the solution to the interval $[T,T+\del]$. Hence we need to establish boundedness of the functions $f(t,v_t(\cdot)), g(t,v_t(\cdot))$, and  also Lipschitz continuity  in $v_t(\cdot)$. 
   
For $r=1,2,$ we define a norm $\|h(\cdot)\|_{r,\infty}^*$ for  $h(\cdot)$, similar to the norm $\|\cdot\|_{r,\infty}$ of (\ref{I1}), as follows:
\be \label{AK*4}
\|h(\cdot)\|_{r,\infty}^* \ = \ \int_0^1|h(y)| \ dy+\sup_{y>1}|h(y)|+ \sup_{0<y<\infty}\sum_{k=1}^ry^k\left|\frac{d^kh(y)}{dy^k}\right| \ .
\ee
We see from (\ref{C2}), (\ref{D2}), (\ref{A4}) that there is a constant $C_1$ depending only on $m,M$ such that  for $r=1,2,$
\be \label{AK4}
\|\xi(\cdot,t)\|_{r,\infty}\le C_1(m,M)\left[e^{-t/p}\|\xi(\cdot,0)\|_{r,\infty}+\|h\|^*_{r,\infty}\right], \quad  0\le t\le T \ .
\ee
The uniform Lipschitz continuity of $dI$ implies there is a constant $\ga>0$ such that
\be \label{AL4}
\|dI(\zeta_1(\cdot))-dI(\zeta_2(\cdot))\|_{L^1(\R^+)} \ \le \   \ga\|\zeta_1(\cdot)-\zeta_2(\cdot)\|_{1,\infty} \ ,
\ee
for any nonnegative functions $\zeta_1,\zeta_2:(0,\infty)\ra\R^+$. From (\ref{AK4}) with $r=1$ and (\ref{AL4}) we conclude that
\begin{multline} \label{AM4}
\|dI(\xi(\cdot,t))\|_{L^1(\R^+} \ \le \ \|dI(0(\cdot))\|_{L^1(\R^+}  \\
+\ga C_1(m,M)\left[e^{-t/p}\|\xi(\cdot,0)\|_{1,\infty}+\|h\|^*_{1,\infty}\right], \quad  0\le t\le T \ .
\end{multline}
From (\ref{AA*4}) we see there is a constant $C_2$ depending only on $m,M$ such that
\be \label{AN4}
|G(t,y,v_t(\cdot))|\le C_2(m,M)e^{-t/p}\left[\|\xi(\cdot,0)\|_{1,\infty}+\sup_{y>\ve_0}|h(y)|\right] \ , \quad 0<t<T, \ y\ge \ve_0 \ .
\ee
It is easy to see from (\ref{Y4}) there is a constant $C_3$ depending only on $m,M$ such that
\be \label{AO4}
|F(t,y,v_t(\cdot))|\le C_3(m,M)\|h\|^*_{1,\infty} \ , \quad 0<t<T, \ y\ge \ve_0 \ .
\ee
From (\ref{AJ4}), (\ref{AM4})-(\ref{AO4}) we conclude that the functions $f(t,v_t(\cdot)), \ g(t,v_t(\cdot))$ are bounded on $0<t<T$ by a constant which depends only on $m,M$ and $\|\xi(\cdot,0)\|_{1,\infty}$. 

To prove Lipschitz continuity we first observe from (\ref{AE4})  there is a constant $C_4$ depending only on $m,M$ such that 
\be \label{AP4}
|dF(t,y,v_t(\cdot);\tau)| \ \le \ C_4(m,M)e^{-(t-\tau)/p}\|h(\cdot)\|^*_{2,\infty}  \ , \quad 0<\tau<t<T, \ y\ge \ve_0 \ .
\ee
Let $v_t^1,v_t^2:[0,t]\ra\R^+$ be two continuous functions satisfying $m\le v_t^1(s),v_t^2(s)\le M, \ 0\le s\le t$. From (\ref{AA*4}) we see there is a constant $C_5$ depending only on $m,M$ such that
\be \label{AQ4}
|G(t,y,v^1_t(\cdot))-G(t,y,v^2_t(\cdot))|\le C_5(m,M)e^{-t/p}\|\xi(\cdot,0)\|_{2,\infty}\|v_t^1(\cdot)-v^2_t(\cdot)\|_\infty \ , \ y\ge \ve_0 \ .
\ee
Evidently from (\ref{C2}), (\ref{D2}), (\ref{A4}), (\ref{X4}) we can consider $\xi(y,t)$ as a function of $t,y,v_t(\cdot)$. One easily sees that
\begin{multline} \label{AR4}
|\xi(t,y,v^1_t(\cdot))-\xi(t,y,v^2_t(\cdot))|\le \\
C_6(m,M)\left[e^{-t/p}\|\xi(\cdot,0)\|_{1,\infty}+\|h\|^*_{1,\infty}\right]\|v_t^1(\cdot)-v^2_t(\cdot)\|_\infty \ , \ y\ge \ve_0 \ ,
\end{multline}
for some constant $C_6$ depending only on $m,M$.  Similarly we have that
\begin{multline} \label{AS4}
|yD\xi(t,y,v^1_t(\cdot))-yD\xi(t,y,v^2_t(\cdot))|\le \\
C_7(m,M)\left[e^{-t/p}\|\xi(\cdot,0)\|_{2,\infty}+\|h\|^*_{2,\infty}\right]\|v_t^1(\cdot)-v^2_t(\cdot)\|_\infty \ , \ y\ge \ve_0 \ ,
\end{multline}
for some constant $C_7$ depending only on $m,M$. The inequalities (\ref{AL4}) and (\ref{AP4})-(\ref{AS4}) imply the Lipschitz continuity of the functions  $f(t,v_t(\cdot)), \ g(t,v_t(\cdot))$ in $v_t(\cdot)$.

Suppose now  a solution of (\ref{AH4}), (\ref{AI4}) exists up to time $T\ge 0$ and satisfies (\ref{AJ4}) for some $m,M>0$.   For $0<\del,\ve<1/2$ let $\mathcal{E}_{\ve,\del}$ be the space of continuous functions  $\chi:[0,\del]\ra\R^+$ such that $\chi(0)=1$ and $1-\ve\le \inf\chi(\cdot)\le \sup\chi(\cdot)\le 1+\ve$.  We define a mapping $K$ on functions in $\mathcal{E}_{\ve,\del}$ by 
\be \label{AT4}
K\chi(\tau) \ = \ \exp\left[\int_T^{T+\tau} f(t,v_t(\cdot)) \ dt -\int_T^{T+\tau} g(t,v_t(\cdot)) \ dt \right]  \ , \quad 0<\tau\le\del \ ,
\ee
where the function $v_t(\cdot)$ is defined for any $T<t<T+\del$ by
\begin{eqnarray} \label{AU4}
v_t(s) \ &=& \ v_T(s)\chi(t-T) \quad {\rm for \ }0<s<T,  \\
v_t(s) \ &=& \ \chi(t-T)/\chi(s-T) \quad {\rm for \ } T<s<t  .  \nonumber
\end{eqnarray}
Evidently if $I(t), \ 0<t<T+\del$, is a solution to (\ref{AH4}), (\ref{AI4}) then $\chi(\tau)=[I(T)/I(T+\tau)]^{1/p}, \ 0<\tau<\del,$ is a fixed point of the mapping $K$, so $\chi=K\chi$.

For $\chi\in\mathcal{E}_{\ve,\del}$ and $T<t<T+\del$, we define $v_{t}(\cdot)$ by (\ref{AU4}). Since $m\le \inf v_T(\cdot)\le\sup v_T(\cdot)\le M$, we have that
\be \label{AV4}
m(1-\ve) \ \le \  \inf v_t(\cdot) \ \le \ \sup v_t(\cdot) \ \le \  M(1+\ve), \quad T<t<T+\del,
\ee
provided $\ve<\min[1-1/M, \ 1/m-1]$.  We also have from (\ref{C2}), (\ref{D2}), (\ref{A4}) that for some universal constant $C$,
\be \label{AW4}
|\xi(y,t)-\xi(y,T)|  \ \le  C\del\left[ \ ||\xi(\cdot,T)\|_{1,\infty}+\|h(\cdot)\|^*_{1,\infty}\right] \ , \quad y\ge \ve_0, \  T<t<T+\del.
\ee
Similarly we have that
\be \label{AX4}
|yD\xi(y,t)-yD\xi(y,T)|  \ \le  C\del\left[ \ ||\xi(\cdot,T)\|_{2,\infty}+\|h(\cdot)\|^*_{1,\infty}\right] \ , \quad y\ge \ve_0, \  T<t<T+\del,
\ee
for some universal constant $C$.
It follows from (\ref{AJ4}), (\ref{AK4}), (\ref{AL4}), (\ref{AW4}), (\ref{AX4}) that $\del>0$ can be chosen sufficiently small, depending only on $m,\|\xi(\cdot,0)\|_{2,\infty}$, with the property
\be \label{AY4}
pI(\xi(\cdot,t))+\left[dI(\xi(\cdot,t)),  \ A\xi(\cdot,t)\right]\ge m/2, \quad T<t<T+\del.
\ee
From (\ref{AN4}), (\ref{AO4}), (\ref{AU4}), (\ref{AV4}), (\ref{AY4})  we see that $\del>0$ can be chosen sufficiently small, depending only on $m,M,\|\xi(\cdot,0)\|_{2,\infty}$, such that $K$ maps $\mathcal{E}_{\ve,\del}$ into itself.  We similarly see from (\ref{AP4})-(\ref{AS4}), that for sufficiently small $\del>0$, with the same dependence, the mapping is a contraction with respect to the metric induced by the uniform norm. Hence by the contraction mapping theorem the solution to (\ref{AH4}) on the interval $0<t<T$ can be extended to the interval $T<t<T+\del$. 
\end{proof}
Next we generalize Proposition 4.1 to the non-linear DDE.
\begin{theorem} Assume that $h(\cdot)$ and $I(\cdot)$ satisfy the conditions of Lemma 3.1, Lemma 4.1, and also that $\|\xi(\cdot,0)\|_{2,\infty}<\infty$.  Then there exists $\ve>0$ such that if $\|\xi(\cdot,0)-\xi_p(\cdot)\|_{1,\infty}<\ve$,  the nonlinear DDE (\ref{AH4}), (\ref{AI4}) with given initial data $I(0)>0$ has a unique positive $C^1$ solution $I(t), \ t\ge 0,$ globally in time. For any $q>p$ there exists $C_q,\ve_q>0$ such that if $\ve<\ve_q$ then
\be \label{AZ4}
\left|\frac{d}{dt}\log I(t)\right| \ \le \ C_qe^{-t/q}\|\xi(\cdot,0)-\xi_p(\cdot)\|_{1,\infty} \ .
\ee
\end{theorem}
\begin{proof}
We observe that the inequality (\ref{AN4}) can be improved to
\begin{multline} \label{BA4}
|G(t,y,v_t(\cdot))|\le \ C_1e^{-t/p}\Big[ \ \|\xi(\cdot,0)-\xi_p(\cdot)\|_{1,\infty} \\
+\|v_t(\cdot)-1(\cdot)\|_\infty\  \Big]  \quad {\rm if \ }\|v_t(\cdot)-1(\cdot)\|_\infty\le 1/2, \ y\ge \ve_0 \ ,
\end{multline}
for some constant $C_1$ depending only on $\|h(\cdot)\|^*_{1,\infty}$.   Similarly to (\ref{AK4}) there is a constant $C_2$ depending only on $\|h(\cdot)\|^*_{2,\infty}$ such that
\begin{multline} \label{BB4}
\|\xi(\cdot,t)-\xi_p(\cdot)\|_{1,\infty} \ \le \ C_2\left[e^{-t/p}\|\xi(\cdot,0)-\xi_p(\cdot)\|_{1,\infty}+\|v_t(\cdot)-1(\cdot)\|_\infty\right] \\
 {\rm if \ }\|v_t(\cdot)-1(\cdot)\|_\infty\le 1/2 \ .
\end{multline}
It follows from  (\ref{BB4}) and the Lipschitz continuity of $dI$  that there exists positive $\eta<1/2$ such that 
\begin{multline} \label{BC4}
pI(\xi(\cdot,t))+\left[dI(\xi(\cdot,t)),  \ A\xi(\cdot,t)\right]\ge \frac{1}{2}\left\{pI(\xi_p))+\left[dI(\xi_p),  \ A\xi_p\right] \right\}  \\
{\rm provided} \quad \|v_t(\cdot)-1(\cdot)\|_\infty+ \|\xi(\cdot,0)-\xi_p(\cdot)\|_{1,\infty}  \ < \ \eta \ .
\end{multline}

In order to prove global existence of the solution to (\ref{AH4}), (\ref{AI4}) we observe that (\ref{AH4}) is equivalent to a non-linear Volterra integral equation by setting
\be \label{BD4}
u(t) \ = \ \frac{1}{p}\frac{d}{dt}\log I(t) \ , \quad {\rm whence \ } v_t(s)=\exp\left[-\int_s^t u(s') \ ds'\right] \ , \  \ 0<s<t.
\ee
Hence if we show that this Volterra equation has a solution $u(\cdot)\in L^1(\R^+)$ with small norm then global existence follows from Lemma 4.1. The Volterra equation can be written as
\be \label{BE4}
u(t)+\int_0^tK(t,s,v_t(\cdot))u(s)  \ ds \ = \ g(t,v_t(\cdot)) \  ,
\ee
where the function $K$ is given by the formula
\begin{multline} \label{BF4}
K(t,s,v_t(\cdot)) \ = \\ - \int_0^sd\tau\left\{\int_0^1v_t(\tau)^\la \ d\la\right\} \  
\frac{\left[dI(\xi(\cdot,t)), \ \int_0^1d\la \ dF(t,\cdot,\la v_t(\cdot)+(1-\la)1(\cdot);\tau)\right]}{pI(\xi(\cdot,t))+\left[dI(\xi(\cdot,t)),  \ A\xi(\cdot,t)\right]}   \ .
\end{multline}
It follows from (\ref{BB4}) that there exists $\ve_0$ with $0<\ve_0\le 1/2$ such that
\begin{multline} \label{BG4}
pI(\xi(\cdot,t,v_t(\cdot)))+\left[dI(\xi(\cdot,t,v_t(\cdot))),  \ A\xi(\cdot,t,v_t(\cdot))\right] \\ 
\ge \  \frac{1}{2}\left\{pI(\xi_p)+[dI(\xi_p),A\xi_p]\right\}>0
\end{multline}
provided 
\be \label{BH4}
\|\xi(\cdot,0)-\xi_p(\cdot)\|_{1,\infty}+\|v_t(\cdot)-1(\cdot)\|_\infty \ \le  \ \ve_0 \ .
\ee
We conclude from (\ref{AP4}), (\ref{BG4})  that
\be \label{BI4}
|K(t,s,v_t(\cdot))| \ \le \  C_1e^{-(t-s)/p}, \quad 0<s<t,
\ee
for some constant $C$  if (\ref{BH4}) holds.  Let $K(\cdot)$ be the function (\ref{G3}) and $C_2$ a suitable constant. We have from (\ref{L4}) and the  comment after (\ref{AF4}) that for any $\tau,\ve>0$ there exists $\eta>0$ such that
\begin{multline} \label{BJ4}
|K(t,s,v_t(\cdot))-K(t-s)| \ \le \ \ve +C_2e^{-t/p}\quad {\rm for \ } s+\tau\ge t\ge s\ge 0 \ , \\
{\rm provided \ } \|v_t(\cdot)-1(\cdot)\|_\infty \ \le  \ \eta \ .
\end{multline}

We show that for any $T>0$ there exists $\del_T>0$ depending on $T$ such that a  solution to (\ref{BE4}) exists up to time $T$ provided $\|\xi(\cdot,0)-\xi_p(\cdot)\|_{1,\infty} \le \del_T$. Furthermore, one has
\be \label{BK4}
\|v_T(\cdot)-1(\cdot)\|_\infty \ \le \ C_T\|\xi(\cdot,0)-\xi_p(\cdot)\|_{1,\infty}  \ ,
\ee
where $C_T$ is a constant also depending  on $T$. This follows by arguing as in Lemma 4.1. In fact let us suppose we have established (\ref{BK4}) for some $T\ge 0$. Hence there exists $\tilde{\del}_T>0$ such that if $\|\xi(\cdot,0)-\xi_p(\cdot)\|_{1,\infty} \le \tilde{\del}_T$ then (\ref{AJ4}) holds for some positive $m,M$ independent of $T$. Lemma 4.1 now implies that the solution to (\ref{BE4}) exists up to time $T+\del$, where $\del>0$ is independent of $T$.  We can estimate
 $\|v_{T+\del}(\cdot)-1\|_\infty$ from (\ref{AT4}), upon setting $K\chi=\chi$.  From (\ref{AE4}), (\ref{AI4}) we have that
 \be \label{BL4}
 |f(t,v_t(\cdot))| \ \le \ C\|v_t(\cdot)-1(\cdot)\|_\infty \quad {\rm for  \ some  \ constant  \ }C  \ ,
 \ee
 provided $\|v_t(\cdot)-1(\cdot)\|_\infty \le 1/2$. Similarly we have from (\ref{BA4})  that
  \be \label{BM4}
 |g(t,v_t(\cdot))| \ \le \  Ce^{-t/p}\Big[ \ \|\xi(\cdot,0)-\xi_p(\cdot)\|_{1,\infty} 
+\|v_t(\cdot)-1(\cdot)\|_\infty\  \Big]  \ .
 \ee
 It follows from (\ref{AT4}), (\ref{AU4}), (\ref{BL4}), (\ref{BM4}) that
 \be \label{BN4}
 |\chi(\tau)-1| \ \le \ C\tau\|\xi(\cdot,0)-\xi_p(\cdot)\|_{1,\infty}  \ , \quad 0<\tau<\del \ ,
 \ee
 for some constant $C$.  Evidently (\ref{AU4}), (\ref{BK4}), (\ref{BN4}) imply that the inequality (\ref{BK4}) with $T+\del$ replacing $T$ holds provided $\|\xi(\cdot,0)-\xi_p(\cdot)\|_{1,\infty} \le \tilde{\del}_T$.
 
 Next we show that the $\del_T, \ C_T$ of (\ref{BK4}) can be chosen independent of $T>0$.  To see this we denote for $T\ge 0$ by $u_T(\cdot)$ the function $u_T(t)=u(t+T), \ t\ge 0$.  From (\ref{BE4}) we have that
 \be \label{BO4}
u_T(t)+\int_0^tK(t+T,s+T,v_{t+T}(\cdot))u_T(s)  \ ds \ = \ g(t) \  ,
\ee
where $g(\cdot)$ is given by 
\be \label{BP4}
g(t) \ = \ g(t+T,v_{t+T}(\cdot))-\int_0^TK(t+T,s,v_{t+T}(\cdot))u(s) \ ds \ .
\ee
For $\del,M>0,$ let $T_{\del,M}>0$ be defined by
\begin{multline} \label{BQ4}
T_{\del,M} \ = \ \sup\left\{ T>0: \ \|v_T(\cdot)-1(\cdot)\|_\infty \ \le \ M\|\xi(\cdot,0)-\xi_p(\cdot)\|_{1,\infty} \ \right\} \\
{\rm provided \ }  \|\xi(\cdot,0)-\xi_p(\cdot)\|_{1,\infty}  \ < \ \del \ .
\end{multline}
From (\ref{BK4}) it follows that for any $T>0$, if $\del=\del_T, \ M=C_T$ then $T_{\del,M}\ge T$.  We also have from (\ref{BI4}), (\ref{BJ4}) that there exists $T_\infty,\eta_\infty>0$ such that for  $\|v_T(\cdot)-1(\cdot)\|_\infty \ \le \eta_\infty, \ T\ge T_\infty,$ the integral equation (\ref{BO4}) with $T=T_\infty$ is invertible in $L^1(\R^+)$ and $\|u_{T_\infty}\|_{L^1(\R^+)}\le C_\infty\|g\|_{L^1(\R^+)} $ for some constant $C_\infty$.

We need to show that $T_{\del,M}=\infty$ for some $\del,M>0$, so we assume for contradiction that $T_{\del,M}<\infty$ for all $\del,M>0$.  We also have from (\ref{BK4}) that $\lim_{\del\ra0,M\ra\infty}T_{\del,M}=\infty$, so we will consider $\del,M$ with $T_{\del,M}>T_\infty$ and $\del$ sufficiently small so that (\ref{BH4}) with $t=T_{\del,M}$ holds.  From the invertibility of (\ref{BO4}) we have that
\be \label{BR4}
\int_{T_\infty}^{T_{\del,M}} |u(t)| \ dt \ \le \ C_\infty\int_0^{T_{\del,M}-T_\infty} |g(t)|  \ dt \ .
\ee
From (\ref{BI4}), (\ref{BM4}) the integral on the RHS of (\ref{BR4}) can be estimated as
\begin{multline} \label{BS4}
\int_0^{T_{\del,M}-T_\infty} |g(t)|  \ dt  \ \le \ C_1\|\xi(\cdot,0)-\xi_p(\cdot)\|_{1,\infty} \\
+C_2e^{-T_{\infty}/p}\|v_{T_{\del,M}}(\cdot)-1(\cdot)\|_\infty  \ ,
\end{multline}
for some constants $C_1,C_2$ with $C_2$ not depending on $T_\infty$.  Observe now that
\be \label{BT4}
\|v_{T_{\del,M}}(\cdot)-1(\cdot)\|_\infty  \ \le \ C_3\|\xi(\cdot,0)-\xi_p(\cdot)\|_{1,\infty} +C_4\int_{T_\infty}^{T_{\del,M}} |u(t)| \ ds \ ,
\ee
where $C_4$ is independent of $T_\infty$.  It follows on choosing $T_\infty$ sufficiently large that (\ref{BK4}) holds with $T=T_{\del,M}$  and a constant $C_T$ determined from the constants in (\ref{BR4})-(\ref{BT4}). Evidently if $M$ is large this constant will be strictly less than $M$, contradicting the definition of $T_{\del,M}$. 

We have shown that if $\ve>0$ in the statement of the lemma  is sufficiently small, then the function $t\ra \frac{d}{dt}\log I(t)$ is in $L^1(\R^+)$ with $L^1$ norm bounded by $C\|\xi(\cdot,0)-\xi_p(\cdot)\|_{1,\infty}$ for some constant $C$.  To obtain the exponential decay (\ref{AZ4}) we repeat the above argument using the function $t\ra u(t)e^{t/q}$. 
\end{proof}
We give now an alternative proof of Theorem 3.1 with $\del_1(\cdot), \ \del_2(\cdot)$ given by (\ref{U3}), (\ref{V3}).  
\begin{proof}[Proof of Theorem 3.1]
From (\ref{D2}) we have that
\begin{multline} \label{BU4}
\xi(y,t)-\xi_p(y) \ = \ \exp\left[-\int_0^t \rho(s) \ ds\right]\left\{\xi(y(0),0)-\xi_p(y(0))\right\} \\
+\left\{\exp\left[-\int_0^t \rho(s) \ ds\right]\xi_p(y(0))-e^{-t/p}\xi_p(y_p(0))\right\} \\
+\int_0^t ds \left\{  h(y(s))\exp\left[-\int_s^t \rho(s') \ ds'\right]  - h(y_p(s))e^{-(t-s)/p} \ \right\} \ .
\end{multline}
To estimate the RHS of (\ref{BU4}) we use the inequality
\be \label{BV4}
\left|\exp\left[\sig\int_s^t \rho(s') \ ds'\right] -e^{\sig(t-s)/p}\right| \ \le \  Ce^{\sig(t-s)/p}\int_s^t |u(s')| \ ds' \ , \quad \sig=\pm1 \ , \ 0<s<t \ ,
\ee
which holds for some constant $C$.  From (\ref{C2}) and (\ref{BV4}) with $\sig=1$ it follows there is a constant $C$ such that
\be \label{BW4}
|y(s)-y_p(s)| \ \le \ C|y_p(s)|\int_s^t |u(s')| \ ds' \ , \quad 0<s<t \ .
\ee
The inequality (\ref{BV4}) with $\sig=-1$, (\ref{BW4}) and (\ref{AZ4}) enables us to estimate the RHS of (\ref{BU4}) and its derivatives with respect to $y$. This yields the inequality  (\ref{AT3}). 
\end{proof}

\vspace{.1in}

\section{The case $h(\cdot)\equiv {\rm constant}$}
It is well known that certain scalar differential delay equations are equivalent to a system of ordinary differential equations (see Chapter 7 of \cite{smith}). This is the situation for (\ref{AH4}), (\ref{AI4}) when $h(\cdot)\equiv h_\infty>0$ is a constant.  In this section we shall  use this property to prove Theorem 1.2 in the case $h(\cdot)$ constant. We define $[I_1(t),I_2(t)], \ t>0,$ by
\be \label{A5}
I_1(t) \ = \ \left(\frac{I(t)}{I(0)}\right)^{1/p}, \quad I_2(t)\  =\  \frac{1}{p}\int_0^te^{-(t-s)/p}   \left(\frac{I(s)}{I(0)}\right)^{1/p} \ ds  \ .
\ee
From (\ref{X4}) we have that
\be \label{B5}
z(0) \ = \ e^{t/p}y+pe^{t/p}\frac{I_2(t)}{I_1(t)} \ .
\ee
Hence (\ref{D2}) implies that
\be \label{C5}
\xi(y,t) \ = \ e^{-t/p}\frac{1}{I_1(t)}\xi\left(I_1(t)z(0),0\right)+ph_\infty \frac{I_2(t)}{I_1(t)} \ .
\ee
It follows from (\ref{AH4}) and (\ref{A5})-(\ref{C5}) that  $[I_1(t),I_2(t)]$ satisfies a system of equations 
\begin{eqnarray} \label{D5}
\frac{dI_1(t)}{dt}+\al\left(I_1(t),I_2(t),t\right)\left[I_1(t)-I_2(t)\right] \ &=&  \  \beta\left(I_1(t),I_2(t),t\right) \ ,\\
\frac{dI_2(t)}{dt} \ &=& \ \frac{1}{p}\left[I_1(t)-I_2(t)\right] \ , \nonumber
\end{eqnarray}
where the functions $\al(I_1,I_2,t), \ \beta(I_1,I_2,t)$ are determined by (\ref{AI4}). Thus from (\ref{B5}), (\ref{C5}) we define
\be \label{E5}
\xi(I_1,I_2,y,t) \ = \ e^{-t/p}\frac{1}{I_1}\xi\left(e^{t/p}I_1y+pe^{t/p}I_2,0\right)+ph_\infty \frac{I_2}{I_1} \ .
\ee
Then from (\ref{Y4}), (\ref{AI4}), we have that
\be \label{F5}
\al(I_1,I_2,t) \ = \ -\frac{h_\infty[dI(\xi(I_1,I_2,\cdot,t)),1(\cdot)]}{pI(\xi(I_1,I_2,\cdot,t))+[dI(\xi(I_1,I_2,\cdot,t)),A\xi(I_1,I_2,\cdot,t)]} \ .
\ee
We can obtain a formula for the function $\beta(I_1,I_2,t)$ from (\ref{Z4}), (\ref{AA*4}), (\ref{AI4}).   We define the function $\ga(I_1,I_2,y,t)$ by
\be \label{G5}
\ga(I_1,I_2,y,t) \ = \ \frac{1}{p}e^{-t/p}\xi\left(e^{t/p}I_1y+pe^{t/p}I_2,0\right) - \left(1+\frac{y}{p}\right)I_1D\xi\left(e^{t/p}I_1y+pe^{t/p}I_2,0\right) \ .
\ee
Then $\beta(I_1,I_2,t)$ is given by the formula
\be \label{H5}
\beta(I_1,I_2,t) \ = \ -\frac{[dI(\xi(I_1,I_2,\cdot,t)),\ga(I_1,I_2,\cdot,t)]}{pI(\xi(I_1,I_2,\cdot,t))+[dI(\xi(I_1,I_2,\cdot,t)),A\xi(I_1,I_2,\cdot,t)]} \ .
\ee
\begin{proof}[Proof of Theorem 1.2]
We have from property $(d)$ of $I(\cdot)$ and Lemma 2.2 that there exist positive constants $c,C$ such that $c\le I_1(t), I_2(t)\le C$ for all $t\ge 0$.  Hence  from (\ref{G5}), (\ref{H5}) and the uniform lower bound for the denominator on the RHS of (\ref{G5}) established in Theorem 1.1, there exists a constant $C_1$ such that
\be \label{I5}
0 \ \le \ \beta(I_1(t),I_2(t),t) \ \le \ C_1e^{-t/p} \ , \quad t\ge 0 \ .
\ee
Setting $J(t)=I_1(t)-I_2(t)$ we have from (\ref{D5}) that
\be \label{J5}
\frac{dJ(t)}{dt}+m(t)J(t) \ = \ g(t) \ ,  \quad m(t)\ge \frac{1}{p} \ , \ \ |g(t)|\le C_1e^{-t/p} \ .
\ee
The solution to (\ref{J5}) is given by
\be \label{K5}
J(t) \ = \ J(0)\exp\left[-\int_0^tm(s) \ ds\right]+\int_0^t dt'  \ g(t')\exp\left[-\int_{t'}^tm(s) \ ds\right] \ .
\ee
We conclude from (\ref{J5}), (\ref{K5}) that
\be \label{L5}
|J(t)| \ \le \ |J(0)|e^{-t/p}+C_1te^{-t/p} \ .
\ee
The result follows using (\ref{E5}) from the uniform bounds on $I_1(\cdot),I_2(\cdot)$ and (\ref{L5}).  
\end{proof}

\vspace{.1in}

\section{Differential Delay Equations and Volterra Integral Equations}
In this section we prove some results for linear Volterra integral equations and their corresponding DDEs. In $\S8$ we  generalize parts of the argument used in these proofs to obtain results for the non-linear DDE  (\ref{B4}).  Our first consideration is the non-translation invariant Volterra equation
\be \label{A6}
u(t)+\int_0^tK(t,s)u(s) \ ds \ = \ g(t) \ , \quad t\ge 0 \ .
\ee
The solution to this equation can be formally written as
\be \label{B6}
u(t) \ = \ g(t)-\int_0^tr(t,s) g(s) \ ds \ , \quad t\ge 0,
\ee
where $r(\cdot,\cdot)$ is the {\it resolvent kernel} for the Volterra equation (see Chapter 9 of \cite{grip}).  We first summarize the proof of a beautiful result of Gripenberg (Theorem 9.1 of Chapter 9 of \cite{grip} and  Theorem 5 of \cite{grip80}), which illustrates a close relationship between methods for estimating solutions of Volterra integral equations and solutions of DDEs:
\begin{proposition}
Assume the kernel $K(\cdot,\cdot)$ for (\ref{A6}) is continuous non-negative and bounded, the functions $t\ra K(t,s)$ are decreasing on $[s,\infty)$ for all $s\ge 0$, the function $w(t)=\int_0^tK(t,s) \ ds$ converges as $t\ra\infty$, and that
\be \label{C6}
\lim_{T\ra\infty}\sup_{t\ge0}\int_0^{\max\{t-T,0\}} K(t,s)  \ ds \ < \ 1 \ .
\ee
 Then 
\be \label{D6}
\sup_{t\ge 0}\int_0^t |r(t,s)| \ ds \ < \ \infty \ .
\ee
\end{proposition}
\begin{proof}
It will be sufficient to show there exists a constant $C$ such that for all $g\in L^\infty(\R^+)$ the solution $u(\cdot)$ of (\ref{A6}) satisfies  $\|u\|_\infty\le C\|g\|_\infty$.
To do this we observe from (\ref{C6}) that there exists  $\ga_0, \ 0<\ga_0<1,$ and $T_0>0$ such that
\be \label{E6}
\int_0^{\max\{t-T_0,0\}} K(t,s)  \ ds \ \le \ga_0 < \ 1 \quad {\rm for \ all \ } t\ge 0 \ .
\ee 
Suppose now that $\tau>0$ is such that $|u(\tau)|=\sup_{0<t<\tau}|u(t)|$.  From  (\ref{A6}) we have that
\be \label{F6}
u(\tau)+ \int^\tau_{\max\{\tau-T_0,0\}} K(\tau,s)u(s)  \ ds  \ = \ -\int_0^{\max\{\tau-T_0,0\}} K(\tau,s) u(s)  \ ds+g(\tau) \ .
\ee
In the case when the function $u(\cdot)$ does not change sign in the interval $[\max\{\tau-T_0,0\},\tau]$  we can conclude from (\ref{F6}) that 
$|u(\tau)|\le |g(\tau)|/(1-\ga_0)$.  Alternatively, there exists $\tau_0$ in the interval $[\max\{\tau-T_0,0\},\tau]$ such that $u(\tau_0)=0$. This implies that $|u(\tau)|$ is bounded, with the bound depending only on the maximum of $|g(\cdot)|$ in the interval $[\max\{\tau-T_0,0\},\tau]$. To see this we differentiate (\ref{A6}) to obtain the DDE
\be \label{G6}
\frac{du(t)}{dt}+ K(t,t)u(t)+\int_0^t\frac{\pa K(t,s)}{\pa t}u(s) \ ds \ = \ g'(t) \ \ .
\ee
Integrating (\ref{G6}) for $t>\tau_0$ with initial condition $u(\tau_0)=0$ we obtain the identity
\be \label{H6}
u(\tau) \ = \ \int_{\tau_0}^\tau dt \  \exp\left[-\int_{t}^\tau K(s,s) \ ds\right]\left\{g'(t)-\int_0^t\frac{\pa K(t,s)}{\pa t}u(s) \ ds\right\} \ .
\ee
We see upon integration by parts that the integral involving $g'(\cdot)$ on the RHS of (\ref{H6}) is the same as
\be \label{I6}
g(\tau)-\exp\left[-\int_{\tau_0}^\tau K(s,s) \ ds\right] g(\tau_0)-\int_{\tau_0}^\tau dt \  \exp\left[-\int_{t}^\tau K(s,s) \ ds\right] \ K(t,t)g(t) \ ,
\ee
which is bounded in absolute value by $3\|g(\cdot)\|_\infty$. The second term on the RHS of (\ref{H6})  is bounded in absolute value by
\be \label{J6}
|u(\tau)|\int_{\tau_0}^\tau dt \  \exp\left[-\int_{t}^\tau K(s,s) \ ds\right]\int_0^t\left|\frac{\pa K(t,s)}{\pa t}\right| \ ds \ .
\ee
If we can show that the coefficient of $|u(\tau)|$ in (\ref{J6}) is less than $1$ then (\ref{H6}) implies that $|u(\tau)|$ is bounded by a constant times $\|g(\cdot)\|_\infty$. Observe now that
\be \label{K6}
\left(-\frac{d}{dt}\right) \left[\int_t^\tau K(s,s) \ ds+ w(t)\right] \ = \ \int_0^t\left|\frac{\pa K(t,s)}{\pa t}\right| \ ds \ ,
\ee
and consequently that
\begin{multline} \label{L6}
\int_{\tau_0}^\tau dt \  \exp \left[-\int_t^\tau K(s,s) \ ds- w(t)\right] \int_0^t\left|\frac{\pa K(t,s)}{\pa t}\right| \ ds \\
= \ \exp \left[- w(\tau)\right]-\exp \left[-\int_{\tau_0}^\tau K(s,s) \ ds- w(\tau_0)\right]  .
\end{multline}
It follows from (\ref{L6}) that
\begin{multline} \label{M6}
\int_{\tau_0}^\tau dt \  \exp\left[-\int_{t}^\tau K(s,s) \ ds\right]\int_0^t\left|\frac{\pa K(t,s)}{\pa t}\right| \ ds \ \le \\
 \exp \left[\sup_{\tau_0<t<\tau}w(t)- w(\tau_0)\right]\left\{\exp \left[w(\tau_0)- w(\tau)\right]-\exp \left[-\int_{\tau_0}^\tau K(s,s) \ ds\right]\right\} \ .
\end{multline}
Since $\lim_{t\ra\infty}w(t)$ exists, the RHS of (\ref{M6}) is strictly less than $1$ for $\tau$ sufficiently large.
\end{proof}
\begin{rem}
In the translation invariant case $K(t,s)=K(t-s)$, Proposition 6.1 implies that $r(t,s)=r(t-s)$ and the function $r(\cdot)$ is integrable on $
R^+$. This is the result for Volterra equations which we used in $\S3,\S4$ to prove local asymptotic stability of  solutions to (\ref{A1}). 
\end{rem}

Next we consider a class of linear DDEs of the form
\be \label{N6}
\frac{dI(t)}{dt}+a(t)I(t)+\int_0^tk(t,s)[I(t)-I(s)] \ ds \ = \ f(t) \ ,  \ \ t>0 \ .
\ee
Setting $u(t)=dI(t)/dt, \ t\ge0,$ we see from (\ref{N6}) that $u(\cdot)$ satisfies the Volterra equation (\ref{A6}) with kernel $K$ and $g$ given by
\be  \label{O6}
K(t,s) \ = \ a(t)+\int_0^s k(t,s') \ ds' \ , \quad g(t) \ = \ f(t)-a(t)I(0) \ .
\ee
The DDE (\ref{H4}) is of the form (\ref{N6}) with the functions $a(\cdot), \ k(\cdot,\cdot)$ given by
\be \label{P6}
a(t) \ =  \ K(t) \ , \quad k(t,s) \ = \ - K'(t-s) \ .
\ee
Then (\ref{O6}) gives $K(t,s)=K(t-s)$.  In this case we may conclude the following from the translation invariant results of Chapter II of \cite{grip}, or Proposition 6.1 applied to the translation invariant situation:  When the function $f(\cdot)$ on the RHS of (\ref{N6}) is integrable and the function $t\ra K(t)$ is non-negative decreasing and integrable then $u(\cdot)$ is also integrable. In particular, $I(t)$ converges as $t\ra\infty$. 

We shall obtain conditions on the functions $a(\cdot), \ k(\cdot,\cdot)$, in the non-translation invariant case, for the existence of $\lim_{t\ra\infty}I(t)$ when $f(\cdot)$ is integrable. Our methods resemble those of Yorke \cite{yorke} (see also Chapter $5$, section $5$ of \cite{hale}),  which are used to prove asymptotic stability of solutions to a {\it non-linear} DDE.  In the translation invariant case one has $\lim_{t\ra\infty}I(t)=0$ when $a(\cdot)\equiv a>0$ and $k(t,s)=k(t-s)$ is non-negative and integrable (see \cite{app} for this and related results). 
\begin{lem}
Assume  the function $k(\cdot,\cdot)$ of (\ref{N6}) is non-negative, the functions $s\ra k(t,s), \ 0<s<t,$ are integrable for all $t>0$, and
\be \label{Q6}
\sup_{s>0}\int_s^\infty k(t,s) \ dt \ < \ \infty \ .
\ee
Assume further that the function $a^-(\cdot)=-\min[a(\cdot),0]$ is in  $L^1(\R^+)$. Then if $f(\cdot)\in L^1(\R^+)$  the solution $I(\cdot)$ of (\ref{N6})  with initial condition $I(0)$ is bounded and $\|I(\cdot)\|_\infty\le C_1|I(0)|+C_2\|f(\cdot)\|_{L^1(\R^+)}$ for some constants $C_1,C_2$.  
\end{lem}
\begin{proof}
We define the function $I_{\rm max}(t)=\sup_{0<s<t}I(s)$ and  set $T^+_0=\sup\{t>0: \ I(s)\le 0, \ 0\le s\le t\}$. If $T^+_0=\infty$ then $I(\cdot)$ is non-positive. We consider the situation $T^+_0<\infty$, in which case   $I(T^+_0)=0$ if $T^+_0>0$. Let  $T>T^+_0$ be such that $I(T)=\sup_{0<s<T} I(s)$.  Then we have that
\be \label{R6}
I(T)-I(T^+_0) \ = \ \int_{(T^+_0,T)-\{T^+_0<t<T: I_{\rm max}(t)>I(t)\}} \frac{dI(t)}{dt} \ dt \ .
\ee
It follows  from (\ref{N6}), (\ref{R6}) that
\be \label{S6}
\sup_{0<t<T}I(t)-I(T^+_0)-\sup_{0<t<T}|I(t)|\int_{T^+_0}^T a^-(t) \  dt  \ \le \int_{T^+_0}^T |f(t)|  \ dt \ .
\ee
Evidently (\ref{S6}) holds for all $T\ge T^+_0$. 
We can make a similar argument to estimate $\min I(\cdot)$. Thus we set $T^-_0=\sup\{t>0: \ I(s)\ge 0, \ 0\le s\le t\}$. Similarly we have that
\be \label{T6}
\inf_{0<t<T}I(t)-I(T^-_0)+\sup_{0<t<T}|I(t)|\int_{T^-_0}^T a^-(t) \  dt  \ \ge -\int_{T^-_0} ^T |f(t)|  \ dt \ .
\ee
We conclude from (\ref{S6}), (\ref{T6}) that
\be \label{U6}
\sup_{0<t<T}|I(t)| \ \le \ |I(0)|+\sup_{0<t<T}|I(t)|\int_0^T a^-(t) \  dt  +\int_0^T |f(t)|  \ dt \ .
\ee
From (\ref{U6}) we obtain a bound on $\sup_{0<t<T_1}|I(t)|$ for any $T_1$ satisfying $\int_0^{T_1}a^-(t)  \ dt<1$.  We can extend the bound on $I(t)$ beyond $t\le T_1$ by rewriting the integral in (\ref{N6}) as an integral over the interval $(T_1,t)$ instead of $(0,t)$,  and replacing $a(\cdot), \ f(\cdot)$ by $a_1(\cdot), \ f_1(\cdot)$, where 
\be   \label{V6}
a_1(t)= a(t)+ \int_0^{T_1}k(t,s) \ ds, \quad f_1(t) = f(t)+\int_0^{T_1}k(t,s)I(s) \ ds \ .
\ee
Since $k(\cdot,\cdot)$ is non-negative one has  $a^-_1(\cdot)\le a^-(\cdot)$.  Furthermore, it follows from (\ref{Q6}) that $f_1(\cdot)$ is integrable.  Hence by repeating the previous argument, we obtain a bound on $\sup_{0<t<T_2}|I(t)|$ for any $T_2$ satisfying 
$\int_0^{T_2}a^-(t)  \ dt<2$. On continuing this process, we obtain an increasing sequence of times $T_1,T_2,...$ where we require $\int_0^{T_n}a^-(t)  \ dt<n, \ n=1,2,..$. Evidently since $a^-(\cdot)$ is integrable we can choose $T_n=\infty$ for some finite $n$, whence we obtain an upper bound on $\sup_{t>0}|I(t)|$. 
\end{proof} 
\begin{proposition}
Assume that $k(\cdot,\cdot)$ of (\ref{N6}) satisfies the conditions of Lemma 6.1, the function $b(t)=\int_0^t k(t,s) \ ds, \ t\ge 0,$ is bounded,  and in addition that
\be \label{W6}
\lim_{\ga\ra \infty}\sup_{T>0}\int_0^{\max\{T-\ga,0\}}ds\int_T^\infty k(t,s) \ dt \ = \ 0 \ .
\ee
Assume further that the function $a(\cdot)$ is integrable on $\R^+$.   Then  the solution $I(\cdot)$ of (\ref{N6}) converges at large time i.e.  $\lim_{t\ra\infty}I(t)=I_\infty$.  
\end{proposition}
\begin{proof}
We first assume  that for any $\ve,\ga>0,\ \tau>\ga,$ there exists $T_{\ve,\ga,\tau}>\tau$ such that $|I(t)-I(s)|<\ve$ for $t,s\in[T_{\ve,\ga,\tau}-\ga,T_{\ve,\ga,\tau}]$. Then we can argue as in Lemma 6.1 to conclude the result.  Thus for $t>T_{\ve,\ga,\tau}$ we set 
$I_{\rm max}(t)=\sup_{T_{\ve,\ga,\tau}<s<t} I(s)$ and consider $T> T_{\ve,\ga,\tau}$ such that $I(T)= I_{\rm max}(T)$.  Integration of (\ref{N6}), using the identity (\ref{R6}),  yields the inequality 
\begin{multline} \label{X6}
I(T)-I(T_{\ve,\ga,\tau}) \ \le \int_{T_{\ve,\ga,\tau}}^\infty |f(t)|+|a(t)|\|I\|_\infty \ dt \\
+\int_{T_{\ve,\ga,\tau}}^T dt \ \left\{\ve\int_{T_{\ve,\ga,\tau}-\ga }^{T_{\ve,\ga,\tau}} k(t,s) \ ds+2\|I(\cdot)\|_\infty\int_0^{T_{\ve,\ga,\tau}-\ga} k(t,s) \ ds \ \right\} \ .
\end{multline}
We have now that
\be \label{Y6}
\int_{T_{\ve,\ga,\tau}}^T dt \ \int_{T_{\ve,\ga,\tau}-\ga }^{T_{\ve,\ga,\tau}} k(t,s) \ ds \ \le \ C\ga \ ,
\ee
where $C$ is an upper bound for the LHS of (\ref{Q6}). 
It follows from (\ref{W6})-(\ref{Y6}) that for any $\del>0$ there exists $T_\del>0$ such that $\sup_{t>T_\del}[I(t)-I(T_\del)]<\del$.  Since we can obtain an analogous estimate for the infimum, we conclude that  $\lim_{t\ra\infty} I(t)=I_\infty$ exists.

Alternatively, there exists  $\ve_0,\ga_0>0, \ \tau_0>\ga_0$ such that $\sup_{s,t\in[T-\ga_0,T]}|I(t)-I(s)| \ge \ve_0$ for all $T\ge \tau_0$.   We integrate (\ref{N6}) to obtain for $0<T_1<T_2$ the formula
\begin{multline} \label{Z6}
I(T_2) \ = \ I(T_1)\exp\left[-\int_{T_1}^{T_2} b(t) \ dt \ \right]+\\
\int_{T_1}^{T_2} \exp\left[-\int_t^{T_2} b(s) \ ds \ \right]\left\{f(t)-a(t)I(t)+\int_0^t k(t,s)I(s) \ ds\right\} \ dt \ .
\end{multline}
Letting $I^+_\infty=\limsup_{t\ra\infty}I(t)$, there exists  for any $\del>0, \ N=1,2,..,$ a  time $T_{\del,N}>\max[\tau_0,N]$  such that $I(T_{\del,N})\ge I_\infty^+-\del$ and $I(t)\le I_\infty^++\del$ for $T_{\del,N}-N\le t\le T_{\del,N}$.  Since the oscillation of $I(\cdot)$ in the interval $[T_{\del,N}-\ga_0,T_{\del,N}]$ exceeds $\ve_0$,  there exists $\tau_{\del,N}\in[T_{\del,N}-\ga_0,T_{\del,N}]$ such that $I(\tau_{\del,N})\le I_\infty^++\del-\ve_0$.  We set $T_1=\tau_{\del,N}, \ T_2=T_{\del,N}$ in (\ref{Z6}) and conclude that
\begin{multline} \label{AA6}
I_\infty^+-\del \ \le \ [I^+_\infty+\del-\ve_0]\exp\left[-\int_{\tau_{\del,N}}^{T_{\del,N}} b(t) \ dt \ \right] \\
+\left\{1-   \exp\left[-\int_{\tau_{\del,N}}^{T_{\del,N}} b(t) \ dt \ \right]  \right\}(I^+_\infty+\del) \\
+ \int_{\tau_{\del,N}}^\infty |f(t)| +\|I(\cdot)\|_\infty\left\{|a(t)|+\int^{T_{\del,N}-N}_0 k(t,s) \ ds\right\}\ dt \ .
\end{multline}
Using (\ref{W6}) and the boundedness of the function $b(\cdot)$, we obtain a contradiction from (\ref{AA6}) by choosing $\del$ sufficiently small and $N$ sufficiently large.
\end{proof}

\vspace{.1in}

\section{Some Optimal Control Problems}
In this section we establish some key properties of the function $f(t,v_t(\cdot))$ defined by (\ref{AI4}), which will enable us to obtain global asymptotic stability results for the DDE (\ref{AH4}).  We shall accomplish this by obtaining global properties of the functional $F(t,y,v_t(\cdot))$ defined by (\ref{Y4}). We have already seen from (\ref{AF4})  that the gradient of $F(t,y,v_t(\cdot))$ with respect to $v_t(\cdot)$ at $v_t(\cdot)\equiv 1$ is non-negative modulo exponentially small terms, provided $h(\cdot)$ satisfies the conditions of Lemma 3.1.  We shall prove the following global result:
\begin{theorem}
Let $h:(0,\infty)\ra\R$ be a $C^2$ non-negative decreasing and convex function such that $yh''(y)+h'(y)\ge 0, \ y>0,$ and the function $y\ra y^2h''(y)$  decreases. Then the function $F(t,y,v_t(\cdot))$ of (\ref{Y4}) has the property that  the maximum of $F(t,y,v_t(\cdot))$ on the set $0<v_t(\cdot) \le 1$ occurs at $v_t(\cdot)\equiv 1$, and  the minimum of $F(t,y,v_t(\cdot))$ on the set $v_t(\cdot)\ge 1$ also occurs at $v_t(\cdot)\equiv 1$.
\end{theorem}
\begin{rem}
We compare the conditions on $h(\cdot)$ in Theorem 7.1 to the conditions on $h(\cdot)$ in Lemma 3.1.  Parallel to (\ref{N3})  we have that
\be \label{DA7}
h(y) \ = \ h_\infty+\int_y^\infty \frac{k(y')}{y'} \ dy' \ , \quad y>0,
\ee
where $k(\cdot)$ is assumed $C^1$ non-negative decreasing. This implies $h(\cdot)$ is $C^2$ non-negative decreasing convex and satisfies $yh''(y)+h'(y)\ge 0, \ y>0$.  Note that (\ref{DA7}) is the $p\ra \infty$ limit of (\ref{N3}). To ensure that the the function $y\ra y^2h''(y)$  decreases we require that the function $k(\cdot)$ also be convex. 
\end{rem}
We carry this out by obtaining the solution to some  optimal control problems \cite{fr}.  Let $y>0,T\in\R$, and consider the linear dynamics 
\be \label{A7}
\frac{dx(t)}{dt} \ = \ -\frac{1}{p}x(t)-v(t) \ ,  \ t<T, \quad x(T) \ = \ y  ,
\ee
with terminal condition $x(T)=y$ and controller $v(\cdot)$. The solution to (\ref{A7}) is evidently given by
\be \label{B7}
x(t) \ = \ e^{(T-t)/p}y+\int_t^Tds \ e^{(s-t)/p} \  v(s) \ .
\ee

Let $g:(0,\infty)\ra\R^+$ be a positive decreasing function and for $y>0, \ t<T,$ define the function
\be \label{C7}
q(x,y,t,T) \ =  \  \max_{0< v(\cdot)\le 1}\left[ \int_t^T ds \  g\left(   \frac{x(s)}{v(s)} \right)  \ \Bigg| \ x(t)=x \ \right] \ ,
\ee
where $x(\cdot)$ satisfies (\ref{A7}) and $(x,t)$ belongs to the {\it reachable set} of the control system. Thus $q(x,y,t,T)$ is defined only for $x$ satisfying 
\be \label{L7}
e^{(T-t)/p}y \ < \ x \ < e^{(T-t)/p}[y+p]-p  \ .
\ee
Letting
\be \label{D7}
q(y,v(\cdot),t,T) \ = \ \int_t^T ds \  g\left(   \frac{x(s)}{v(s)}\right) \ ,
\ee
we have that the gradient $dq$ of $q$ with respect to $v(\cdot)$ is given by
\begin{multline} \label{E7}
dq(y,v(\cdot),t,T;\tau) \ = \ -\frac{x(\tau)}{v(\tau)^2}g'\left(   \frac{x(\tau)}{v(\tau)}\right)+ \\
\int_t^\tau ds \  g'\left(   \frac{x(s)}{v(s)}\right) \frac{1}{v(s)} e^{(\tau-s)/p} \  , \quad t<\tau<T.
\end{multline}
Setting $v(\cdot)\equiv 1$ in (\ref{E7}) then
\be \label{F7}
dq(y,1(\cdot),t,T;\tau) \ = \ -x_p(\tau)g'(x_p(\tau))+ 
\int_t^\tau ds \  g'(x_p(s)) e^{(\tau-s)/p} \  , 
\ee
where $x_p(s)=e^{(T-s)/p}(y+p)-p$.
We have now that
\begin{multline} \label{G7}
\int_t^\tau ds \  g'(x_p(s)) e^{(\tau-s)/p}  \ = \ \frac{e^{-(T-\tau)/p}}{1+y/p}\int_t^\tau ds \  \left(-\frac{d}{ds}\right) g(x_p(s))  \\
= \ \frac{g(x_p(t))-g(x_p(\tau))}{1+x_p(\tau)/p} \ .
\end{multline}
From (\ref{F7}), (\ref{G7}) we have then
\be \label{H7}
dq(y,1(\cdot),t,T;\tau) \ = \ \frac{-x_p(\tau)\left[1+x_p(\tau)/p\right]g'(x_p(\tau))-g(x_p(\tau))+g(x_p(t))}{1+x_p(\tau)/p} \ .
\ee
It follows from (\ref{H7}) that $dq$ is non-negative at $v(\cdot)\equiv1$ provided 
\be \label{I7}
x[1+x/p]g'(x)+g(x)-g(\infty) \ \le \  0 \quad  {\rm for \  all \ } x>0 \ .
\ee
Thus we have that
\be \label{J7}
g(x) \ = \ g(\infty)+\left(\frac{1}{p}+\frac{1}{x}\right)k(x) \ , \quad {\rm where \ } k'(\cdot)\le 0, \ \lim_{x\ra\infty}k(x)=0 \ .
\ee

We have shown that if $g(\cdot)$ is non-negative decreasing and satisfies (\ref{J7}) then $v(\cdot)\equiv1$ is a {\it local} maximum of the functional (\ref{D7}) on the set $\{0<v(\cdot)\le 1\}$.  We shall show under somewhat stronger conditions on $g(\cdot)$, corresponding to taking $p=\infty$ in (\ref{I7}),  that it is also a {\it global} maximum. We do this by obtaining the solution to the Hamilton-Jacobi (HJ) equation for (\ref{C7}), which is given by
\begin{multline} \label{K7}
\frac{\pa q(x,y,t,T)}{\pa t}-\frac{x}{p} \frac{\pa q(x,y,t,T)}{\pa x} \\
+\sup_{0<v<1}\left[ -v\frac{\pa q(x,y,t,T)}{\pa x}   +g\left(\frac{x}{v}\right)\right] \ = \ 0 \ .
\end{multline} 
\begin{proposition}
Assume $g(\cdot)$ is continuous non-negative decreasing, and  also that the function $x\ra x[g(x)-g(\infty)], \ x>0,$ is decreasing. For any $y>0, \ t<T,$ let $(x,t)$ satisfy (\ref{L7}) and $\tau_{x,t}$ be defined by the equation
\be \label{M7}
\exp\left[\frac{\tau_{x,t}}{p}\right] \ = \ e^{T/p}\left(1+\frac{y}{p}\right)-\frac{x}{p}e^{t/p}  \ .
\ee
Then $t<\tau_{x,t}<T$ and the function $q(x,y,t,T)$ of (\ref{C7}) is given by the formula,
\be \label{N7}
q(x,y,t,T) \ = \ [\tau_{x,t}-t]g(\infty)+\int_{\tau_{x,t}}^T ds  \ g(x_p(s)) \ ds \ .
 \ee
\end{proposition}
\begin{proof}
As a possible solution to (\ref{K7}) we consider  trajectories $x(\cdot)$ for (\ref{A7}) starting at $x$ at time $t<T$, with $(x,t)$ satisfying (\ref{L7}). The control is set with $v=0$ until the trajectory hits the curve $x_p(\cdot)$, and then the control is set at $v=1$, so the trajectory continues along $x_p(\cdot)$ until it reaches $y$ at time $T$.  This is a so called {\it bang-bang} control mechanism, which  occurs quite often \cite{fr} in solutions to control problems where the controls are confined to a bounded convex set.  Let $\tau_{x,t}$ be the time the curve $x(\cdot)$ reaches $x_p(\cdot)$.  We have that
\be \label{O7}
e^{(t-\tau)/p}x \ = \ e^{(T-\tau)/p}(y+p)-p \ , \quad \tau \ = \ \tau_{x,t} \ ,
\ee
whence $\tau_{x,t}$ is given by (\ref{M7}). The function $q$ of (\ref{N7})  satisfies the PDE
 \be \label{P7}
 \frac{\pa q(x,y,t,T)}{\pa t}-\frac{x}{p} \frac{\pa q(x,y,t,T)}{\pa x} +g(\infty) \ = \ 0 \ .
 \ee
To see this observe that
\begin{eqnarray} \label{Q7}
\frac{\pa q(x,y,t,T)}{\pa x} \ &=& \ [g(\infty)-g(x_p(\tau_{x,t}))] \frac{\pa \tau_{x,t}}{\pa x} \ , \\
\frac{\pa q(x,y,t,T)}{\pa t} \ &=& \ [g(\infty)-g(x_p(\tau_{x,t}))] \frac{\pa \tau_{x,t}}{\pa t}-g(\infty) \ . \nonumber
\end{eqnarray}
From (\ref{M7}) we see that
\be \label{R7}
 \frac{\pa \tau_{x,t}}{\pa t} \ = \ \frac{x}{p} \frac{\pa \tau_{x,t}}{\pa x} \ ,
\ee
whence (\ref{P7}) follows from (\ref{Q7}), (\ref{R7}). Since  $\pa \tau_{x,t}/\pa x<0$ we see also that $\pa q/\pa x\ge0$.  We conclude from (\ref{P7}) that the function $q$ defined by (\ref{N7}) is a solution of the HJ equation (\ref{K7}) provided
\be \label{S7}
g\left(\frac{x}{v}\right)-g(\infty) \ \le \ v\frac{\pa q(x,y,t,T)}{\pa x}  \quad {\rm for \ } 0<v\le1 \ .
\ee

To prove (\ref{S7}) we first note from (\ref{M7}) that 
\be \label{T7}
 \frac{\pa \tau_{x,t}}{\pa x} \ = \ -\frac{pe^{t/p}}{e^{T/p}(y+p)-xe^{t/p}} \ .
\ee
We see from (\ref{T7}) that as $x$ approaches $x_p(t)$ then $\pa \tau_{x,t}/\pa x$ approaches $-1$.  The condition (\ref{S7}) on $x_p(\cdot)$ becomes then
\be \label{U7}
g\left(\frac{x}{v}\right)-g(\infty) \ \le \  v[g(x)-g(\infty)] \ , \quad x\in \{x_p(s) \ : \ s\le T\} \ , \ 0<v<1.
\ee
Evidently (\ref{U7}) holds provided the function $x\ra x[g(x)-g(\infty)]$ is positive decreasing.  In that case the condition (\ref{S7}) becomes
\be \label{V7}
x[g(x)-g(\infty)] \ \le \ x\frac{\pa q(x,y,t,T)}{\pa x}  \ .
\ee 
We show (\ref{V7}) holds by using the maximum principle. Observe that $w(x,t)=x\pa q(x,y,t,T)/\pa x$ is a solution to the PDE
\be \label{W7}
 \frac{\pa w(x,t)}{\pa t}-\frac{x}{p} \frac{\pa w(x,t)}{\pa x} \ = \ 0 \ .
\ee
Setting $u(x,t)=w(x,t)-x[g(x)-g(\infty)]$ we have from (\ref{W7}) that
\be \label{X7}
 \frac{\pa u(x,t)}{\pa t}-\frac{x}{p} \frac{\pa u(x,t)}{\pa x} \ \le \ 0 \ .
\ee
We have already shown that $u(x,t)= 0$ for $x\in \{x_\infty(s) \ : \ s\le T\}$. Hence by the method of characteristics $u(x,t)\ge 0$ for all $(x,t)$ satisfying $t<T$ and (\ref{L7}).

We have shown that the function (\ref{N7}) is a $C^1$ solution to the HJ equation (\ref{K7}) for the variational problem (\ref{C7}). It follows now from the usual {\it verification theorem} method \cite{fr} that (\ref{N7}), together with its corresponding bang-bang control settings, solves the variational problem. Thus let $x(\cdot)$ be a solution to (\ref{A7}) with controller $v(\cdot)$ satisfying $0<v(\cdot)\le 1$. Then from (\ref{K7})  the function $q(x,y,t,T)$  defined by (\ref{N7}) satisfies
\begin{multline} \label{Y7}
q(x(t),y,t,T) \ = \ -\int_t^T ds \ \frac{d}{ds} q(x(s),y,s,T) \\
 = \ -\int_t^T ds \ \left[\frac{\pa q(x(s),y,s,T)}{\pa t}-\frac{\pa q(x(s),y,s,T)}{\pa x}\left\{\frac{x(s)}{p}+v(s)\right\} \ \right] \\
 \ge  \ \int_t^T ds \  g\left(   \frac{x(s)}{v(s)}\right) \ = \ q(y,v(\cdot),t,T) \ .
\end{multline}  
\end{proof} 
\begin{rem}
Since the function (\ref{N7}) satisfies $\pa q(x,y,t,T)/\pa x\ge 0$, it follows from Proposition 7.1  that the solution of the variational problem  $\max_{0<v(\cdot)\le 1} q(y,v(\cdot),t,T)$ is given by $v(\cdot)\equiv 1$. 
\end{rem}

Next we consider the variational problem analogous to (\ref{C7}) given by 
\be \label{Z7}
q(x,y,t,T) \ =  \  \min_{1\le v(\cdot)<\infty}\left[ \int_t^T ds \  g\left(   \frac{x(s)}{v(s)} \right)  \ \Bigg| \ x(t)=x \ \right] \ ,
\ee
where $x(\cdot)$ satisfies (\ref{A7}) and $(x,t)$ belongs to the reachable set of the control system. Thus $q(x,y,t,T)$ is defined only for $(x,t)$ satisfying 
\be \label{AA7}
e^{(T-t)/p}[y+p]-p \ < \ x \ < \ \infty \ . 
\ee
The Hamilton-Jacobi (HJ) equation for (\ref{Z7})  is given by
\begin{multline} \label{AB7}
\frac{\pa q(x,y,t,T)}{\pa t}-\frac{x}{p} \frac{\pa q(x,y,t,T)}{\pa x} \\
+\inf_{1<v<\infty}\left[ -v\frac{\pa q(x,y,t,T)}{\pa x}   +g\left(\frac{x}{v}\right)\right] \ = \ 0 \ .
\end{multline} 
Note that there are important differences between the HJ equations (\ref{K7}) and (\ref{AB7}). We can write both of them in the form
\be \label{AC7}
\frac{\pa q(x,y,t,T)}{\pa t}-\frac{x}{p} \frac{\pa q(x,y,t,T)}{\pa x} +G\left(x,\frac{\pa q(x,y,t,T)}{\pa x}\right) \  = \ 0,
\ee 
for some function $G(x,\xi)$. In the case of (\ref{K7}) the function $G(x,\xi)$ is {\it convex} in $\xi$, whereas it is {\it concave} in $\xi$ for (\ref{AB7}).  

There are also important differences in the optimal control settings for the variational problems (\ref{C7}) and (\ref{Z7}). We have shown in Proposition 7.1 that the optimum for (\ref{C7}) is given by bang-bang control. For (\ref{Z7}) this is not the case.  To see why we consider the function $G(x,\xi,v)$ defined by
\be \label{AD7}
G(x,\xi,v) \ = \ -v\xi+g\left(\frac{x}{v}\right) \ .
\ee 
We assume the function $v\ra G(x,\xi,v)$ is convex, which is the case provided the function $z\ra -z^2g'(z)$ decreases.  The maximum of $G(x,\xi,v)$ on the interval $0<v<1$ is attained at either $v=0$ or $v=1$, whence we expect the optimal control setting to be bang-bang in the case of (\ref{C7}). The minimum of $G(x,\xi,v)$ on the interval $1<v<\infty$ is attained at  $v=1$ if $\xi\le -xg'(x)$. If $\xi>-xg'(x)$ then the minimizer of $\min_{v\ge 1} G(x,\xi,v)$ is the solution to the equation
\be \label{AE7}
-\xi-\frac{x}{v^2}g'\left(\frac{x}{v}\right) \ =  \ 0 \ ,  \ \  {\rm whence}  \ -\frac{x^2}{v^2}g'\left(\frac{x}{v}\right) \ = \ x\xi \ = \zeta \ .
\ee
A solution to (\ref{AE7}) exists for all $\xi>-xg'(x)$ provided $\lim_{z\ra 0}z^2g'(z)=-\infty$. 
From (\ref{AE7}) it follows that the minimizing $v=v_{\rm min}(x,\xi)=xh(\zeta)$ for some function $h(\cdot)$. The corresponding HJ equation has therefore the form
\be \label{AF7}
\frac{\pa q(x,y,t,T)}{\pa t}+ H\left(x\frac{\pa q(x,y,t,T)}{\pa x}\right) \ = \ 0 \ ,
\ee
where
\be \label{AG7}
H(\zeta) \ = \ -\frac{\zeta}{p}-\zeta h(\zeta)+g\left(\frac{1}{h(\zeta)}\right) \ .
\ee
Note that $\zeta=x\pa q(x,y,t,T)/\pa x$ is constant along characteristics for the HJ equation (\ref{AF7}), whence  it follows from (\ref{AE7}) that $v(\cdot)/x(\cdot)$ is also constant along characteristics.

The considerations of the previous paragraph lead us to propose a solution to (\ref{AB7}). For $s<t<T$ let $x_p(s,t)$ be the solution to the terminal value problem
\be \label{AH7}
\frac{dx_p(s,t)}{ds} \ = \ -\left[\frac{1}{p}+\frac{1}{x_p(t)}\right]x_p(s,t) \ ,  \ s<t<T, \quad x_p(t,t) \ = \ x_p(t) \   .
\ee
Setting $x(s)=x_p(s,t)-x_p(s)$, we see from (\ref{A7}) with $v(\cdot)\equiv 1$ and (\ref{AH7}) that
\be \label{AI7}
\frac{dx(s)}{ds} \ = \ -\left[\frac{1}{p}+\frac{1}{x_p(t)}\right]x(s)-\frac{1}{x_p(t)}\left[x_p(s)-x_p(t)\right] \ ,  \ s<t<T, \quad x(t) \ = \ 0 \   .
\ee
Since the function $s\ra x_p(s)$ is decreasing, it follows from (\ref{AI7}) that $x(s)>0$ for $s<t$.  Hence the trajectory $x_p(s,t), \ s<t,$ lies in the reachable set (\ref{AA7}) for the variational problem (\ref{Z7}).  We can show similarly that the trajectories $x_p(\cdot,t), \ t<T$, do not intersect. Thus for $t_1<t_2<T$ let $x(s)=x_p(s,t_2)-x_p(s,t_1), \ s<t_1$. We have already seen that $x(t_1)>0$, and from (\ref{AH7}) we also have that
\be \label{AJ7}
\frac{dx(s)}{ds} \ = \ -\left[\frac{1}{p}+\frac{1}{x_p(t_1)}\right]x(s)-x_p(s,t_2)\left[\frac{1}{x_p(t_2)}-\frac{1}{x_p(t_1)}\right] \ ,  \ s<t_1 \ .
\ee
Since $x_p(t_1)>x_p(t_2)$ we conclude from (\ref{AJ7}) that $x_p(s,t_2)>x_p(s,t_1),\ s<t_1$.  Since the trajectories $x_p(\cdot,t), \ t<T,$ do not entirely cover the reachable set we complement them with a set of trajectories with terminal point $y$ at time $T$. Thus for $s<T,  0<\la<y$ we define $y_p(s,\la)$ as the solution to
\be \label{AK7}
\frac{dy_p(s,\la)}{ds} \ = \ -\left[\frac{1}{p}+\frac{1}{\la}\right]y_p(s,\la) \ ,  \ s<T, \quad y_p(T,\la) \ = \ y \   .
\ee
If $t<T$ and $x_p(t)<x<x_p(t,T)$ then there exists unique $\tau=\tau_{x,t}$ such that $t<\tau<T$ and $x_p(t,\tau)=x$.  If $x>x_p(t,T)$ then there exists unique $\la=\la_{x,t}$ such that $0<\la<y$ and $y_p(t,\la)=x$.  We define now a function $q(x,y,t,T)$ for $t<T$ and $(x,t)$ satisfying  (\ref{AA7}) by
\begin{multline} \label{AL7}
q(x,y,t,T) \ = \ (\tau_{x,t}-t)g(x_p(\tau_{x,t}))+\int_{\tau_{x,t}}^T g(x_p(s)) \ ds \quad {\rm if \ } x_p(t)<x<x_p(t,T) \ , \\
q(x,y,t,T) \ = \ (T-t)g(\la_{x,t}) \quad {\rm if \ }  x>x_p(t,T) \ .
\end{multline}
\begin{proposition}
Assume $g(\cdot)$ is $C^1$ non-negative decreasing, and  also that the function $z\ra -z^2g'(z), \ z>0,$ is decreasing. For any $y>0, \ t<T,$ let $(x,t)$ satisfy (\ref{AA7}). Then  the function $q(x,y,t,T)$ of (\ref{Z7}) is given by the formula (\ref{AL7}).
\end{proposition}
\begin{proof}
We first consider the case $x>x_p(t,T)$. The partial derivatives of $\la_{x,t}$ can be computed by using the formula
\be \label{AM7}
\exp\left[\left(\frac{1}{p}+\frac{1}{\la_{x,t}}\right)(T-t)\right]y \ = \ x \  .
\ee
Thus we have that
\be \label{AN7}
\frac{\pa \la_{x,t}}{\pa x} \ = \ -\frac{\la_{x,t}^2}{(T-t)x} \ , \quad \frac{\pa \la_{x,t}}{\pa t} \ = \ -\frac{\la_{x,t}^2}{(T-t)}\left(\frac{1}{p}+\frac{1}{\la_{x,t}}\right) \ .
\ee
It follows from (\ref{AL7}), (\ref{AN7}) that
\begin{multline} \label{AO7}
x\frac{\pa q(x,y,t,T)}{\pa x} \ = \ -g'(\la_{x,t})\la_{x,t}^2  \ , \\
\frac{\pa q(x,y,t,T)}{\pa t} \ = \ -g(\la_{x,t})-g'(\la_{x,t})\la_{x,t}^2 \left(\frac{1}{p}+\frac{1}{\la_{x,t}}\right) \ . 
\end{multline}
Hence  $q$ is a solution to the PDE
\begin{multline} \label{AP7}
\frac{\pa q(x,y,t,T)}{\pa t}-\frac{x}{p} \frac{\pa q(x,y,t,T)}{\pa x} -v(x,t)\frac{\pa q(x,y,t,T)}{\pa x} +g\left(\frac{x}{v(x,t)}\right) \ = \ 0 \ , \\
{\rm where \ \ } \frac{x}{v(x,t)} \ = \ \la_{x,t} \ .
\end{multline}
Note that $v(x,t)>1$ since $\la_{x,t}<y<x$.  We also have that
\be \label{AQ7}
\frac{\pa}{\pa v}\left[-v\frac{\pa q(x,y,t,T)}{\pa x}+g\left(\frac{x}{v}\right) \right] \ = \ 0 \quad {\rm at \ } v=v(x,t) \ .
\ee
Hence, in view of the convexity requirement on $g(\cdot)$, we conclude that $q(x,y,t,T)$ satisfies the HJ equation (\ref{AB7}) in the region $\{(x,t): \ t<T, \ x>x_p(t,T)\}$. 

Next we consider the region $\{(x,t): \ t<T, \ x_p(t)<x<x_p(t,T)\}$.  In that case we have
\be \label{AR7}
\exp\left[\left(\frac{1}{p}+\frac{1}{x_p(\tau_{x,t})}\right)(\tau_{x,t}-t)\right]x_p(\tau_{x,t}) \ = \ x \ .
\ee
Differentiating (\ref{AR7}) with respect to $x$ gives
\be \label{AS7}
\frac{\pa \tau_{x,t}}{\pa x} \ = \ \frac{px_p(\tau_{x,t})^2}{[x_p(\tau_{x,t})+p](\tau_{x,t}-t)x} \ .
\ee
Similarly we have that
\be \label{AT7}
\frac{\pa \tau_{x,t}}{\pa t} \ = \ \frac{x_p(\tau_{x,t})}{(\tau_{x,t}-t)} \ .
\ee
From (\ref{AL7}), (\ref{AS7}) we have that
\begin{multline} \label{AU7}
x\frac{\pa q(x,y,t,T)}{\pa x} \ = \ -(\tau_{x,t}-t)g'(x_p(\tau_{x,t}))\left[\frac{x_p(\tau_{x,t})}{p}+1\right]x\frac{\pa \tau_{x,t}}{\pa x} \\
= \ -x_p(\tau_{x,t})^2g'(x_p(\tau_{x,t})) \ ,
\end{multline}
and also from (\ref{AT7}) that 
\begin{multline} \label{AV7}
\frac{\pa q(x,y,t,T)}{\pa t} \ = \ -g(x_p(\tau_{x,t}))-(\tau_{x,t}-t)g'(x_p(\tau_{x,t}))\left[\frac{x_p(\tau_{x,t})}{p}+1\right]\frac{\pa \tau_{x,t}}{\pa t}  \\
= \ -g(x_p(\tau_{x,t}))-x_p(\tau_{x,t})g'(x_p(\tau_{x,t}))\left[\frac{x_p(\tau_{x,t})}{p}+1\right] \ .
\end{multline}
It follows from (\ref{AU7}), (\ref{AV7}) that  $q$ is a solution to the PDE
\begin{multline} \label{AW7}
\frac{\pa q(x,y,t,T)}{\pa t}-\frac{x}{p} \frac{\pa q(x,y,t,T)}{\pa x} -v(x,t)\frac{\pa q(x,y,t,T)}{\pa x} +g\left(\frac{x}{v(x,t)}\right) \ = \ 0 \ , \\
{\rm where \ \ } \frac{x}{v(x,t)} \ = \ x_p(\tau_{x,t}) \ .
\end{multline}
Note that since $x>x_p(\tau_{x,t})$ we have $v(x,t)>1$ in (\ref{AW7}).  Furthermore, the identity (\ref{AQ7}) also holds. We therefore conclude that $q$ is a solution  to the HJ equation (\ref{AB7}).  Since $q$ is a $C^1$ solution to the HJ equation for $(x,t)$ in the reachable set we can argue as in Proposition 7.1 to show that
the solution to the variational problem (\ref{Z7}) is given by (\ref{AL7}). 
\end{proof}
\begin{rem}
Since the function (\ref{AL7}) satisfies $\pa q(x,y,t,T)/\pa x\ge 0$, it follows from Proposition 7.2  that the solution of the variational problem  $\min_{1<v(\cdot)<\infty} q(y,v(\cdot),t,T)$ is given by $v(\cdot)\equiv 1$. 
\end{rem}
Let $g:(0,\infty)\ra\R^+$ be a non-negative decreasing function and for $y>0, \ t<T$ define the function $q(x,y,t,T)$ by 
\be \label{AX7}
q(x,y,t,T) \ =  \  \max_{0\le v(\cdot)\le 1}\left[ \int_t^T ds \  g\left(   \frac{x(s)}{v(s)} \right)e^{-(T-s)/p}v(s)  \ \Bigg| \ x(t)=x \ \right] \ ,
\ee
where $x(\cdot)$ satisfies (\ref{A7}) and $(x,t)$ belongs to the reachable set (\ref{L7}). 
Letting
\be \label{AY7}
q(y,v(\cdot),t,T) \ = \ \int_t^T ds \  g\left(   \frac{x(s)}{v(s)}\right)e^{-(T-s)/p}v(s) \ ,
\ee
we have that the gradient $dq$ of $q$ with respect to $v(\cdot)$ is given by
\begin{multline} \label{AZ7}
e^{(T-\tau)/p}dq(y,v(\cdot),t,T;\tau) \ = \ -\frac{x(\tau)}{v(\tau)}g'\left(   \frac{x(\tau)}{v(\tau)}\right)+g\left(   \frac{x(\tau)}{v(\tau)}\right) \\
 +\int_t^\tau ds \  g'\left(   \frac{x(s)}{v(s)}\right)   \ , \quad t<\tau<T.
\end{multline}
If $v(\cdot)\equiv 1$ then
\be \label{BA7}
e^{(T-\tau)/p}dq(y,1(\cdot),t,T;\tau) \ = \ -x_p(\tau)g'(x_p(\tau))+g(x_p(\tau))+\int_t^\tau ds \ g'(x_p(s)) \ .
\ee
We have now that
\begin{multline} \label{BB7}
\int_t^\tau ds \ g'(x_p(s)) \ = \ \frac{p}{y+p}\int_t^\tau ds \ e^{-(T-s)/p}\left(-\frac{d}{ds}\right)g(x_p(s)) \\
= \ \frac{1}{y+p}\int_t^\tau ds \ e^{-(T-s)/p}g(x_p(s))+\frac{p}{x_p(t)+p}g(x_p(t))-\frac{p}{x_p(\tau)+p}g(x_p(\tau)) \ .
\end{multline}
Since $g(\cdot)$ is non-negative decreasing, it follows from (\ref{BA7}), (\ref{BB7}) that $v(\cdot)\equiv 1$ is a local maximum for $q(y,v(\cdot),t,T) $ on the set $0<v(\cdot)<1$.

The HJ equation associated with (\ref{AX7}) is given by
\begin{multline} \label{BC7}
\frac{\pa q(x,y,t,T)}{\pa t}-\frac{x}{p} \frac{\pa q(x,y,t,T)}{\pa x} \\
+\sup_{0<v<1}\left[ -v\frac{\pa q(x,y,t,T)}{\pa x}   +e^{-(T-t)/p}g\left(\frac{x}{v}\right)v\right] \ = \ 0 \ .
\end{multline} 
We shall obtain the solution to the variational problem (\ref{AX7}) by producing a $C^1$ solution to the HJ equation (\ref{BC7}). Just as in Proposition 7.1 our solution is given by bang-bang control settings. 
\begin{proposition}
Assume $g(\cdot)$ is continuous non-negative decreasing. For any $y>0, \ t<T,$ let $(x,t)$ satisfy (\ref{L7}) and $\tau_{x,t}$ be defined by (\ref{M7}). 
Then $t<\tau_{x,t}<T$ and the function $q(x,y,t,T)$ of (\ref{AX7}) is given by the formula,
\be \label{BD7}
q(x,y,t,T) \ = \ \int_{\tau_{x,t}}^T ds  \ g(x_p(s))e^{-(T-s)/p} \ ds \ .
 \ee
\end{proposition}
\begin{proof}
We have from (\ref{BD7}) that
\be \label{BE7}
\frac{\pa q(x,y,t,T)}{\pa x} \ = \ -\exp\left[(T-\tau_{x,t})/p\right]g(x_p(\tau_{x,t})) \frac{\pa \tau_{x,t}}{\pa x} \ , 
\ee
and similarly that
\be \label{BF7}
\frac{\pa q(x,y,t,T)}{\pa t} \ = \ -\exp\left[(T-\tau_{x,t})/p\right]g(x_p(\tau_{x,t})) \frac{\pa \tau_{x,t}}{\pa t} \ .
\ee
It follows from (\ref{R7}), (\ref{BE7}), (\ref{BF7}) that $q$ is a solution to the PDE
\be \label{BG7}
\frac{\pa q(x,y,t,T)}{\pa t}-\frac{x}{p} \frac{\pa q(x,y,t,T)}{\pa x}  \ = \ 0 \ .
\ee
Since $q$  is a solution to (\ref{BG7}) we need only show that
\be \label{BH7}
e^{-(T-t)/p}g\left(\frac{x}{v}\right)v \ \le \ v\frac{\pa q(x,y,t,T)}{\pa x}  \quad {\rm for \ } 0<v<1 \ ,
\ee
in order to prove that $q$ is a solution to the HJ equation (\ref{BC7}). 
This is equivalent to showing that
\be \label{BI7}
u(x,t) \ = \ \frac{\pa q(x,y,t,T)}{\pa x} -e^{-(T-t)/p}g(x) \ \ge \ 0 \ .
\ee
Since $\pa \tau_{x,t}/\pa x$ approaches $-1$ as $x\ra x_p(t)$, it follows from (\ref{BE7}) that $u(x,t)=0$ if $x=x_p(t)$.  We have now that
\be \label{BJ7}
 \frac{\pa u(x,t)}{\pa t}-\frac{x}{p} \frac{\pa u(x,t)}{\pa x}-\frac{1}{p}u(x,t) \ = \frac{x}{p}e^{-(T-t)/p}g'(x) \ \le \ 0 \ .
\ee
It follows by the method of characteristics that $u(x,t)\ge 0$ for all $(x,t)$ satisfying (\ref{L7}). Hence the function $q$ of (\ref{BD7}) is a solution to the HJ equation (\ref{BC7}). Since $q$ is a $C^1$ solution to the HJ equation for $(x,t)$ in the reachable set we argue again as in Proposition 7.1 to show that
the solution to the variational problem (\ref{AX7}) is given by (\ref{BD7}). 
\end{proof}
\begin{rem}
Since the function (\ref{BD7}) satisfies $\pa q(x,y,t,T)/\pa x\ge 0$, it follows from Proposition 7.3  that the solution of the variational problem  $\max_{0<v(\cdot)<1} q(y,v(\cdot),t,T)$, with $q$ as in (\ref{AY7}), is given by $v(\cdot)\equiv 1$. 
\end{rem}
Finally we consider the variational problem analogous to (\ref{AX7}) given by 
\be \label{BK7}
q(x,y,t,T) \ =  \  \min_{1\le v(\cdot)<\infty}\left[ \int_t^T ds \  g\left(   \frac{x(s)}{v(s)} \right)e^{-(T-s)/p}v(s)   \ \Bigg| \ x(t)=x \ \right] \ ,
\ee
where $x(\cdot)$ satisfies (\ref{A7}) and $(x,t)$ belongs to the reachable set (\ref{AA7}) of the control system. The Hamilton-Jacobi (HJ) equation for (\ref{BK7})  is given by
\begin{multline} \label{BL7}
\frac{\pa q(x,y,t,T)}{\pa t}-\frac{x}{p} \frac{\pa q(x,y,t,T)}{\pa x} \\
+\inf_{1<v<\infty}\left[ -v\frac{\pa q(x,y,t,T)}{\pa x}   +e^{-(T-t)/p}g\left(\frac{x}{v}\right)v\right] \ = \ 0 \ .
\end{multline} 
The minimization problem (\ref{BK7}) is trivial in the case of $g(\cdot)$ constant since then the function $q(y,v(\cdot),t,T)$ of (\ref{AY7}) is independent of $v(\cdot)$. In fact if $g(\cdot)\equiv 1$ we have from (\ref{A7}) that
\be \label{BM7}
q(y,v(\cdot),t,T) \ = \ -\int_t^T ds \ e^{-(T-s)/p}\left[\frac{dx(s)}{ds}+\frac{x(s)}{p}\right] \ .
\ee
Evaluating the integral on the RHS of (\ref{BM7}), we conclude that $q(x,y,t,T)=e^{-(T-t)/p}x-y$ for $(x,t)$ in the reachable set (\ref{AA7}). Note that since $\pa q(x,y,t,T)/\pa x=e^{-(T-t)/p}$, the infimum in (\ref{BL7}) is now simply zero. 

In order to solve the HJ equation (\ref{BL7}) for more general $g(\cdot)$,  we consider the function
\be \label{BN7}
G(x,\xi,v,t) \ = \ -v\xi+e^{-(T-t)/p}g\left(\frac{x}{v}\right)v \ .
\ee 
We assume the function $v\ra G(x,\xi,v,t)$ is convex, which is the case provided $g(\cdot)$ is convex. The minimum of $G(x,\xi,v,t)$  on the interval $1<v<\infty$ is attained at $v=1$ if $e^{T-t)/p}\xi\le g(x)-xg'(x)$. If $e^{(T-t)/p}\xi>g(x)-xg'(x)$ then the minimizer of  $\min_{v\ge 1} G(x,\xi,v,t)$ is the solution to the equation 
\be \label{BO7}
g\left(\frac{x}{v}\right)- \frac{x}{v}g'\left(\frac{x}{v}\right) \ = \ e^{(T-t)/p}\xi  \ = \zeta \ .
\ee
A solution to (\ref{BO7}) exists for all $\zeta>g(x)-xg'(x)$ provided $\lim_{z\ra 0}[g(z)-zg'(z)]=\infty$. 
From (\ref{BO7}) it follows that the minimizing $v=v_{\rm min}(x,\xi)=xh(\zeta)$ for some function $h(\cdot)$. The corresponding HJ equation has therefore the form
\be \label{BP7}
\frac{\pa q(x,y,t,T)}{\pa t}+ xe^{-(T-t)/p}H\left(e^{(T-t)/p}\frac{\pa q(x,y,t,T)}{\pa x}\right) \ = \ 0 \ ,
\ee
where
\be \label{BQ7}
H(\zeta) \ = \ -\frac{\zeta}{p}-\zeta h(\zeta)+g\left(\frac{1}{h(\zeta)}\right)h(\zeta) \ .
\ee
From the Hamiltonian equations of motion we have that $\xi=\pa q(x,y,t,T)/\pa x$ evolves along characteristics according to the ODE
\be \label{BR7}
\frac{d\xi(t)}{dt} \ = \ -\frac{\pa}{\pa x}\left[  xe^{-(T-t)/p}H\left(e^{(T-t)/p}\xi\right)\right]  \ = \ -\left[ e^{-(T-t)/p}H\left(e^{(T-t)/p}\xi\right)\right]  \ .
\ee
Setting $\zeta(s)=e^{(T-s)/p}\xi(s)$ we have from (\ref{BQ7}), (\ref{BR7}) that $\zeta(\cdot)$ is a solution to the autonomous ODE
\be \label{BS7}
\frac{d\zeta(s)}{ds}  \ = \ \zeta h(\zeta)-g\left(\frac{1}{ h\left( \zeta\right) }\right)h\left( \zeta\right)\ .
\ee
We can in principle construct a solution to the HJ equation (\ref{BL7}) in the reachable set (\ref{AA7})  by solving (\ref{BS7}).  Thus at a point $[x_p(t),t]$ on the boundary of the reachable set we set $\zeta$ at $[x_p(t),t]$ to be the solution to $1=x_p(t)h(\zeta)$. Then we solve (\ref{BS7}) for times $s<t$ with this value of $\zeta$ as the terminal condition. This gives us the values of the optimal controller along the characteristic, and so we can construct the characteristic by solving (\ref{A7}) for times $s<t$ with terminal condition $x_p(t)$.  

We assume now that the function $g(\cdot)$ is non-negative decreasing and convex. This implies for non-degenerate $g(\cdot)$ that (\ref{BO7}) can be solved uniquely to determine $h(\zeta)$.  In order to implement our strategy for constructing a solution to the HJ equation (\ref{BL7}) we need to make some extra assumptions on $g(\cdot)$. To see what these are  let the function $s\ra\zeta(s,t), \ s<t$ be a solution to (\ref{BS7}) with terminal condition  $x_p(t)h(\zeta(t,t))=1$. The corresponding characteristic equation (\ref{A7}) is then given by 
\be \label{BT7}
\frac{dx(s)}{ds} \ = \ -\frac{x(s)}{p}-x(s)h(\zeta(s,t)) \ .
\ee
We need to have that $x(s)h(\zeta(s,t))>1$ for $s<t$ in order that the optimal controller $v>1$ along the characteristic. We have that
\begin{multline} \label{BU7}
\frac{d}{ds}\left[  x(s)h(\zeta(s,t))  \right] \ = \\
 - x(s)h(\zeta(s,t))\left\{   \frac{1}{p}+h(\zeta(s,t))-h'(\zeta(s,t))\left[  \zeta(s,t)-g\left(\frac{1}{ h\left( \zeta(s,t)\right) }\right)\right]\right\}  \ .
\end{multline}
 We obtain an expression for $h'(\zeta)$ by observing from (\ref{BO7}) that $h(\cdot)$ is a solution to the equation 
\be \label{BV7}
g\left(\frac{1}{ h( \zeta)}\right)-\frac{1}{h(\zeta)}g'\left(\frac{1}{ h( \zeta)}\right) \ = \ \zeta \ .
\ee
On differentiating (\ref{BV7}) we obtain the relation
\be \label{BW7}
\frac{h'(\zeta)}{h(\zeta)^3}g''\left(\frac{1}{ h( \zeta)}\right) \ = \ 1 \ .
\ee
It follows from (\ref{BV7}), (\ref{BW7}) that the RHS of (\ref{BU7}) is negative provided $zg''(z)+g'(z)>0, \ z>0$. Hence we  impose the extra assumption on $g(\cdot)$ that the non-negative function $z\ra -zg'(z)$ is decreasing.

In order to evaluate the function $q(x,y,t,T)$ of (\ref{BK7}) we just need to know the values $z(s)=x(s)/v(s)=1/h(\zeta(s)), \ s<t,$  along the characteristics in the reachable set (\ref{AA7}) which terminate on the curve $x_p(\cdot)$.  It follows from (\ref{BS7}), (\ref{BV7}), (\ref{BW7}) that $z(s)$ is a solution to the ODE
\be \label{BX7}
\frac{dz(s)}{ds} \ = \ \frac{g'(z(s))}{z(s)g''(z(s))} \ .
\ee
Observing that
\be \label{BY7}
-\int^z\frac{z'g''(z')}{g'(z')} \ dz' \ = \  F(z)+{\rm constant} \ ,
\ee
where the function $F(\cdot)$ is given by
\be \label{BZ7}
F(z) \ = \ -z\log[-g'(z)]+\int_{z_0}^z\log[-g'(z')] \ dz' \quad {\rm for \ any \ } z_0>0 \ ,
\ee
we see that the solution to (\ref{BX7}) satisfies $F(z(s))={\rm constant}-s$.  This enables us to obtain $z(s)$ just under the assumption that $g(\cdot)$ is $C^1$. 
\begin{lem}
Assume the function $g:\R^+\ra\R^+$ is  $C^1$, nonnegative decreasing, and that the function $z\ra-zg'(z)$ is also decreasing.  Then $g(\cdot)$ is convex. 
Define $z_\infty=\sup\{z>0: g'(z)<0\}$. Then $F(\cdot)$ is strictly increasing in the  interval $0<z<z_\infty$,  and $F(z_2)-F(z_1)\ge z_2-z_1$ for $0<z_1<z_2<z_\infty$.  In addition one has $\lim_{z\ra z_\infty}F(z)=\infty$. 
\end{lem}
\begin{proof}
To see that $g(\cdot)$ is convex we show the function $z\ra g'(z)$ is increasing. Thus  we have
\be \label{CA7}
z_1[g'(z_2)-g'(z_1)]  \ge   z_2g'(z_2)-z_1g'(z_1) \ \ge \  0 \ , \quad {\rm for \ } 0<z_1<z_2 \ .
\ee
We also have that
\begin{multline} \label{CB7}
F(z_2)-F(z_1) \ = \ z_1\log[-z_1g'(z_1)]-z_2\log[-z_2g'(z_2)]+\int_{z_1}^{z_2}\log[-z'g'(z')] \ dz' \\
-z_1\log z_1+z_2\log z_2-\int_{z_1}^{z_2}\log z' \ dz' \ = \ \\
z_1\log[-z_1g'(z_1)]-z_2\log[-z_2g'(z_2)]+\int_{z_1}^{z_2}\log[-z'g'(z')] \ dz'+(z_2-z_1) \\
\ge \ z_1\log[-z_1g'(z_1)]-z_2\log[-z_2g'(z_2)]+(z_2-z_1)\log[-z_2g'(z_2)] +(z_2-z_1) \\
= \ z_1\{\log[-z_1g'(z_1)]-\log[-z_2g'(z_2)]\}+(z_2-z_1) \ \ge \ z_2-z_1 \ .
\end{multline}
Evidently (\ref{CB7}) implies that $\lim_{z\ra z_\infty} F(z)=\infty$ if $z_\infty=\infty$.
To show that $\lim_{z\ra z_\infty} F(z)=\infty$ when $z_\infty<\infty$, it will be sufficient to prove that for an increasing function $h:(0,z_\infty)\ra\R$ such that $\lim_{z\ra z_\infty} h(z)=\infty$, then one also has that
\be \label{CC7}
\lim_{z\ra z_\infty} \left[\al h(z)-\int_{z_0}^z h(z') \ dz' \ \right] \ = \ \infty \quad {\rm for \ any \ } \al>0 \ .
\ee
Once (\ref{CC7}) has been established, we just set $h(z)=-\log[-zg'(z)]$ to conclude that $\lim_{z\ra z_\infty}F(z)=\infty$. To see (\ref{CC7}) we note that the LHS of (\ref{CC7}) is bounded below by $C+\al h(z)/2$ for some constant $C$, whence (\ref{CC7}) follows. 
\end{proof}
We are now in a position to construct the solution to the HJ equation (\ref{BL7}) under the assumption that $g(\cdot)$ satisfies the conditions of Lemma 7.1.  For $0<\la<\la_\infty=\min[z_\infty,y]$ and $s<T$ we define $z_p(s,\la)$ as the unique solution to the equation $F(z_p(s,\la))=F(\la)+T-s$. Corresponding to the function $z_p$ are a set of characteristics $y_p(s,\la)$ defined as solutions to
\be \label{CD7}
\frac{dy_p(s,\la)}{ds} \ = \ -\left[\frac{1}{p}+\frac{1}{z_p(s,\la)}\right]y_p(s,\la) \ ,  \ s<T, \quad y_p(T,\la) \ = \ y \   .
\ee
We define the function $q(x,y,t,T)$ in the region $\{(x,t): \ t<T, \ x>y_p(t,\la_\infty)\}$ by 
\be \label{CE7}
q(x,y,t,T) \ = \ \int_t^T \frac{g(z_p(s,\la_{x,t}))}{z_p(s,\la_{x,t})}y_p(s,\la_{x,t})e^{-(T-s)/p} \ ds \ , \quad {\rm where  \ } y_p(t,\la_{x,t})=x \ .
\ee
Next let $T_\infty=T$ if $z_\infty<y$, and otherwise if $z_\infty<\infty$  the unique solution to the equation $x_p(T_\infty)=z_\infty$.  For $z_\infty<\infty$ and $(x,t)$ in the set
\be \label{CF7}
t<T_\infty\le T, \quad x_p(t)<x<\exp\left[\left\{\frac{1}{p}+\frac{1}{z_\infty}\right\}(T_\infty-t)\right] \max[z_\infty,y]\ ,
\ee
we define $q(x,y,t,T)$ by 
\begin{multline} \label{CG7}
q(x,y,t,T) \ = \  \int_{T_\infty}^T g(x_p(s))e^{-(T-s)/p} \ ds \\
+ g(z_\infty)\left[e^{-(T-t)/p}x-y-p\left\{1-e^{-(T-T_\infty)/p}\right\}\right] \ .
\end{multline}
If $z_\infty\le y$ then (\ref{CE7}) and (\ref{CG7}) define $q(x,y,t,T)$ for all $(x,t)$ in the reachable set (\ref{AA7}).  If $z_\infty>y$ then $T_\infty<T$ and the subset of the reachable set defined by $\{(x,t) \ : \ T_\infty<t< T, \quad x_p(t)<x<y_p(t,y)\}$ lies between the two previously defined regions of the reachable set. For $T_\infty<t< T, \ s<t,$ we define the function $\tilde{z}_p(s,t)$ as the unique solution to the equation $F(\tilde{z}_p(s,t))=F(x_p(t))+t-s$. Corresponding to the function $\tilde{z}_p$ are a set of characteristics $x_p(s,t)$ defined as solutions to
\be \label{CH7}
\frac{dx_p(s,t)}{ds} \ = \ -\left[\frac{1}{p}+\frac{1}{\tilde{z}_p(s,t)}\right]x_p(s,t) \ ,  \ s<t, \quad x_p(t,t) \ = \ x_p(t) \   .
\ee
For any $(x,t)$  satisfying $T_\infty<t< T, \quad x_p(t)<x<y_p(t,y)$,  let $\tau_{x,t}$ be the unique solution to the equation $x_p(t,\tau_{x,t})=x$.  We define the function $q(x,y,t,T)$ by
\begin{multline} \label{CI7}
q(x,y,t,T) \ = \ \int_t^{\tau_{x,t}} \frac{g(\tilde{z}_p(s,\tau_{x,t}))}{\tilde{z}_p(s,\tau_{x,t})}x_p(s,\tau_{x,t})e^{-(T-s)/p} \ ds   \\
+\int^T_{\tau_{x,t}} g(x_p(s))e^{-(T-s)/p} \ ds \ .
\end{multline}

In order to show that $q$ is well defined by the formulas (\ref{CE7}), (\ref{CG7}), (\ref{CI7}) we need to prove that the trajectories $y_p(\cdot,\cdot), \ x_p(\cdot,\cdot)$  lie in the reachable set (\ref{AA7}) and do not intersect.  
\begin{lem}
Assume $g(\cdot)$ satisfies the conditions of Lemma 7.1. Then the trajectories $y_p(s,\la), \ s<T, \ \la<\la_\infty,$ defined by (\ref{CD7}) lie in the reachable set (\ref{AA7}) and do not intersect.  If $T_\infty<T$ then the trajectories $x_p(s,t), \ s<t, \ T_\infty<t<T,$ defined by (\ref{CH7}) also lie in the reachable set (\ref{AA7}) and do not intersect. 
\end{lem}
\begin{proof}
In view of the monotonicity of the function $F(\cdot)$ defined by (\ref{BZ7}), it is clear from (\ref{CD7}) that the trajectories $y_p(\cdot,\la), \ 0<\la<\la_\infty,$ do not intersect. To prove that $y_p(\cdot,\la)$ lies in the reachable set we set $x(s)=y_p(s,\la)-x_p(s)$ and observe that $x(\cdot)$ satisfies the equation
\be \label{CU7}
\frac{dx(s)}{ds} \ = \ -\left[\frac{1}{p}+\frac{1}{z_p(s,\la)}\right]x(s)-\frac{x_p(s)}{z_p(s,\la)}+1 \ , \quad s<T, \ x(T)=0 \ .
\ee
Hence the function $x(\cdot)$ is non-negative provided $z_p(s,\la)<x_p(s), \ s<T$.  This follows from Lemma 7.1 since
\begin{multline} \label{CV7}
z_p(s_2,\la)-z_p(s_1,\la) \ \le F(z_p(s_2,\la))-F(z_p(s_1,\la)) \\
 = \ s_1-s_2 \ < x_p(s_2)-x_p(s_1) \ , \quad s_2<s_1<T \ .
\end{multline}

We can similarly see that the solution $x_p(\cdot,t)$ to (\ref{CH7}) lies in the reachable set since $\tilde{z}_p(s,t)<x_p(s),  \ s<t$.  To see that the trajectories  $x_p(\cdot,t), \ T_\infty<t<T,$ do not intersect  we consider for $T_\infty<t_1<t_2<T$ the function $x(s)=x_p(s,t_2)-x_p(s,t_1), \ s<t_1$. From (\ref{CH7}) it follows that $x(\cdot)$ is a solution to the equation
\be \label{CW7}
\frac{dx(s)}{ds} \ = \ -\left[\frac{1}{p}+\frac{1}{\tilde{z}_p(s,t_1)}\right]x(s)-\left[\frac{1}{\tilde{z}_p(s,t_2)}-\frac{1}{\tilde{z}_p(s,t_1)}\right]x_p(s,t_2)  \ .
\ee
Since the trajectory $x_p(\cdot,t_2)$ is in the reachable set we have that $x(t_1)>0$.  We also have that $\tilde{z}_p(t_1,t_2)<x_p(t_1)=\tilde{z}_p(t_1,t_1)$, whence by the monotonicity of $F(\cdot)$ it follows that $\tilde{z}_p(s,t_2)< \tilde{z}_p(s,t_1), \ s<t_1$. We conclude then from (\ref{CW7}) that $x(s)>0, \ s<t_1$. 
\end{proof}
\begin{proposition}
Assume $g(\cdot)$ satisfies the conditions of Lemma 7.1, and in addition that $g(\cdot)$ is $C^2$ on the interval $(0,z_\infty)$.  For any $y>0, \ t<T$ let $(x,t)$ satisfy  (\ref{AA7}).  Then the function $q(x,y,t,T)$ of (\ref{BK7}) is given by one of the formulae (\ref{CE7}), (\ref{CG7}), (\ref{CI7}).  
\end{proposition}
\begin{proof}
We show that the function $q(x,y,t,T)$ defined by   (\ref{CE7}), (\ref{CG7}), (\ref{CI7}) is a $C^1$ solution to the HJ equation (\ref{BL7}). First we consider $(x,t)$ in the region (\ref{CF7})  where $q(x,y,t,T)$ is defined by (\ref{CG7}).  Evidently $\pa q(x,y,t,T)/\pa x=g(z_\infty)e^{-(T-t)/p}$ and $q(x,y,t,T)$ is a solution to the PDE (\ref{BG7}). Since $\inf_{x>0}g(x)=g(z_\infty)$ and $x>z_\infty$, it follows that the infimum in (\ref{BL7}) is zero. Hence $q(x,y,t,T)$ is a solution to the PDE (\ref{BL7}) for $(x,t)$ in the region (\ref{CF7}). 

Next we consider the formula (\ref{CE7}) for $q(x,y,t,T)$.  In order to obtain formulas for the derivatives of  $q$ we first need formulas for the derivatives of $z_p(s,\la)$ and $y_p(s,\la)$ with respect to $\la$.  Since the function $z\ra-zg'(z)$ is decreasing and $g(\cdot)$ is $C^2$ we have that $zg''(z)+g'(z)\ge 0$. This implies that $g''(z)>0$ for $0<z<z_\infty$.  From (\ref{BZ7}) and the definition of $z_p(s,\la)$ we have that
\be \label{CJ7}
\frac{z_p(s,\la)g''(z_p(s,\la))}{g'(z_p(s,\la))}\frac{\pa z_p(s,\la)}{\pa \la} \ = \ \frac{\la g''(\la)}{g'(\la)} \ .
\ee
Hence $z_p(\cdot,\cdot)$ is $C^1$ and the function $\la\ra \pa z_p(s,\la)/\pa \la$ is positive for $s<T, \ 0<\la<\la_\infty$.  From (\ref{CD7}), (\ref{CE7}) we have that
\be \label{CK7}
\exp\left[\int_t^T\left\{\frac{1}{p}+\frac{1}{z_p(s,\la_{x,t})}\right\} \ ds\right]y \ = \ x \ .
\ee
In view of the positivity of the function $(s,\la)\ra \pa z_p(s,\la)/\pa \la$, it follows from (\ref{CK7})  that the function $(x,t)\ra\la_{x,t}$ is $C^1$ and
\be \label{CL7}
\frac{\pa \la_{x,t}}{\pa t} \ = \ \left[\frac{x}{p}+\frac{x}{z_p(t,\la_{x,t})}\right]\frac{\pa \la_{x,t}}{\pa x} \ .
\ee
Hence the function $(x,t)\ra q(x,y,t,T)$ of (\ref{CE7}) is $C^1$. Furthermore, on differentiating (\ref{CE7}) with respect to $x$ and $t$ and using (\ref{CL7}) we conclude that $q$ is a solution to the PDE
\begin{multline} \label{CM7}
\frac{\pa q(x,y,t,T)}{\pa t}-\frac{x}{p} \frac{\pa q(x,y,t,T)}{\pa x} -v(x,t)\frac{\pa q(x,y,t,T)}{\pa x} \\
 +e^{-(T-t)/p}g\left(\frac{x}{v(x,t)}\right)v(x,t) \ = \ 0 \ , \quad
{\rm where \ \ } \frac{x}{v(x,t)} \ = \ z_p(t,\la_{x,t}) \ .
\end{multline}

From (\ref{CM7}) we see that in order to prove $q(x,y,t,T)$ is a solution to the HJ equation (\ref{BL7}) it is sufficient to show that $v(x,t)$ of (\ref{CM7})  satisfies
\be \label{CN7}
v(x,t)>1, \quad g(z_p(t,\la_{x,t}))-z_p(t,\la_{x,t})g'(z_p(t,\la_{x,t})) \ = \ e^{(T-t)/p}\frac{\pa q(x,y,t,T)}{\pa x} \ .
\ee
Note that the identity in (\ref{CN7}) is the same as (\ref{BO7}). To show that $v(x,t)>1$ we proceed as in (\ref{BU7}).  Thus we have from the definition of $z_p(s,\la)$ and  (\ref{CD7})  that
\be \label{CO7}
\frac{d}{ds}\left[\frac{y_p(s,\la)}{z_p(s,\la)}\right] \ = \ -\frac{y_p(s,\la)}{z_p(s,\la)}\left[\frac{1}{p}+\frac{1}{z_p(s,\la)}+\frac{g'(z_p(s,\la))}{z_p(s,\la)^2g''(z_p(s,\la))}\right] \ .
\ee
Since $zg''(z)+g'(z)\ge 0$ the RHS of (\ref{CO7}) is negative and $y_p(T,\la)/z_p(T,\la)=y/\la\ge y/\la_\infty\ge 1$. Setting $\la=\la_{x,t}$ we conclude that $v(x,t)=y_p(t,\la_{x.t})/z_p(t,\la_{x,t})>1$. To establish the identity of (\ref{CN7}) we define a function $V(\cdot,\cdot)$ by
\be \label{CP7}
V(t,\la) \ = \ \frac{\pa}{\pa \la} q(y_p(t,\la),y,t,T)-\xi(t,\la)\frac{\pa y_p(t,\la)}{\pa \la} \ , \quad t<T, \ 0<\la<\la_\infty \ ,
\ee
where $q$ is given by (\ref{CE7})  and
\be \label{CQ7}
\xi(t,\la) \ = \ e^{-(T-t)/p}\left[g(z_p(t,\la))-z_p(t,\la)g'(z_p(t,\la))\right] \ .
\ee
Evidently $V(T,\cdot)\equiv 0$.  We also have that
\begin{multline} \label{CR7}
\frac{\pa V(t,\la)}{\pa t} \ = \ -\frac{\pa}{\pa \la}\left[\frac{g(z_p(t,\la))}{z_p(t,\la)}y_p(t,\la)e^{-(T-t)/p}\right] \\
-\frac{\pa \xi(t,\la)}{\pa t}\frac{\pa y_p(t,\la)}{\pa \la} +\xi(t,\la)\frac{\pa}{\pa \la} \left[  \left\{\frac{1}{p}+\frac{1}{z_p(t,\la)}\right\}y_p(t,\la)        \right] \ ,
\end{multline}
where we have used (\ref{CD7}).  The RHS of (\ref{CR7}) can be written as $A(t,\la)y_p(t,\la)+B(t,\la)\pa y_p(t,\la)/\pa \la$. It is evident that $A(\cdot,\cdot)\equiv 0$. Using the formula (\ref{BX7}) for $\pa z_p(t,\la)/\pa t$ and (\ref{CQ7}) we can also see that $B(\cdot,\cdot)\equiv 0$.  Hence $V(\cdot,\cdot)\equiv 0$ and so the identity of (\ref{CN7}) holds.  

Finally we consider the formula (\ref{CI7}) for $q(x,y,t,T)$. Similarly to (\ref{CK7}) we have from  (\ref{CH7})  that
\be \label{CS7}
\exp\left[\int_t^{\tau_{x,t}}\left\{\frac{1}{p}+\frac{1}{\tilde{z}_p(s,\tau_{x,t})}\right\} \ ds\right]x_p(\tau_{x,t}) \ = \ x \ .
\ee
On differentiating (\ref{CS7}) with respect to $x,t$ we see that the function $(x,t)\ra\tau_{x,t}$ is $C^1$  provided the function $(s,\tau)\ra \pa \tilde{z}_p(s,\tau)/\pa \tau$ is negative. To see this we differentiate the equation $F(\tilde{z}_p(s,\tau))=F(x_p(\tau))+\tau-s$ with respect to $\tau$ to obtain the identity
\be \label{CT7}
F'(\tilde{z}_p(s,\tau))\frac{\pa \tilde{z}_p(s,\tau))}{\pa \tau} \ = \ F'(x_p(\tau)) \frac{dx_p(\tau))}{d\tau}+1 \ .
\ee
From Lemma 7.1 we have that $F'(\cdot)\ge 1$ and from (\ref{A7}) that $x'_p(\tau)<-1$, whence the RHS of (\ref{CT7}) is strictly negative. It follows easily now on differentiation of (\ref{CS7})  that 
\be \label{CX7}
\frac{\pa \tau_{x,t}}{\pa t} \ = \ \left[\frac{x}{p}+\frac{x}{\tilde{z}_p(t,\tau_{x,t})}\right]\frac{\pa \tau_{x,t}}{\pa x} \ .
\ee
Since the function $(x,t)\ra\tau_{x,t}$ is $C^1$,  it follows that the function $(x,t)\ra q(x,y,t,T)$ of (\ref{CI7}) is $C^1$. Furthermore, on differentiating (\ref{CI7}) with respect to $x$ and $t$ and using (\ref{CX7}) we conclude that $q$ is a solution to the PDE
\begin{multline} \label{CY7}
\frac{\pa q(x,y,t,T)}{\pa t}-\frac{x}{p} \frac{\pa q(x,y,t,T)}{\pa x} -v(x,t)\frac{\pa q(x,y,t,T)}{\pa x} \\
 +e^{-(T-t)/p}g\left(\frac{x}{v(x,t)}\right)v(x,t) \ = \ 0 \ , \quad
{\rm where \ \ } \frac{x}{v(x,t)} \ = \ \tilde{z}_p(t,\tau_{x,t}) \ .
\end{multline}

From (\ref{CY7}) we see that in order to prove $q(x,y,t,T)$ is a solution to the HJ equation (\ref{BL7}) it is sufficient to show that $v(x,t)$ of (\ref{CY7})  satisfies
\be \label{CZ7}
v(x,t)>1, \quad g(\tilde{z}_p(t,\tau_{x,t}))-\tilde{z}_p(t,\tau_{x,t})g'(\tilde{z}_p(t,\tau_{x,t})) \ = \ e^{(T-t)/p}\frac{\pa q(x,y,t,T)}{\pa x} \ .
\ee
We can establish (\ref{CZ7}) by arguing as before, replacing $y_p(s,\la)$ by $x_p(s,\tau)$ and $z_p(s,\la)$ by $\tilde{z}_p(s,\tau)$ in (\ref{CO7})-(\ref{CR7}). Hence the function $q(x,y,t,T)$ is a $C^1$ solution to the HJ equation (\ref{BL7}). The fact that it is also $C^1$ across the boundaries of the various regions follows from 
(\ref{CN7}), (\ref{CZ7}). 
\end{proof}
\begin{rem}
Let $q(y,v(\cdot),t,T)$ be defined by (\ref{AY7}). Since the function defined by the formulae (\ref{CE7}), (\ref{CG7}), (\ref{CI7}) satisfies $\pa q(x,y,t,T)/\pa x\ge 0$, it follows from Proposition 7.4  that the solution of the variational problem  $\min_{1<v(\cdot)<\infty} q(y,v(\cdot),t,T)$ is given by $v(\cdot)\equiv 1$. 
\end{rem}
\begin{proof}[Proof of Theorem 7.1]
We apply propositions 7.1, 7.2 with $g(\cdot)=-(1+y/p)h'(\cdot)$ and propositions 7.3,7.4 with $g(\cdot)=p^{-1}h(\cdot)$.
\end{proof}

\vspace{.2in}

\section{Global Asymptotic Stability}
In this section we shall prove Theorem 1.2 by generalizing the results of $\S6$ for the linear DDE (\ref{N6})  to the non-linear DDE (\ref{B4}). Observe that Theorem 1.1 implies that $\sup_{t\ge 0}\|\xi(\cdot,t)\|_{2,\infty}\le M$ for some $M>0$.  Hence from the properties (c),(d) of the functional $I(\cdot)$, we conclude that $c_M\le I(\xi(\cdot,t))\le I(0(\cdot)), \ t\ge 0$.   This is a non-linear version of the result of Lemma 6.1. Next we establish a non-linear version of Proposition 6.2.
\begin{proposition}
Assume that $h(\cdot), \ \xi(\cdot,0)$ and $I(\cdot)$ satisfy the conditions of Theorem 1.1, and in addition that $h(\cdot)$ satisfies the conditions of Theorem 7.1. 
Then the initial value problem for (\ref{AH4}), (\ref{AI4}) with given $I(0)>0$ has a unique $C^1$ solution $I(t), \ t\ge 0$, globally in time.  Furthermore, $I(t)$ converges as $t\ra\infty$ to some $I_\infty>0$.
\end{proposition}
\begin{proof}
The global existence of a $C^1$ solution to the DDE (\ref{AH4}), (\ref{AI4})  follows from Theorem 1.1 and Lemma 4.1.  Since property (d) of the functional $I(\cdot)$ implies that $\rho(\xi(\cdot,t)\ge 0, \ t\ge 0,$ and $c_M\le I(\xi(\cdot,t))\le I(0(\cdot)),  \ t\ge 0$, we have that the function $G$ of (\ref{AA*4}) satisfies an inequality $|G(t,y,v_t(\cdot))|\le C_1e^{-t/p}, \ y\ge\ve_0,  \ t\ge 0, $ for some constant $C_1$, where $\ve_0>0$ is the constant occurring in property (a) of $I(\cdot)$.  Hence from Theorem 1.1 we conclude that  the function $g$ of (\ref{AI4}) satisfies an inequality $|g(t,v_t(\cdot))|\le C_2e^{-t/p},   \ t\ge 0, $ for some constant $C_2$. 

To prove convergence of $I(t)$ as $t\ra\infty$, we first assume that for any $\ve,\ga>0,\  \tau>\ga$ and $\ve<1/2$, there exists $T_{\ve,\ga,\tau}>\tau$ such that $|I(t)/I(s)-1|<\ve$ for $t,s\in[T_{\ve,\ga,\tau}-\ga, T_{\ve,\ga,\tau}]$.  For $t>T_{\ve,\ga,\tau}$ we set $I_{\rm max}(t)=\sup_{T_{\ve,\ga,\tau}<s<t}I(s)$ and consider $T>T_{\ve,\ga,\tau}$ such that $I(T)=I_{\rm max}(T)$.  Using an identity similar to (\ref{R6}), we have from (\ref{AH4}) that
\begin{multline} \label{A8}
\frac{1}{p}\log\left[\frac{I(T)}{I(T_{\ve,\ga,\tau})}\right] + \\
\int_{(T_{\ve,\ga,\tau},T)-\{T_{\ve,\ga,\tau}<t<T:I_{\rm max}(t)>I(t)\}} f(t,v_t(\cdot)) \ dt \ \le C_2pe^{-T_{\ve,\ga,\tau}/p} \ .
\end{multline}
We can estimate the second term on the LHS of (\ref{A8}) by using Theorem 7.1.  First we write the function $F$ of (\ref{Y4}) as $F=F_1+F_2$, where 
\be \label{B8}
F_1(t,y,v_t(\cdot)) \ = \ \int_{T_{\ve,\ga,\tau}-\ga}^t ds \left\{  \frac{1}{p}h\left(   \frac{z(s)}{v_t(s)}  \right) e^{-(t-s)/p}v_t(s)
- \left(1+\frac{y}{p}\right)   h'\left(    \frac{z(s)}{v_t(s)}  \right)\right\}  \ .
\ee
It is easy to see that $|F_2(t,y,v_t(\cdot))|\le C_3e^{-\ga/p}e^{-(t-T_{\ve,\ga,\tau})/p}, \ y\ge\ve_0,  \ t\ge T_{\ve,\ga,\tau}, $ for some constant $C_3$. Next we define the function $\tilde{v}_t(\cdot)$ as
\begin {eqnarray} \label{C8}
\tilde{v}_t(s) \ &=& \ v_t(T_{\ve,\ga,\tau}), \quad T_{\ve,\ga,\tau}-\ga<s<T_{\ve,\ga,\tau} \ , \\
\tilde{v}_t(s) \ &=& \ v_t(s), \quad T_{\ve,\ga,\tau}<s<t \ .  \nonumber 
\end{eqnarray}
We define also the function $\tilde{z}(s), \  \quad T_{\ve,\ga,\tau}-\ga<s<t,$ by the formula (\ref{X4}) with $\tilde{v}_t(\cdot)$ in place of $v_t(\cdot)$.   Defining now the function $\tilde{F}_1$ by (\ref{B8}) with $\tilde{v}_t(\cdot), \ \tilde{z}(\cdot)$ in place of $v_t(\cdot), \ z(\cdot)$, we have that
 \begin{multline} \label{D8}
F_1(t,y,v_t(\cdot))- \tilde{F}_1(t,y,\tilde{v}_t(\cdot)) \ = \\
 \frac{1}{p}\int_{T_{\ve,\ga,\tau}-\ga}^{T_{\ve,\ga,\tau}} ds  \   \left[ h\left(   \frac{z(s)}{v_t(s)}  \right)v_t(s)-  h\left(   \frac{\tilde{z}(s)}{\tilde{v}_t(s)}  \right)\tilde{v}_t(s)\right]e^{-(t-s)/p} \\
- \left(1+\frac{y}{p}\right) \int_{T_{\ve,\ga,\tau}-\ga}^{T_{\ve,\ga,\tau}} ds  \   \left[ h'\left(   \frac{z(s)}{v_t(s)}  \right)-  h'\left(   \frac{\tilde{z}(s)}{\tilde{v}_t(s)}  \right)\right]  \ .
\end{multline}
It follows from (\ref{D8}) and the assumptions on the function $h(\cdot)$ that there is a constant $C_4$ such that
 \begin{multline} \label{E8}
|F_1(t,y,v_t(\cdot))- \tilde{F}_1(t,y,\tilde{v}_t(\cdot))| \ \le \\
C_4\ga e^{-(t-T_{\ve,\ga,\tau})/p}\sup_{T_{\ve,\ga,\tau}-\ga<s<T_{\ve,\ga,\tau}}\left[ \frac{|v_t(s)-\tilde{v}_t(s)|}{v_t(s)}+ 
\frac{v_t(s)}{z(s)}\left| \frac{z(s)}{v_t(s)} -  \frac{\tilde{z}(s)}{\tilde{v}_t(s)} \right|\right] \ ,
\end{multline}
provided $y\ge \ve_0$.
From our assumptions on $I(\cdot)$ in the interval  $[T_{\ve,\ga,\tau}-\ga,T_{\ve,\ga,\tau}]$, we see that the first term in the supremum on the RHS of (\ref{E8}) is bounded  above by $\ve$.  We have also from (\ref{X4}) that $|1-\tilde{z}(s)/z(s)|<\ve$ for $s\in [T_{\ve,\ga,\tau}-\ga,T_{\ve,\ga,\tau}]$, whence we conclude that the supremum on the RHS of (\ref{E8}) is bounded by $4\ve$.

To estimate the second term on the LHS of (\ref{A8}) we write $f(t,v_t(\cdot))=f_1(t,v_t(\cdot))+f_2(t,v_t(\cdot))$, corresponding to the decomposition $F=F_1+F_2$. 
From our bound on $F_2$ we see there is a constant $C_5$ such that
\be \label{F8}
\int_{T_{\ve,\ga,\tau}}^\infty |f_2(t,v_t(\cdot))| \ dt \ \le \  C_5e^{-\ga/p} \ .
\ee
Letting $\tilde{f}_1(t,\tilde{v}_t(\cdot))$ be the function (\ref{AI4}) corresponding to $\tilde{F}_1$ in place of $F$, we have from (\ref{E8}) that
\be \label{G8}
\int_{T_{\ve,\ga,\tau}}^\infty |f_1(t,v_t(\cdot))-\tilde{f}_1(t,\tilde{v}_t(\cdot))| \ dt \ \le \  C_6\ga\ve \ ,
\ee
for some constant $C_6$. Observe next that by Theorem 7.1 one has $\tilde{F}_1(t,y,\tilde{v}_t(\cdot))\le \tilde{F}_1(t,y,1(\cdot))$ for $t>T_{\ve,\ga,\tau}$ such that $I(t)=I_{\rm max}(t)$, whence $\tilde{f}_1(t,\tilde{v}_t(\cdot))\ge 0$. We conclude then from (\ref{A8}), (\ref{F8}), (\ref{G8}) that
\be \label{H8}
\frac{1}{p}\log\left[\frac{I(T)}{I(T_{\ve,\ga,\tau})}\right]  \ \le \  C_2pe^{-\tau/p}+C_5e^{-\ga/p}+C_6\ga\ve \ .
\ee
 Since the constants $C_2,C_5,C_6$ in (\ref{H8}) are independent of $\ve,\ga,\tau$, we conclude that for any $\del>0$ there exists $T_\del>0$ such that $\sup_{t>T_\del}[I(t)/I(T_\del)-1]<\del$. Since we can make an exactly analogous argument with the function $I_{\rm min}(t)=\inf_{T_{\ve,\ga,\tau}<s<t}I(s)$, we conclude that $\lim_{t\ra\infty}I(t)=I_\infty>0$ exists. 
 
 Alternatively there exists $\ve_0,\ga_0>0, \ \tau_0>\ga_0$ such that $\sup_{s,t\in[T-\ga_0,T]}|I(t)/I(s)-1|\ge \ve_0$ for all $T\ge \tau_0$. Letting $I^+_\infty=\limsup_{t\ra\infty}I(t)$, there exists  for any $\del>0, \ N=1,2,..,$ a  time $T_{\del,N}>\max[\tau_0,N]$  such that $I(T_{\del,N})\ge I_\infty^+-\del$ and $I(t)\le I_\infty^++\del$ for $T_{\del,N}-N\le t\le T_{\del,N}$.  Since the oscillation of $I(\cdot)$ in the interval $[T_{\del,N}-\ga_0,T_{\del,N}]$ exceeds $\ve_0$,  there exists $\tau_{\del,N}\in[T_{\del,N}-\ga_0,T_{\del,N}]$ such that $I(\tau_{\del,N})\le (I_\infty^++\del)/(1+\ve_0)$.  We proceed similarly to before by writing the function $F$ of (\ref{Y4}) as $F=F_1+F_2$, where $F_1$ is given by (\ref{B8}), but with the interval of integration now $[T_{\del,N}-N,t]$ in place of $[T_{\ve,\ga,\tau}-\ga,t]$. As previously, one has the bound $|F_2(t,y,v_t(\cdot))|\le C_3e^{(\ga_0-N)/p}, \ y\ge\ve_0,  \ t\in[T_{\del,N}-\ga_0,T_{\del,N}]$. Evidently $F_1(t,y,v_t(\cdot))$ depends only on the values of $I(s)$ for $s\in [T_{\del,N}-N,t]$. We define $\tilde{v}_t(s)=I(s)/(I^+_\infty+\del)$ for  $s\in [T_{\del,N}-N,t]$, and $\tilde{F}_1(t,y,\tilde{v}_t(\cdot))$ in the same way as $F_1(t,y,v_t(\cdot))$ but with $\tilde{v}_t(\cdot)$ replacing $v_t(\cdot)$. The difference $F_1(t,y,v_t(\cdot))-\tilde{F}_1(t,y,\tilde{v}_t(\cdot))$ has the representation (\ref{D8}), but with the interval of integration now $[T_{\del,N}-N,t]$ in place of $[T_{\ve,\ga,\tau}-\ga,T_{\ve,\ga,\tau}]$.  Instead of (\ref{E8}) we have the estimate
 \begin{multline} \label{J8}
|F_1(t,y,v_t(\cdot))- \tilde{F}_1(t,y,\tilde{v}_t(\cdot))| \ \le \\
C_7\int_{T_{\del,N}-N}^t ds \ e^{-(t-s)/p}\left[ \frac{|v_t(s)-\tilde{v}_t(s)|}{v_t(s)}+ 
\frac{v_t(s)}{z(s)}\left| \frac{z(s)}{v_t(s)} -  \frac{\tilde{z}(s)}{\tilde{v}_t(s)} \right|\right] \ ,
\end{multline}
for some constant $C_7$. Setting $J(t)=I^+_\infty+\del-I(t),$ it follows from (\ref{J8}) that there is a constant  $C_8$ such that $|F_1(t,y,v_t(\cdot))- \tilde{F}_1(t,y,\tilde{v}_t(\cdot))| \le C_8J(t)$ for $t\in [T_{\del,N}-\ga_0,T_{\del,N}]$. We estimate the second term on the LHS of (\ref{AH4}) by  writing $f(t,v_t(\cdot))=f_1(t,v_t(\cdot))+f_2(t,v_t(\cdot))$, corresponding to the decomposition $F=F_1+F_2$. 
From our bound on $F_2$ we see there is a constant $C_9$ such that $|f_2(t,v_t(\cdot))|\le C_9e^{-N/p}$ for $t\in [T_{\del,N}-\ga_0,T_{\del,N}]$.
Letting $\tilde{f}_1(t,\tilde{v}_t(\cdot))$ be the function (\ref{AI4}) corresponding to $\tilde{F}_1$ in place of $F$, we also have that
$ |f_1(t,v_t(\cdot))-\tilde{f}_1(t,\tilde{v}_t(\cdot))|\le C_{10}J(t)$  for some constant $C_{10} $ if $t\in [T_{\del,N}-\ga_0,T_{\del,N}]$. Furthermore,  Theorem 7.1 implies that $\tilde{F}_1(t,y,\tilde{v}_t(\cdot))\le \tilde{F}_1(t,y,1(\cdot))$ for $t\in [T_{\del,N}-\ga_0,T_{\del,N}]$, whence $\tilde{f}_1(t,\tilde{v}_t(\cdot))\ge 0$ if $t\in [T_{\del,N}-\ga_0,T_{\del,N}]$. It follows now  that
\be \label{K8}
\frac{dJ(t)}{dt}+C_{11}J(t) \ \ge \ -C_{12}e^{-N/p}-C_{13}e^{-t/p} \ , \quad t\in[T_{\del,N}-\ga_0,T_{\del,N}] \ ,
\ee
for some positive constants $C_{11},C_{12},C_{13}$. 
Note that in deriving (\ref{K8}) we use the fact that the function $t\ra J(t)$ is non-negative.  Integrating (\ref{K8}) over the interval $[\tau_{\del,N},T_{\del,N}]$, we obtain the inequality
\be \label{L8}
J(T_{\del,N} )\ \ge e^{-C_{11}\ga_0}J(\tau_{\del,N})-C_{12}\ga_0e^{-N/p}-C_{13}pe^{\ga_0/p} e^{-T_{\del,N}/p} \ .
\ee
Observe  that $J(T_{\del,N})\le 2\del$ and $J(\tau_{\del,N})\ge\ve_0(I^+_\infty+\del)/(1+\ve_0)$.  Since $T_{\del,N}\ge N$, the inequality (\ref{L8}) yields a contradiction if $\del>0$ is sufficiently small and $N$ sufficiently large. 
\end{proof}
\begin{proof}[Proof of Theorem 1.2]
The result follows from Theorem 3.1 once we show that for any $\ve>0$ there exists $T_\ve>0$ such that $\|\xi(\cdot,T_\ve)-\xi_p(T_\ve)\|_{1,\infty}\le\ve$. To see this first let $\xi_0(\cdot,\cdot)$ be defined as
\be \label{M8}
\xi_0(y,t) \ = \ \int_0^t ds \ h(y(s))\exp\left[-\int_s^t\rho(s') \ ds'\right] \ , \quad y,t>0 \ .
\ee
Note that $\xi_0(\cdot,\cdot)$ is not the same as the solution to (\ref{A1}) with $\xi(\cdot,0)\equiv 0$ since the function $\rho(\cdot)$ depends on the initial data. From (\ref{D2}), (\ref{A4}) and the bounds on the function $I(\cdot)$ we see there is a constant $C$ such that $\|\xi(\cdot,t)-\xi_0(\cdot,t)\|_{1,\infty}\le Ce^{-t/p}, \ t\ge 0$.  For any $\nu$ with $0<\nu<1$ let $\tau_\nu>0$ be such that $|[I(s)/I(t)]^{1/p}-1|\le\nu$ if $s,t\ge \tau_\nu$.  We see then from (\ref{C2}) that for  $t>\tau_\nu,$
\be \label{N8}
(1-\nu)y_p(s) \ \le y(s) \ \le (1+\nu)y_p(s) ,  \quad \tau_\nu<s<t \ ,
\ee
where $y_p(\cdot)$ is the solution to (\ref{B2}) corresponding to the equilibrium $\rho=1/p$.  For $t>\tau_\nu$ let $\xi_{0,\nu}(\cdot,\cdot)$ be defined as in (\ref{M8}) but with the interval of integration $[0,t]$ replaced by the interval $[\tau_\nu,t]$. We similarly define the functions $\tilde{\xi}_0(\cdot,\cdot)$ and $\tilde{\xi}_{0,\nu}(\cdot,\cdot)$ by replacing the function $y(\cdot)$ in (\ref{M8}) with $y_p(\cdot)$ and $\rho(\cdot)$ with $1/p$.   Then there is a constant $C$  such  that
\begin{eqnarray} \label{O8}
\|\xi_{0,\nu}(\cdot,t)-\xi_0(\cdot,t)\|_{1,\infty} \ &\le&   Ce^{-(t-\tau_\nu)/p} \ , \quad t>\tau_\nu \ , \\
\|\tilde{\xi}_{0,\nu}(\cdot,t)-\xi_p(t)\|_{1,\infty} \ &\le&   Ce^{-(t-\tau_\nu)/p} \ , \quad t>\tau_\nu \ . \nonumber
\end{eqnarray}
From (\ref{N8}) we also have that $\|\xi_{0,\nu}(\cdot,t)-\tilde{\xi}_{0,\nu}(\cdot,t)\|_{1,\infty} \le C\nu, \ t>\tau_\nu,$ for some constant $C$.  Hence by choosing $C\nu=\ve/2$ and $T_\ve>\tau_\nu$ sufficiently large we  conclude from (\ref{O8}) that $\|\xi(\cdot,T_\ve)-\xi_p(T_\ve)\|_{1,\infty}\le\ve$.
\end{proof}

\end{document}